\documentclass[11pt, reqno]{amsart}

\usepackage[english]{babel}
\usepackage{amsthm} 
\usepackage{amssymb,cite}
\usepackage{amsmath}
\usepackage{hyperref}
\usepackage{amsfonts}
\usepackage{stmaryrd}
\usepackage{fancyhdr}
\usepackage{graphicx,color}
\usepackage{epstopdf}
\usepackage[left=1in, right=1in, top=1in, bottom=1in]{geometry}
\usepackage{mathrsfs}  
\usepackage{subfig}

\newtheorem{theorem}{Theorem}[section]
\newtheorem{prop}[theorem]{Proposition}

\newtheorem{hypothesis}[theorem]{Hypothesis}
\newtheorem{lemma}[theorem]{Lemma}
\newtheorem{definition}[theorem]{Definition}

\newtheorem{note}[theorem]{Note}
\newtheorem{corollary}[theorem]{Corollary}

\newtheorem{example}[theorem]{Example}
\newtheorem{excursus}[theorem]{Excursus}

\newcommand{\set}{\mathcal{D}}

\newcommand{\barz}{\bar{\zeta}}
\newcommand{\barZ}{\bar{Z}}
\newcommand{\barQ}{\bar{Q}}
\newcommand{\mub}{\bar{\mu}}

\newcommand{\bfzz}{{\bf{Z}}}
\newcommand{\A}{\mathcal{A}}

\newcommand{\N}{\mathbb{N}}
\newcommand{\E}{\mathbb{E}}
\newcommand{\cL}{\mathcal{L}}

\newcommand{\cP}{\mathcal{P}}
\newcommand{\cQ}{Q}
\newcommand{\cV}{\mathcal{V}}

\newcommand{\hf}{\hfill$\Box$} 
\newcommand{\pa}{\partial}
\newcommand{\lv}{\left\vert}
\newcommand{\rv}{\right\vert}
\newcommand{\be}{\begin{equation}}
\newcommand{\ee}{\end{equation}}
\newcommand{\dd}{,{\dots},}

\newcommand{\deln}{\hat{\Delta}}
\newcommand{\delnn}{\hat{\Delta}_{0}}
\newcommand{\curlys}{\mathscr{S}}
\newcommand{\capitals}{S}

\newcommand{\R}{\mathbb{R}}

\usepackage{booktabs} % for much better looking tables
\usepackage{array} % for better arrays (eg matrices) in maths
\usepackage{paralist} % very flexible & customisable lists (eg. enumerate/itemize, etc.)
\usepackage{verbatim} % adds environment for commenting out blocks of text & for 
\usepackage{fancyvrb}
\usepackage{float}
\usepackage{caption}
\usepackage{bbm} %For blackboard bold 1

\usepackage{hyperref}
\usepackage{url}
\usepackage[toc,page]{appendix}
\usepackage{color}

\usepackage[shortlabels]{enumitem} % For referencing items

\newcommand{\jacobian}[2]{\mathcal{J}_{#2} #1}  %% I have added this. Jacobian of 1 with respect to 2. - PD
\newcommand{\distlevel}[1]{\Delta_{#1}}   %% The distribution spanned by the vector fields up to level #1. - PD

\newcommand{\cZ}{\mathcal{Z}}
\newcommand{\Ad}{\mathrm{Ad}}
\newcommand{\ad}{\mathrm{ad}}
\newcommand{\zQ}{\mathcal{Q}}
\newcommand{\lz}{z}
\newcommand{\mulim}{\overline{\mu}^{\Sinfty}}
\newcommand{\Wlim}{W^\infty}
\newcommand{\wo}{\voperp}
\newcommand{\Lt}{\mathcal{L}}
\newcommand{\So}{\capitals_{x_0}} %%Fixed initial submanifold for local conv section
\newcommand{\Sbar}{\overline{\capitals}_{x_0}}
\newcommand{\St}{\capitals_{e^{t\voperp}(x_0)}}
\newcommand{\Stbar}{\overline{\capitals}_{e^{t\voperp}(x_0)}}
\newcommand{\Sinfty}{\capitals_{\overline{x}}}
\newcommand{\Sinftybar}{\overline{\capitals}_{\overline{x}}}

\newcommand{\DVinfty}{\mathcal{D}_V^{2,\infty}(\R^N)}

\newcommand{\pn}{z}
\newcommand{\pnn}{\zeta}

\newcommand{\fq}{\mathfrak{q}}

\newcommand{\rar}{\rightarrow}
\newcommand{\bfz}{{\bf{z}}}
\newcommand{\vodel}{V_0^{(\deln)}}
 \newcommand{\voperp}{V_0^{(\perp)}}
\newcommand{\fRm}{\mathcal{R}_{m}} %%These I have added - PD
\newcommand{\fRmo}{\mathcal{R}_{m,0}}

\begin{document}

\title[Long-time behaviour of degenerate diffusions]{Long-time behaviour of degenerate diffusions: 
UFG-type SDEs and  time-inhomogeneous hypoelliptic processes}
\author{ T. Cass, D. Crisan, P. Dobson and M. Ottobre}

\address{\noindent \textsc{Thomas Cass, Department of Mathematics, Imperial College London, Huxley Building, 180 Queen's Gate,  London SW7 2AZ, UK}} 
\email{thomas.cass@imperial.ac.uk}
\address{\noindent \textsc{Dan Crisan, Department of Mathematics, Imperial College London, Huxley Building, 180 Queen's Gate,  London SW7 2AZ, UK}} 
\email{d.crisan@imperial.ac.uk}
\address{\noindent \textsc{Paul Dobson, Maxwell Institute for Mathematical Sciences,  Department of Mathematics, Heriot-Watt University, Edinburgh EH14 4AS, UK}} 
\email{pd14@hw.ac.uk}
\address{\noindent \textsc{Michela Ottobre, Maxwell Institute for Mathematical Sciences, Department of Mathematics, Heriot-Watt University, Edinburgh EH14 4AS, UK}} 
\email{m.ottobre@hw.ac.uk}

%\keywords{UFG condition, cubature methods, exponential decay} 
%\subjclass[2000]{bla}

\begin{abstract}
We study the long time behaviour of a large class of diffusion processes on  $\R^N$,   generated by second order differential operators of (possibly) degenerate type. The operators that we consider {\em need not}  satisfy the H\"ormander condition. Instead,  they satisfy the so-called  UFG condition, introduced by Herman,  Lobry and Sussman in the context of geometric control theory and later by Kusuoka and Stroock, this time with probabilistic motivations. In this paper  we study UFG diffusions and demonstrate the importance of such a class of processes in several respects: roughly speaking i) we show that UFG processes constitute a family of SDEs which exhibit multiple invariant measures and for which one is able to describe a systematic procedure to determine the basin of attraction of each invariant measure (equilibrium state).   ii) We use an explicit change of coordinates introduced in differential geometry by Frobenius to prove that every UFG diffusion can be, at least locally, represented as a system consisting of an SDE coupled with an  ODE, where the ODE evolves independently of the SDE part of the dynamics.  iii) As a result,  UFG diffusions are inherently ``less smooth" than hypoelliptic SDEs; more precisely, we prove that UFG processes do not admit a density with respect to Lebesgue measure on the entire space, but only on suitable time-evolving submanifolds, which we describe. iv) We show that our results and techniques, which we devised for UFG processes, can be applied to the study of the long-time behaviour of non-autonomous hypoelliptic SDEs and therefore produce several results on this latter class of processes as well.    v) Because processes that satisfy the (uniform) parabolic H\"ormander condition are UFG processes, our paper contains a wealth of results about the long time behaviour of (uniformly) hypoelliptic processes which are non-ergodic, in the sense that they exhibit multiple invariant measures.   
 \vspace{5pt}
\vspace{5pt}
\\
{\sc AMS Classification (MSC 2010).} 60H10, 35K10,  35B35, 35B65, 58J65, 49J55, 93E03, 37H10, 

\end{abstract}

\date{\today}

\maketitle

\section{Introduction}

\subsection{Context and scope of the paper. }\label{context}
 Consider  stochastic differential equations (SDEs) in $\R^N$ of the form 
\be\label{SDE}
X_t=X_0 + \int_0^t V_0(X_s) ds + \sqrt{2} \sum_{i=1}^d\int_0^t  V_i(X_s) \circ  dB^i (s),
\ee
where $V_0,\ldots,V_d$ are smooth  vector fields on $\R^N$, $\circ$ denotes Stratonovich integration and  $B^1(t) \dd B^d(t)$  are one dimensional independent standard Brownian motions.   
The Markov semigroup $\{\cP_t\}_{t\ge 0}$ associated with the SDE \eqref{SDE} is defined on the set $C_b(\R^N)$ of continuous and bounded functions as
\be\label{semigroup}
\cP_t: C_b(\R^N) \rightarrow C_b(\R^N), \quad 
(\cP_t f)(x):= \mathbb{E} \left[f(X_t \vert X_0=x) \right].
\ee

We recall that, given a vector field $V:\R^N \rar \R^N$, we can interpret $V$ both as a vector-valued function on $\R^N$ and as a first order differential operator on $\R^N$:
\be\label{notvectoper}
V=(V^1(x), V^2(x) \dd V^N(x)) \quad \mbox{ or }
\quad V= \sum_{j=1}^NV^j(x)\pa_j, 
 \quad x \in \R^N, \pa_j= \pa_{x^j} \,.
\ee
 With this notation, 
the Kolmogorov operator associated with the semigroup $\cP_t$ is the second order differential operator given  on smooth functions by 
\be\label{gen1}
\cL=V_0+\sum_{i=1}^{d}V_i^2. 
\ee
The Markov diffusion $X_t$ is called hypoelliptic (elliptic, respectively) when  the operator $\cL$ is hypoelliptic (elliptic, respectively) \cite{Bakry}. The study of diffusion processes of hypoelliptic type has by now produced a fully-fledged theory, involving several branches of mathematics: stochastic analysis, analysis of differential operators, (sub-)Riemannian geometry and control theory.  One of the key steps in the development of such a theory has been the seminal paper of H\"ormander \cite{H1} and a large body of work has been dedicated for over forty years to the study of diffusion processes under the H\"ormander Condition (HC) (in one if its many forms),  which is a sufficient condition for hypoellipticity. 
In particular, the ergodic theory for hypoelliptic SDEs is well developed, see \cite{ReyBellet, Hairerbath, OttobrePavliotis, EckmannHairer} and references therein -- throughout the paper we define a process to be  {\em ergodic} if it admits a unique invariant measure (stationary state).

To the best of our knowledge, this is the first paper that attempts to build a framework for the study of the long time asymptotics of solutions of SDEs which are non-necessarily hypoelliptic.  We will work in the setting in which the vector fields $V_0,\ldots,V_d$ satisfy a weaker condition, the so-called {\em UFG condition}. The acronym UFG stands for {\em Uniformly Finitely Generated}. We give a precise  statement of the UFG condition in  Definition \ref{defufg}. 
For the moment let us just point out that, while  the  Parabolic H\"ormander condition imposes the following
 \be
\bigcup_{j\geq 1} \mathrm{span} \{\mathfrak{L}_j(x)\} =\R^N \quad \mbox{for every } x \in \R^N,
\tag{{\bf PHC}}
\ee
where as customary the hierarchy of operators $\mathfrak{L}_j$ is defined as $\mathfrak{L}_1(x):= \{V_1(x) \dd V_d(x)\}$ and, for $j >1$,  $\mathfrak{L}_j(x)= \mathfrak{L}_{j-1}(x) \cup \{[V, V_k], V \in \mathfrak{L}_{j-1}, k \in\{0 \dd d\}\}$, 
% \footnote{It is customary to assume the convention,  introduced in %\cite{Williams}, that repeated commutators with the vector $V_0$ (i.e. %commutators of the form $[[V_0,V_1], V_0], $ etc) are intended to be %included in this set - despite the fact that, if we take this notation to the %letter, then they are not. For the sake of brevity we assume this %convention now but we make clarifications later. }  
under the UFG condition the  vector space appearing in  ({\bf PHC}) is not required to have constant rank; roughly speaking, it is only required to be {\em finitely generated}.  In particular, we emphasize that the UFG condition does not impose the vector space  in ({\bf{PHC}}) to be equal to ${\mathbb R}^N$  for any $x\in {\mathbb R}^N$.   Hence, in this sense, the UFG condition is   weaker than the parabolic H\"{o}rmander condition. 
The UFG condition has been long known by the (geometric) control theory community, although perhaps under other names (see Section \ref{sec:preliminaries} for a more detailed account on the matter), and it is indeed well-studied in the works of Hermann, Lobry and Sussman \cite{Hermann, Lobry, Sussman}.   It was then considered by Stroock and Kusuoka in the eighties \cite{{KusStr82},{KusStr85},{KusStr87}, Kus03}, though in a completely different context (which we briefly explain below). The purpose there was to study smoothing properties of the semigroup $\cP_t$ under the UFG condition. In this paper we combine the geometric viewpoint with the functional analytic and probabilistic one to introduce new results on the asymptotic behaviour of UFG diffusions. In broad terms, the two main achievements of this paper can be described as follows: 
 
\medskip
{\bf i)} We study the diffusion process \eqref{SDE} in absence of the H\"ormander condition. To this end, we establish explicit connections between the geometric theory of finitely generated Lie algebras and the related stochastic dynamics. Because every (uniformly) hypoelliptic process is a UFG process, our results cover a very large class of SDEs.   In particular we show that our approach can be fruitfully employed to study the asymptotic behaviour of non-autonomous hypoelliptic diffusions.  

\medskip
 {\bf ii)} We  argue that UFG processes constitute a class of SDEs which exhibit, in general,  multiple equilibria and for which one is able, given an initial datum, to determine the invariant measure to which the dynamics will converge. 
%This paper therefore makes a first step towards
%establishing foundations for a future theory of non-ergodic processes.
\medskip

 Let us further remark on the significance of the latter point: 
although a large body of work has been devoted to the study of ergodic processes, the development of a general framework to understand problems with multiple
equilibria is at a very early stage.  It is well known that ergodic processes will, under appropriate general conditions,
converge to their unique equilibrium irrespective of the initial configuration, i.e. they will tend to lose memory of
the initial datum. Clearly this cannot be the case, in general, for more complicated systems.  When the invariant measure is not unique it is typically extremely difficult to determine the basin of attraction of each equilibrium measure and we are indeed not aware of any criteria developed to this effect.   To be more precise, one can ask one of the two (complemetary) questions:  given an initial datum for the SDE,  which equilibrium measure will the process converge to?  Conversely, given an equilibrium measure $\mu$, one may wish to describe the basin of attraction of such a measure, i.e. the set of initial data $x \in \R^N$ such that the process $X_t^{(x)}$ \footnote{We use the notation $X_t^{(x)}$ to emphasize the fact that the initial datum of the process is $X_0=x$.} converges to $\mu$.  Beyond numerical simulations, no  theoretical framework currently exists to tackle this kind of problems.

  In  this paper we  introduce a systematic way to study long-time convergence for a large class of SDEs which will, in general, admit several stationary states. This methodology applies to UFG diffusions and hence, because processes that satisfy the (uniform) parabolic H\"ormander condition are UFG processes, our results produce further understanding on non-ergodic H\"ormander processes - we stress here, and we will emphasize it again in Section \ref{sec:geomofUFGprocesses}, that hypoelliptic processes need not be ergodic (see Section \ref{sec:geomofUFGprocesses} for examples of hypoelliptic processes which are not ergodic).   

The Markov diffusions studied here are {\em linear},  in the sense that their generators  \eqref{gen1} are  linear second order differential operators. As a point of comparison, another  class of systems exhibiting multiple equilibria is the class of so-called {\em collective dynamics}: in this case the system is constituted by a
large number of particles or agents  that interact with each other. The underlying kinetic-PDEs for this type of models are {\em non-linear  in the sense of McKean} and the existence of multiple stationary states here is due to such a nonlinearity. In our case, the nature of the phenomenon is completely different and in a way simpler,  as multiple invariant measures arise as a result of the non-trivial control-theory implied by the UFG condition.

In the remainder of this introduction we comment on the implications and significance of the UFG condition first from an analytic perspective and then from a geometric and probabilistic viewpoint. In Subsection \ref{intro:mainresults} we  explain the main results of the paper and the reasons for studying  UFG diffusions;  we then conclude the introduction with Subsection \ref{intro:organization}, where we illustrate the organization of the paper.

 As is well known, under the (parabolic) { H\"{o}rmander condition}, the transition probabilities of the semigroup $\cP_t$ have a smooth density; furthermore,  $\cP_tf$ is differentiable in every direction and $u(t,x):=(\cP_t f)(x)$ is a classical solution of the Cauchy problem
\begin{align*}
 \pa_t u(t,x) &=\cL u(t,x)\\
 u(0,x)& =f(x).
\end{align*}
In the present  paper we will  relax the hypoellipticity  assumption and study the long-time behaviour of the dynamics \eqref{SDE} in absence of the H\"ormander condition. 

In a series of papers \cite{{KusStr82},{KusStr85},{KusStr87}, Kus03, {CrisanLitterer}, {CrisanDelarue}, Crisan}, Kusuoka and Stroock first and Crisan and collaborators later,  have analyzed the smoothness properties of diffusion semigroups $\{\cP_t\}_{t\ge 0}$ associated with the stochastic dynamics \eqref{SDE} when the vector fields $\{V_i,i=0,1,...,d\}$ satisfy the UFG condition.  Such works  showed that, as opposed to what happens under the PHC, under the UFG condition the semigroup $\cP_t$ is no longer differentiable in every direction; in particular it is no longer differentiable in the direction $V_0$, but it is still differentiable in the direction $\mathcal{V}:=\pa_t-V_0$ when viewed as a function 
$(t,x)\mapsto u(t,x)$ over the  product space 
$(0,\infty)\times {\mathbb R}^N$.  This fact has been proved by means of Malliavin calculus and in this paper we give a geometric and analytic  explanation of such a  phenomenon.  Because of differentiability in the direction $\mathcal{V}$,  a rigourous PDE analysis  can still be built starting from the stochastic dynamics \eqref{SDE}. In this case one can indeed prove that for  every $f\in C_b$ (continuous and bounded), the function $u(t,x):=(\cP_t f)(x)$ is a classical solution of the Cauchy problem
\begin{equation}
\left\{ 
\begin{array}{rl}
\mathcal{V} u(t,x) &=\sum_{i=1}^{d}V_i^2 u(t,x)\\
 u(0,x)& =f(x). \footnote{The notion of classical solution for the PDE \eqref{linearPDE} and further background material can be found in \cite[Appendix A]{CrisanOttobre}.}
 \end{array}
\right.\label{linearPDE}
\end{equation}

 From a geometric and control-theoretical point of view, working with the UFG condition will  imply dealing with distributions of non-constant rank. \footnote{In this paper we use the word distribution only in geometric sense, see definition at the beginning of Section \ref{intro:mainresults}. } 
If the geometric understanding of the H\"ormander condition  is rooted in the classic Frobenius  Theorem, which deals with distributions of constant rank,  the geometry of the UFG condition is described in the works of Hermann,  Lobry and Sussman \cite{Hermann, Lobry, Sussman}.  In these works, the UFG condition was considered for geometric and control theoretical purposes, in particular for the study of reachability (i.e., roughly speaking,  to answer questions regarding the set of points that can be reached by the integral curves of given vector fields). In this respect we should stress that the UFG condition is not optimal from a control-theoretical point of view (an optimal condition for reachability has been described by Sussmann \cite{Sussman}). However, it is the closest to being optimal, while still being easy to check in practice.  

Finally, by a probabilistic standpoint, it is well known that the Parabolic H\"ormander condition (PHC) is a sufficient (and almost necessary)  condition for the law of the process \eqref{SDE} to have a density, see \cite{Hairer},   and this fact has motivated the large literature on hypoelliptic SDEs. Again, the  understanding of this matter relies on   Frobenius  Theorem, as H\"ormander himself noted \cite{H1}. In his seminal paper \cite{Bismut}, Bismut  proved that, when the H\"ormander condition (HC) is enforced in place of the PHC, \footnote{The difference between the PHC and the HC will be clarified in Section \ref{sec:preliminaries}. } the law of the process  no longer admits a density on $\R^N$; however, it  admits a density on  appropriate time-dependent submanifolds of $\R^N$. In this paper we prove that a similar statement holds,  in more generality, for UFG processes, and in Section \ref{sec:8} we explicitly describe the time-dependent manifolds on which the process admits a law.   Throughout the paper we will make several comparisons between the setting of \cite{Bismut} and the present setting.

\subsection{Main Results}\label{intro:mainresults} The main results of this paper are the following: Proposition \ref{correachstoc}, Proposition \ref{prop:dimcanonlydecrease} and  Proposition \ref{thm:meszerosetsunderinvmeas} give a  description of the global behaviour of the dynamics \eqref{SDE}, under the sole assumption that the vector fields $V_0 \dd V_d$ satisfy the UFG condition; Theorem \ref{thm:mainglobalthmforX}  and Theorem  \ref{thm:mainlocalthm} describe   the long time behaviour of non-autonomous hypoelliptic processes and of UFG processes, respectively, identifying invariant measures and  characterizing their basin of attraction; finally in  Theorem \ref{thmsec8} we describe appropriate manifolds where the process $X_t$ admits a density. 
Let us give a rather informal description of  such  results. Precise notation, assumptions and statements are deferred to the relevant sections.

A {\em distribution} $\Delta$ on $\R^N$ is a map that, to each point $x \in \R^N$, associates a linear subspace of the tangent space $T_x \R^N$.   Given a set $ \set $ of smooth vector fields on $\R^N$, the distribution generated by $\set$, denoted by $\Delta_{\set}$,  is the map $x \mapsto \mathrm{span}\{X(x): X \in \set\}$. Let us introduce two  distributions,  $\deln(x)$ and  $\delnn(x)$, that will play a fundamental role in this paper. To avoid having to set too much notation and nomenclature, we introduce them now informally but we will give  precise definitions at the beginning of Section \ref{sec:geomofUFGprocesses}. \footnote{In that section we define them differently, but we then prove that the definition we give there is equivalent to the one we state in \eqref{bowie1} - \eqref{bowie2}. } The distribution  $\deln$ is generated by the vector fields contained in the Lie algebra ({\bf PHC}), i.e. the distribution
\be\label{bowie1}
\deln(x) = \bigcup_{j\geq 1} \mathrm{span} \{\mathfrak{L}_j(x)\}
\ee
 while 
\begin{align}
\delnn(x)& = \mathrm{span}\{\mathrm{Lie}\{V_0(x), V_1(x), \dd V_d(x)\}\}\label{bowie2}\\
&=\mathrm{span} \{V_0(x) \cup \deln(x)\}\,.
\end{align}
Clearly, $\deln(x) \subseteq \delnn(x)$ for every $x \in \R^N$ and the two distributions coincide at $x$ if and only if $V_0(x)$ is a combination of the vectors contained in $\deln$. More precisely, we decompose the vector $V_0$ into a component which belongs to $\deln$, $\vodel$,  and a component which is orthogonal to $\deln$, $\voperp$:
\be\label{defvoperp}
V_0= \vodel + \voperp \, .
\ee
 In other words, $\voperp(x)$ is the projection of $V_0(x)$ on the orthogonal of the vector space $\deln(x)$, so the two distributions coincide if and only if $\voperp=0$. We will see that the vector $\voperp$ plays an important role for the dynamics and, ultimately, it is the component of $V_0$ responsible for the lack of smoothness in the direction $V_0$. \footnote{Note that even when $V_0$ is smooth, $\voperp$ need not be smooth, see Note \ref{technicalpoint} on this matter. } Therefore, in a way, the distribution $\deln$ is the one containing all the directions along which the problem \eqref{linearPDE} is smooth. We will come back to this later. 

Under the UFG condition the integral manifolds (see Section \ref{subsec:geometry} for definition) of $\delnn$ form a partition of the state space $\R^N$. Let $\curlys$ be one such manifold.  \footnote{By definition of integral manifold, on each one of these manifolds the rank of the distribution $\delnn$   is constant and it is equal to the dimension of the manifold itself.} If $X_0 = x \in \curlys$ then $X_t^{(x)} \in \overline{\curlys}$ for all $t \geq 0$. That is, if the process starts from one of the manifolds of the partition, then it remains in the closure of such a manifold;  but, crucially,  it may hit the (topological) boundary $\pa \curlys := \overline{\curlys}\setminus \curlys$ of the manifold $\curlys$. This is the content of Proposition \ref{correachstoc}. Such a statement is obtained by combining the known geometric theory of distributions with non-constant rank with the classical Stroock--Varadhan support theorem.  We further prove that if $X_t$ hits the boundary $\pa \curlys$ of the manifold $\curlys$, then it never leaves it, see Proposition \ref{prop:dimcanonlydecrease} and Note \ref{Note5.2}.   Therefore: i) because the dimension of the boundary $\pa \curlys$ is smaller than the dimension of $\curlys$, along the path of $X_t^{(x)}$ the rank of the distribution cannot increase; ii) if the solution of the SDE leaves the manifold $\curlys$ from where it started, then any invariant measure can only be supported on the boundary $\pa \curlys$ of such a manifold, see Proposition \ref{thm:meszerosetsunderinvmeas}.

Further understanding of the dynamics relies on  the results of  Section \ref{sec:Coordinatechange}: in this section  we show that, after an appropriate change of coordinates, any $N$-dimensional SDE of UFG-type can be written, at least locally, as a system of the form
\begin{align}
dZ_t & = U_0(Z_t, \hat{\zeta}_t) dt + \sum_{j=1}^d U_j(Z_t, \hat{\zeta}_t) \circ dB^j_t  \label{OS1}\\
d{\hat{\zeta}}_t &= \hat{W_0}(\hat{\zeta}_t) dt \, , \label{OS2}
\end{align}
where $\hat{\zeta}_t$ solves an ordinary differential equation (ODE), $\hat{\zeta}_t \in \R^{N-n}$,  \footnote{To make a link with the more precise notation that we will use in Section \ref{sec:Coordinatechange}, we are denoting  here by $\hat{\zeta}_t$ the components $(\zeta_t, a_t)$ in \eqref{ashtree2}-\eqref{ashtree3}, i.e. $\hat{\zeta}_t = (\zeta_t, a_t)$.  } $Z_t \in \R^n$,  $\hat{W_0}:\R^{N-n} \rar \R^{N-n}$ and $U_i:\R^N \rar \R^n$ for every $i \in \{0 \dd d\}$.  Beyond details about the dimensionality of the ODE component,  the important thing is that the solution of the ODE $\hat{\zeta}_t$ evolves independently of the SDE part, while the coefficients of the SDE depend on the evolution of the ODE. We will informally refer to such a representation as being of the form ``ODE+SDE". In general, this representation is only local. This change of coordinates has been known for a long time in differential geometry, at least since Frobenius, see \cite{Isidori};  here we are simply  expressing it in a way which is more congenial to our setting and purposes and we apply it to SDEs.  While the change of coordinates itself is not new, to the best of our knowledge it has  been  used in  SDE theory only by Bismut in \cite[Section 5]{Bismut}, to study the density of SDEs that satisfy the  H\"ormander Condition (HC), but it has never been used to study the long-time behaviour of SDEs. To clearly compare our work with \cite{Bismut}, let us emphasize that  in the notation introduced so far, the HC is satisfied when the distribution $\delnn$ has rank equal to $N$ at every point (see Section \ref{sec:preliminaries} for more details).  In this paper we primarily exploit the representation \eqref{OS1}-\eqref{OS2} to study the long time behaviour of SDEs that satisfy the UFG condition (so, as we said already, the rank of $\delnn$ is not constant) but  we also use it briefly in Section \ref{sec:8} to study the density of UFG processes.\footnote{Further comparisons between the setting of  \cite{Bismut} and the setting of this paper can be found latesr in this section,  in Note \ref{note:compareBismut2} and in  Note \ref{note:compareBismut1}.}   This local representation is both an important technical tool throughout the paper and a fundamental element in understanding the evolution of the dynamics. Referring to the PDE \eqref{linearPDE}, we also note here that the change of coordinates gives a geometric interpretation of the (potential) lack of  smoothness in the direction $V_0$ and of the reason why  smoothness is instead maintained in the direction $\mathcal{V}$, see  Note \ref{noteonsmoothness} on this point. 
 
In view of the discussed change of coordinates, it makes sense to start by studying UFG dynamics for which the representation \eqref{OS1}-\eqref{OS2} is global. For this reason in Section \ref{sec:non-autonomous} we consider systems which are (globally) of the form \eqref{OS1}-\eqref{OS2}, where the ODE is assumed to be one-dimensional and the SDE satisfies a form of H\"ormander condition. 
More precisely, the dynamics studied in Section \ref{sec:non-autonomous} are  non-autonomous hypoelliptic SDEs; because the topic is somewhat of independent interest, this section has been written in such a way that it can be read independently of the rest of the paper. Non-autonomous SDEs and their associated two-parameter semigroup have been studied in \cite{Cattiaux}, where a detailed analysis of the law of the process is carried out, in \cite{CrisanMcMurray} where the associated semigroup is examined, and in \cite{Kunze, daPratoRoeckner}, where the authors introduce  some interesting techniques to deal with the analysis of  invariant measures and  long-time behaviour of  time-inhomogeneous processes. The work \cite{Cattiaux} assumes that the non-autonomous SDE is hypoelliptic, while in \cite{Kunze} a uniform ellipticity assumption is enforced.  From a technical point of view, the results of Section \ref{sec:non-autonomous} extend the approach of  \cite[Section 6.1]{Kunze} to the hypoelliptic setting. However the main difference between our results and the results in \cite{Kunze} is that here we  highlight the fact  that the process may admit several invariant measures and we characterize the basin of attraction of each of them. In this setting convergence to equilibrium is driven by the ODE component. We will indeed show that the basin of attraction of each invariant measure can be completely described by just looking at the behaviour of the solution of the ODE. Because the ODE is assumed to be one-dimensional and autonomous, it can only behave monotonically, so the analysis of the ODE and of the full problem is relatively intuitive in this setting (see Section \ref{sec:non-autonomous} for details).   
%As one may expect, we show that invariant measures of the process %$X_t \in \R^N$ can only be  supported on $(N-1)$-dimensional %hyperplanes of the form $\mathcal{H}_{\barz}:=\{x \in \R^N: %x^N=\barz\}$, where $\barz$ is a stationary point of the ODE. 

In Section \ref{sec:longtimebehaviour} we consider the general case of UFG processes for which the representation  \eqref{OS1}-\eqref{OS2} is only local. While this case is substantially richer than the previous one, the fact that, locally, we can always represent the SDE \eqref{SDE} as a system of the form ODE+SDE, still means that there is some deterministic behaviour which is intrinsic to UFG dynamics. It turns out that one is still able to single out the deterministic behaviour. Recalling the definition of the vector   
 $\voperp$, formula \eqref{defvoperp}, we will show that the ($N$-dimensional) ODE 
$$
d\zeta_t = \voperp (\zeta_t) dt
$$
plays, in this more general context, the same driving role that the ODE \eqref{OS2} had in the context of Section \ref{sec:non-autonomous}.  Motivated by the above discussion,  we introduce the process
$$
\cZ_t:=e^{-t\voperp}(X_t).
$$
This process is non-autonomous and, as we will explain, it can be interpreted geometrically as being a projection of the process $X_t$ on an appropriate integral manifold of the distribution $\deln$.  We apply the techniques of Section \ref{sec:non-autonomous}  to the study of such a non-autonomous process, producing results on the long-time behaviour of $\cZ_t$. We then relate the asymptotic behaviour of $\cZ_t$  to the asymptotic behaviour of $X_t$. 
Notice that the procedure that we have just described  is somewhat the reverse of the one that is traditionally used (and it is, to the best of our knowledge, new): given a non-autonomous system, the established methodology consists of increasing the dimension of the state space by adding time as an auxiliary variable, thereby reducing the given non-autonomous system to a (larger) autonomous one. Here we do the converse: by projecting the process on an appropriate manifold, we reduce to a (lower-dimensional) non-autonomous one, $\mathcal{Z}_t$,  with the advantage that now the techniques of Section \ref{sec:non-autonomous} can be adapted to prove statements on  $\mathcal{Z}_t$. Once the latter process has been understood, we deduce results about  the autonomous process $X_t$ from those shown for $\mathcal{Z}_t$. 

From a probabilistic point of view it is clear that, in absence of the H\"ormander condition, we cannot expect the process $X_t$ to have a density with respect to the Lebesgue measure. This is made explicit by the local representation \eqref{OS1}-\eqref{OS2}, which also clarifies that it is the ODE component to be responsible for the lack of smoothness. Notice also that, in the coordinates \eqref{OS1}-\eqref{OS2}, the vector $\voperp$ is given by $\voperp=(0 \dd 0, \hat{W_0})$, i.e. it is precisely the vector driving the ODE behaviour (we have elaborated on this fact in Note \ref{noteonsmoothness}).  However in Section \ref{sec:8} we show that the law of the SDE \eqref{SDE} still has a density on an appropriate time-dependent submanifold, which can be explicitly described. In order to do so, we correct and then extend the results of \cite{Schiltz}.

One may also wish to point out that systems of the form ``ODE+SDE" appear as diffusion limits of some Metropolis-Hastings type of algorithms, see e.g. \cite{RWM}. It is noted in \cite{RWM} that one may use the ODE as a way to monitor convergence.  We believe that the lack of smoothness of UFG processes could be seen as a perk in the context of sampling. The authors intend to explore this fact in future work.  
Finally, we mention in passing that UFG processes play a fundamental role in the study of cubature methods, see \cite{Crisan} and references therein  for a complete account on the matter.

\subsection{Organisation of the paper}\label{intro:organization}
In Section \ref{sec:notation} we introduce the standing notation for the remainder of the manuscript.  To make the paper self-contained, in Section \ref{sec:preliminaries} we gather background definitions and notions. In particular  Subsection \ref{subsec:UFG} contains details of the UFG condition, while Subsection \ref{subsec:geometry} covers basic definitions and standard results in differential geometry and (stochastic) control theory. In Section \ref{sec:geomofUFGprocesses} we exploit the existing theory of distributions of non-constant rank to produce both global and local results about the SDE \eqref{SDE}, under the UFG condition. In Section \ref{Sec4:global} we cover the {\em global} behaviour of the SDE, in Section \ref{sec:Coordinatechange} we study {\em local} properties. In Section \ref{sec:5} we introduce several results for UFG-diffusions. These results are quite general, in the sense that most of them valid under just the UFG condition. The following Section \ref{sec:non-autonomous} can be read independently of the rest of the manuscript: in this section  we describe  the long-time behaviour of hypoelliptic SDEs of non-autonomous type. The class of SDEs considered in Section \ref{sec:non-autonomous}  is one for which the representation of the form ``ODE+SDE" is global. This is the first section where we address the problem of studying the basin of attraction of different invariant measures. In Section \ref{sec:longtimebehaviour} we instead study the long time behaviour of \eqref{SDE} in the general  UFG case (in which the change of coordinates is only local).  Section \ref{sec:8} is devoted to the study of the density of the process, via Malliavin calculus. Finally, for ease of reading, we chose to relegate almost all the proofs  to the appendix. In particular, Appendix \ref{AppendixA} contains some needed miscellaneous technical facts, Appendix \ref{sec:mainproofs}  contains the proofs of all the statements contained in Section \ref{sec:preliminaries} to Section \ref{sec:8}.

\section{Notation}\label{sec:notation}
We will be interested in $N$-dimensional SDEs, of the form \eqref{SDE}. The letter $N$ will only be used to refer to the dimension of the state space. While examples of UFG diffusions can be found in any dimension, it is fair to say that the theory we develop in  this paper is mostly interesting in dimension $N \geq 2$, so we will make this a standing assumption which will hold unless otherwise stated in specific examples. 

 If $x$ is a point in $\R^N$, we denote the $j$-th  coordinate of $x$ by $x^j$, i.e.  $x= (x^1 \dd x^N)$ (this is coherent with \eqref{notvectoper}). We will often use a local change of coordinates, presented in Section \ref{sec:Coordinatechange}. The change of coordinates will be given by a local diffeomorphism $\Phi :\R^N \rar \R^N$ and the new coordinates will typically be denoted by $\bfz$, i.e. $\bfz=\Phi(x)$. 
In the new coordinate system it will be of particular importance to distinguish the role of the first $n$ coordinates of $\bfz$ from the others ($n$ being an appropriate integer, $n<N$). In particular, if $N-n>1$, we will use the following notation
\begin{align}
\bfz & =(\underbrace{z^1\dd z^n},\underbrace{z^{n+1}}, \underbrace{z^{n+2} \dd z^N}) = (z, \zeta, a),
\label{coordinateblocks}\\
&\qquad  \quad \,\, z \,\,\, \,\, \, \qquad\zeta \,\,\, \qquad 
\quad \,  a \nonumber 
\end{align}
where $(z^1 \dd z^n)=z \in \R^n$, $z^{n+1}=\zeta \in \R$ and $(z^{n+2}\dd z^N)=a \in \R^{N-(n+1)}$. The last block of coordinates plays a role which is different from the role of the first two blocks, as it will be explained (the coordinates in the last block should more be intended as parameters).
If $N=n+1$ then simply $\bfz =(z, \zeta)$. 

A similar  reasoning holds for the vector fields appearing in \eqref{SDE}: for any $j \in \{0 \dd d\}$,  $V_j= (V_j^1 \dd V_j^N)$ and $\tilde{V}_j$ will denote the vector $V_j$,  expressed in the new coordinate system $\bfz=\Phi(x)$. We will show that  in the new coordinate system, one has
\begin{align}
& \tilde{V}_j(\bfz)= (U_j(\bfz), 0 \dd 0) \qquad j =1 \dd d  \label{newvectors1}\\
& \tilde{V}_0= (U_0(\bfz), W_0(\zeta,a), 0 \dd 0) \,, \label{newvectors0}
\end{align}
where $U_j: \R^N \rar \R^n$ while $W_0$ is a real-valued function which depends only on the last two blocks of coordinates of $z$, i.e. $W_0: \R^{N-n} \rar \R$. 

Accordingly, if $\R^N \ni X_t$ is the solution at time $t$ of the SDE \eqref{SDE}, then $X_t^j$ denotes the $j$-th component of $X_t$. We will sometimes want to stress the dependence of the solution $X_t$ on the initial datum;  when this is the case,  we will write $X_t^{(x)}$ if $X_0=x$.  Finally, given a probability measure $\mu$ and a function $f$ which is integrable with respect to $\mu$, we shall define $\mu(f)$ by
$$
\mu(f):=\int f d\mu.
$$

We shall use the following function spaces throughout the paper. For any $N\geq 1$ and closed set $E\subseteq \R^N$;
\begin{itemize}
\item We denote by $C_b(E)$ the space of all functions $f:E\to\R$ which are continuous and bounded;  this space will be endowed with the supremum norm. 
\item We say that a real-valued function $f$ is $C^\infty$ if it has continuous derivatives of all orders. 
\item We denote by $C_c^\infty(\R^N)$ the set of all  functions $f:\R^N\to\R$ which are $C^{\infty}$ and with compact support. % \footnote{Here we are endowing $E$ with the Euclidean topology as a subset of $\R^N$.} 
\end{itemize}

Given a differentiable function $f:\R^N\to\R^N$ we denote by $\jacobian{f}{x}$ the Jacobian matrix of $f$, that is $(\jacobian{f}{x})^{ij}(x) = \partial_{x^j}f^i(x)$.

\section{Preliminaries and Assumptions}\label{sec:preliminaries}

\subsection{The UFG condition}\label{subsec:UFG}
Fix $d \in \N$ and let $\A$ be the set of all $k$-tuples, of any size $k \geq 1$,  of integers of the following form 
$$
\A:= \{\alpha=(\alpha^1 \dd \alpha^{k}), k \in \N : \alpha^j\in \{0, 1 \dd d\} {\mbox{ for all } j \geq  1} \}\setminus \{(0)\}\,.
$$
We emphasise that all $k$-tuples of any length $k\geq 1$ are allowed in $\A$,  except the trivial one, $\alpha=(0)$ (however singletons $\alpha=(j)$ belongs to $\A$ if $j\in\{1\dd d\}$). 
We endow $\A$ with the product operation
$$
\alpha \ast \beta:= (\alpha^1 \dd \alpha^h, \beta^1 \dd \beta^{\ell}), 
$$
for any $\alpha=(\alpha^1 \dd \alpha^h)$ and 
$\beta=(\beta^1 \dd \beta^{\ell})$ in $\A$.  If  $\alpha\in\A$, we define the {\em length} of $\alpha$, denoted by $\|\alpha\|$, to be the integer
$$
\|\alpha\|:= h+\mbox{card}\{i: \alpha_i=0\} , \qquad \mbox{if } \alpha=(\alpha^1 \dd \alpha^h) \,.
$$
For any $m \in \N, m\geq 1$, we then  introduce the sets
\begin{align*}
&\A_m=\{\alpha \in \A: \| \alpha\|\leq m   \} \,.
\end{align*}

Given a vector field (or, equivalently, a first order differential operator) $V=(V^1(x), V^2(x),$ $..., V^N(x))$  on $\R^N$, we refer to the functions $\{V^j(x)\}_{1 \leq j \leq N}$ as to  the {\em components} or {\em coefficients} of the vector field.   We  say that a vector field is smooth or that it is $C^{\infty}$ if all the components $V^j(x)$, $j=1 \dd N$, are $C^{\infty}$ functions. 
Given two differential operators $V$ and $W$, the commutator between $V$ and $W$ is defined as
$$
[V,W]:= VW -WV\,.
$$
Let now $\{V_i: i=0 \dd d \}$ be a collection of vector fields on $\R^N$ and let us define the following ``hierarchy" of operators:
\begin{align*}
V_{[i]} &:= V_i \qquad i=0, 1 \dd d\\
V_{[\alpha \ast i]} & := [V_{[\alpha]}, V_{[i]}], \qquad \alpha \in \A, i=0,1 \dd d\,.
\end{align*}
 This hierarchy is completely analogous to the one constructed in the Introduction, here we just need a more detailed notation. 
Note that if $\| \alpha\|=h$ then $\| \alpha \ast i \| = h+1$ if $i \in \{1 \dd d\}$ and $\| \alpha \ast i\|= h+2$ if $i=0$. If $\alpha \in \A$ is a multi-index of length $h$, with abuse of nomenclature we will say that $V_{[\alpha]}$ is a  differential operator of length $h$. 
 We can then define the space  $\mathcal{R}_m$ to be  the space containing all the operators of the above  hierarchy, up to and including the operators of length $m$ (but excluding $V_0$):
\be\label{(R)}
\mathcal{R}_m:=\left\{ V_{[\alpha]}, \alpha \in \A_m\right\}. 
\ee
Let $C^{\infty}_V(\R^N)$ denote the set  of  bounded smooth functions,  $\varphi= 
\varphi(x): \R^N \rightarrow \R$,  which have bounded derivatives of all orders and such that
\be\label{supphi}
\sup_{x \in \R^N}\lv V_{[\gamma_{1}]} \dots V_{[\gamma_{k}]} \varphi(x) \rv < \infty
\ee
for all $k$ and all $\gamma_{1} \dd \gamma_{k} \in \A_m$. \footnote{The definition of the set $C^{\infty}_V(\R^N)$ depends on $m$ as well, but we do not include this dependence in the notation for simplicity.}
With this notation in place we can now introduce the definition that will be central in this paper.

\begin{definition}[UFG Condition]\label{defufg}
Let $\{V_i: i=0 \dd d \}$ be a collection of %$C^{\infty}$ 
smooth vector fields on $\R^N$ and assume that the coefficients of such vector fields have bounded partial derivatives (of any order). We say that the fields  $\{V_i: i=0 \dd d \}$ satisfy the UFG condition if there exists $m\in\N$ such that for any $\alpha \in \A$ of the form 
$$
\alpha = \alpha'\ast i, \qquad    \alpha' \in \A_m, \, i \in \{0 \dd d\}, $$ 
 one can find  bounded smooth functions $\varphi_{\alpha, \beta}=\varphi_{\alpha, \beta}(x) \in C^{\infty}_V(\R^N)$ such that
\begin{equation}\label{eq:UFG}
V_{[\alpha]}(x) = \sum_{\beta \in \A_m} \varphi_{\alpha,\beta}(x) V_{[\beta]} (x) \,.
\end{equation}

\end{definition}
Again we emphasize that the set of vector fields appearing in the linear combination on the right hand side of the above identity does not include $V_{0}$. 
It may be useful to compare the UFG condition with the H\"ormander condition (HC), the uniform parabolic H\"ormander condition (UPHC) and the Parabolic H\"ormander condition (PHC),  which we recall. The HC is satisfied if 
\begin{equation}
\mathrm{span}\left(\mathrm{Lie}\{V_0(x) \dd V_d(x)\}\right)=\R^N \quad \mbox{for every } x \in \R^N\,.  \tag{\bf HC}
\end{equation}
In other words, the HC is satisfied if $\delnn(x)=\R^N$ for every $x \in \R^N$. 
The PHC has been recalled in the introduction, see ({\bf PHC}). We notice in passing that while the space $\mathcal{R}_m$ is in general different from the space $\mathfrak{L}_m$, \footnote{  $V_{[\alpha]} \in \mathfrak{L}_j$ if and only if $\| \alpha\| \leq 1+2 (j-1) =2j-1$, hence $\mathcal{R}_{2j-1}= \cup_{k=1}^j \mathfrak{L}_j$. }it is the case that $\cup_{j\geq 1}\mathcal{R}_j= \cup_{j\geq 1} \mathfrak{L}_j$.  The UPHC  (see \cite{{CrisanGhazali}}) is instead satisfied if 
\be
\exists \,\ell \geq 1 \mbox{ and } \kappa>0 : \, \sum_{\alpha \in \mathcal{A}_{\ell}}\lv V_{[\alpha]}(x) \cdot y \rv^2 \geq \kappa \lv y \rv^2 \quad  \mbox{for every } x,y \in \R^N\,.  \tag{\bf UPHC}
\ee
In the above each term of the sum is the scalar product between the vector $V_{[\alpha]}(x)$ and the vector $y\in \R^N$.  
Notice that the UPHC is the strongest of all these conditions, in the sense that
\begin{align*}
(\bf{UPHC}) & \Rightarrow (\bf{PHC}) \Rightarrow (\bf{HC})\\
(\bf{UPHC}) & \Rightarrow (\bf{UFG}) \,.
\end{align*}
However neither the HC nor the PHC imply the UFG condition (as one may, in general, need infinitely many fields to satisfy the PHC or the HC). We also note that while the various  H\"ormander conditions are imposed on  an appropriate Lie Algebra,  the UFG condition is rather a condition on the set of vectors $\{V_{[\alpha]}, \alpha \in \A_m\}$, seen as a  module over the ring $C^{\infty}_V$.  

\begin{example}\label{ex:GBM}
Consider one-dimensional geometric Brownian motion, that is, the solution of the following SDE: 
\begin{align*}
dX_t&=-X_tdt+\sqrt{2}X_t dB_t\\
&=-2X_tdt+\sqrt{2}X_t\circ dB_t.
\end{align*}
Here $V_1=x\partial_x, V_0=-2V_1$ and $[V_1,V_0]=0$. These vector fields satisfy the UFG condition with $m=1$ however $V_0$ and $V_1$ vanish when $x=0$ so the HC is not satisfied.
\end{example}

\begin{example}\textup{
Consider the following first order differential operators on $\R^2$
$$
V_0=\sin x \,\partial_y \qquad V_1=\sin x\,\partial_x \,.
$$
Then $\{V_0,V_1\}$ do not satisfy the H\"ormander condition (e.g. there is always a degeneracy at $x=0$) but they do satisfy the UFG condition with $m=4$. If the role of the fields is exchanged, i.e. if we set 
$$
V_0=\sin x\,\partial_x ,  \qquad V_1=  \sin x \,\partial_y\,
$$
then $\{V_0, V_1\}$ still satisfy the UFG condition, this time with $m=1$ (indeed,  $[V_0,V_1]=\cos x V_1$).
{\hfill $\Box$}}
\end{example}

\begin{note}\label{note:higherorderUFG}\textup{ Because the functions $\varphi_{\alpha,\beta}$ appearing in \eqref{eq:UFG} belong to $C_V^\infty(\R^N)$, if the UFG condition holds for some $m \in \N$ then it also holds for any $\ell \geq m, \ell \in \N$.  In other words, if the UFG condition holds for some $m$ in $\N$ then for any $V_{[\gamma]}$ with $\|\gamma\|> m$ one has 
$$
V_{[\gamma]}(x)= \sum_{\beta \in \A_m} \varphi_{\gamma,\beta}(x) V_{[\beta]} (x) \, ,
$$
for some    functions $\varphi_{\gamma,\beta}\in C_{V}^{\infty}(\R^N)$. 
For this reason it is appropriate to remark that in the remainder of the paper, when we assume that ``the UFG condition is satisfied for some $m$", we mean the smallest such $m$.  
}{\hfill$\Box$} 
\end{note}

We will  consider diffusion semigroups $\{\cP_t\}_{t\geq 0}$ of the form \eqref{semigroup}; that is, we consider Markov semigroups associated with the stochastic dynamics \eqref{SDE}. In particular, we will be interested in studying the semigroup $\cP_t$ when the vector fields $\{V_0, V_1 \dd V_d\}$ satisfy the UFG condition. 

As we have already mentioned, the UFG condition is strictly weaker than the uniform Parabolic H\"ormander condition. However one can still prove that, when the UFG condition is satisfied by the vector fields $\{V_0, V_1\dd  V_d\}$ appearing in the generator \eqref{gen1}, the semigroup $\cP_t$ still enjoys good smoothing properties: if $f(x)$ is continuous then $(\cP_t f)(x)$ is differentiable (infinitely many times) in all the directions spanned by the vector fields contained in $\mathcal{R}_m$ (we recall that the set $\mathcal{R}_m$ is defined in \eqref{(R)}). See Appendix \ref{app:UFG} for more details. 
%\cite[Section 2.9]{Crisan}.}
 
 \subsection{Obtuse Angle Condition}
When the semigroup $\cP_t$ is elliptic or  hypoelliptic, several works have dealt with the study of the long and short time behaviour of the derivatives of the semigroup, for a review see \cite{Bakry, Lorenzi}. To the best of our knowledge, the only work addressing the study of  the long-time behaviour of the derivatives of UFG semigroups is \cite{CrisanOttobre}.  In \cite{CrisanOttobre} the authors identify a sufficient condition for exponential decay of the derivatives of the solution of \eqref{linearPDE}. To be more precise, they proved the following: suppose the vector fields $\{V_0, V_1 \dd V_d\}$ satisfy the UFG condition and assume   there exists $\lambda_0>0$ such that for all $f$ sufficiently smooth and  for every $\alpha\in \mathcal{A}_m$ we have 
\be\label{eq:OAC}
([V_{[\alpha\ast 0]}]f(x))(V_{[\alpha]}f(x)) \leq -\lambda_0 \lv V_{[\alpha]} f(x)\rv^2, \quad \mbox{for every } x \in \R^N\,.
\ee
If $\lambda_0$ is sufficiently large then, for every $r>0$  and $t_0>0$, we may find a constant $c_{t_0,r}>0$ such that for any $f\in C_b(\mathbb{R})$, $t\geq t_0$ and $\alpha\in\mathcal{A}_m$ we have
\be\label{eq:graddecayest}
\sup_{\mathbb{B}(0,r)}\lv V_{[\alpha]}(\cP_tf)(x) \rv \leq c_{r,t_0} \lVert f \rVert_\infty e^{-\lambda (t-t_0)},
\ee
for some $\lambda>0$ (which depends on $\lambda_0$). In the above $\mathbb{B}(0,r)$ is the centered ball (of $\R^N$) of radius $r$. Condition \eqref{eq:OAC} was named the {\em Obtuse Angle Condition} (OAC) in \cite{CrisanOttobre}. Here we will need a second order version of such a result, as well.

\begin{lemma}\label{seconorderdecay}
Let $\cP_t$ be the semigroup associated with the SDE \eqref{SDE} and assume that the vector fields $V_0 \dd V_d$ satisfy the UFG condition. Suppose moreover that the following holds: there exists $\lambda_0>0$ such that
\be \label{eq:OAC2}
(V_{[\alpha]}V_{[\beta]}f)(x) \, ([V_{[\alpha]}V_{[\beta]}, V_0] f)(x) \leq  - \lambda_0 \lv (V_{[\alpha]}V_{[\beta]}f)(x) \rv^2, 
\ee
for every $x \in\R^N$ and for all $\alpha, \beta \in \A_m$ such that $\alpha \neq \beta$ and $\alpha, \beta \notin \{1 \dd d\}$. 
If $\lambda_0>0$ is large enough then,  for any $t_0\in (0,1)$ and any $r>0$
 there exists a constant $c_{t_0,r}>0$ such that, for some $\lambda>0$ (which depends, among other things, on $\lambda_0$),  one has  
\be\label{seconorderdecayeq}
\sup_{x \in \mathbb{B}(0,r)}\lv V_{[\beta]}V_{[\alpha]} (\cP_tf) (x)\rv ^2 \leq c_{t_{0},r} \,  e^{-\lambda (t-t_0)}
\|f\|_{\infty},  
\ee
for all $ \alpha, \beta \in \A_m $, all $t>t_0$ and for every $f$ continuous and bounded. 
\end{lemma}
\begin{proof}[Proof of Lemma \ref{seconorderdecay}]
%In \cite{{CrisanOttobre}}, the authors proved estimates of the type \eqref{seconorderdecayeq} for first order derivatives of the semigroup $\cP_t$. More precisely, they show that if \eqref{eq:OAC} is satisfied, then \eqref{eq:graddecayest} holds. 
The proof is simple and can be done by following the same procedures presented in \cite{CrisanOttobre}, using the modifications outlined  in \cite[comments after Corollary 4.9]{CrisanOttobre}. We do not repeat all the details here
\end{proof} 
\begin{example}[UFG condition and Obtuse Angle Condition for linear SDEs]\textup{
	Consider  SDEs in $\R^N$ of the form
	\begin{equation} \label{linearUFG}
	dX_t=(AX_t+D)dt+\sqrt{2}\sum_{i=1}^d C_idB_t^i \, ,
	\end{equation}
	where  $A$ is a constant $N\times N$ matrix,  $B_t^1,\ldots,B_t^d$ are one-dimensional standard Brownian motions and  $D,C_1,\ldots,C_d\in \R^N$ are constant vectors. In this case $V_0(x)=Ax+D, V_i(x)=C_i$, and 
	\begin{equation*}
	V_{[i\ast 0]}=[V_i,V_0]=AC_i, \quad i\in\{1,\ldots,d\}.
	\end{equation*}
	Because $[V_i, V_j]=0$ for every $i,j \in \{1 \dd d\}$, the only relevant commutators are those of the form $V_{[(i,0,\ldots,0)]}$, i.e. repeated commutators with $V_0$. For simplicity, let $\alpha_{i,k}$ be the $(k+1)$-tuple such that $\alpha_{i,k}^1=i$ and $\alpha_{i,k}^j=0$ for $j>1$;  then 
	\begin{equation*}
	V_{[(i,\underbrace{0,\ldots,0}_{k \text{ times}})]} = A^kC_i.
%	&k \text{  times} 
	\end{equation*}
It is now easy to show that, {\em irrespective of the choice of  $A, D, C_1 \dd C_d$ as above,  the UFG condition is always satisfied by SDEs of the form \eqref{linearUFG}}. Indeed, by the Cayley Hamilton Theorem there is a polynomial $p$ of degree at most $N-1$ such that
	$
	A^{N}=p(A)$;  
	so we can write any $V_{[\alpha_{i,k}]}$ as a linear combination of the vectors $V_{[\alpha_{i,\ell}]}$ with $\ell\leq N$.}
	\textup{For comparison we recall that \eqref{linearUFG} is hypoelliptic if and only if the Kalman rank condition is satisfied, namely if
$$
\mathrm{rank}[Q, AQ, A^2Q \dd A^{N-1}Q] = N,
$$
where $Q$ is the overall diffusion matrix of \eqref{linearUFG}, see e.g.  \cite{Lorenzi}. 
As for  the OAC \eqref{eq:OAC}, this is  is satisfied if and only if there exists some  $\lambda>0$ such that for all $f\in C^1(\R^N)$ we have 
	\begin{equation*}\label{eq:linearOAC}
	(\nabla f)^T A^{k+1}C_i C_i^T (A^k)^T \nabla f \leq -\lambda_0 (\nabla f)^T A^{k}C_i C_i^T (A^k)^T (\nabla f) \, ,
	\end{equation*}
for all $i\in\{1,\ldots,d\}$ and $k\in\{0,1,\ldots, N-1\}$. This holds if and only if
\begin{equation*}
(A+\lambda_0 I) G \leq 0
\end{equation*}
for some $\lambda_0>0$, 
where $G=A^kC_iC_i^T(A^k)^T$ for all $i\in\{1,\ldots,d\}$ and $k\in\{0,1,\ldots, N-1\}$.
	\hf
%	Also note that H\"ormander's condition is satisfied if and only if
%	\begin{equation}\label{eq:linearpHC}
%	\mathrm{span}(C_i, AC_i,\ldots A^{N_1}C_i: i\in\{1,\ldots,d\})=\R^N.
%	\end{equation}
%	In particular, if we assume that \eqref{eq:linearpHC} holds then \eqref{eq:linearOAC} is equivalent to requiring that $A$ is strictly negative definite. 
}
\end{example}

\subsection{Geometry}\label{subsec:geometry}

In this section we cover some  basic notions from differential geometry and geometric control theory  on which the rest  of the paper relies. Further details can be found in the excellent references \cite{Sontag, Isidori, Sussman}. For the reader who is already familiar with this material, we point out that, among the results recalled  in  this section, Theorem \ref{thm:Hermann} is possibly  the one which will play the most important role in the remainder of the paper.

Given a vector field $V(x)$ on $\R^N$, we denote by $e^{tV}(x)$ the integral curve of $V$ starting at $t=0$ from $x$, i.e. the curve $\gamma(t): \R \rightarrow \R^N$ such that $\gamma(0)=x$ and $\dot{\gamma}(t)= V(\gamma(t))$ for all $t \in \R$  such that the curve is defined. In general,  integral curves exist only locally. In this paper we consider only smooth, globally defined and globally Lipschitz vector fields (see Hypothesis \ref{SA}), so  integral curves actually exist for every $t\in \R$. 
As already mentioned, a {\em distribution} $\Delta$ on $\R^N$ is a map that, to each point $x \in \R^N$, associates a linear subspace of the tangent space $T_x \R^N$.   Given a set $ \set $ of smooth vector fields on $\R^N$, the distribution generated by $\set$, denoted by $\Delta_{\set}$,  is the map $x \rightarrow \mathrm{span}\{X(x): X \in \set\}$.  Distributions generated by a set of smooth vector fields  are usually referred to as  {\em smooth distributions}. When we write $\Delta_{\set}$ instead of just $\Delta$ it is understood that we are considering smooth distributions rather than general distributions. As customary, we say that the vector field $X$ on $\R^N$ belongs to the distribution $\Delta$ if $X(x) \in \Delta(x)$ for all $x \in \R^N$.  The {\em rank} of $\Delta$ at $x$ is the dimension of the vector space $\Delta(x)$. A {\em piecewise integral curve}, $\gamma$,  of vector fields in the set $\set$ is a curve of the form
$$
\gamma(t_1 \dd t_h)= e^{t_1X_1 }e^{t_2 X_2}\, {\cdots} \, e^{t_h X_h} x          \qquad  h \in \N,  t_j \in \R, x \in \R^N, 
$$
where $X_1\dd X_h \in \set$ (and they are not necessarily all distinct). A submanifold $M \subseteq \R^N$ is an {\em integral manifold} of $\Delta$ if 
$T_xM= \Delta(x)$ for every  $x \in M$. A {\em maximal integral manifold} (MIM) of  $\Delta$, $\mathcal{M}$,  is a connected integral manifold of $\Delta$ which is maximal in the sense  that every other connected integral manifold of $\Delta$ that contains $\mathcal{M}$ coincides with $\mathcal{M}$.  Therefore, two MIMs either coincide or they are disjoint.
\begin{definition}\label{basicdefditstr}
Let $\Delta$ be a distribution on $\R^N$. 
\begin{itemize}
\item $\Delta$ is {\em involutive} if 
$$
X, Y \in \Delta \quad \Longrightarrow \quad [X,Y] \in \Delta \,.
$$
\item $\Delta$ is invariant under the vector field $V$ if the Jacobian matrix $\jacobian{(e^{tV}x)}{x}$ maps $\Delta (x)$ into $\Delta(e^{tV}x)$ for all $x$  and for all $t$.\footnote{A useful criterion to check whether a distribution is invariant under a vector field will be given in Note \ref{Tom}. } 
\item Suppose $\Delta$ is generated by the collection of vector fields $\set= \{X_1 \dd X_k\}$, i.e. $\Delta=\Delta_{\set}$. Then two points   $x,y \in \R^N$  belong to  the same {\em orbit} of $\Delta_{\set}$  if there exists a curve $\gamma: [a,b] \rightarrow \R^N$ such that $\gamma(a)=x$,  $\gamma(b)=y$ and $\gamma$ is a piecewise integral curve of vectors in  $\set$. 
\end{itemize}
\end{definition}

\noindent
In general, the integral manifolds of a given distribution are ``smaller" than the orbits; we refer the reader to   \cite{Sussman} for a detailed explanation of this matter, see in particular \cite[Eqn. (3.1)]{Sussman}. Here  we just illustrate this fact with  a simple but  important classical example. 
\begin{example}\label{ex:nonufgexample}\textup{In $\R^2$, consider the vector fields $X=\pa_x$ and $Y = \psi(x) \pa_y$ where $\psi(x)$ is a smooth function vanishing on the half-plane $x\leq 0$. The orbit of the distribution generated by $X$ and $Y$,  $\Delta_{X,Y}$, is the whole $\R^2$. That is, given any two points in $\R^2$ there is a piecewise integral curve of $\{X,Y\}$ which joins the two points. However the integral manifolds through points $(x,y)$ with $x\leq 0$ are one dimensional. Notice that the distribution in this example is involutive but it  satisfies neither the H\"ormander condition nor the UFG condition. More precisely, in the sense that whether we take $X=V_0$ and $Y=V_1$ or vicecersa, either ways the UFG condition is not satisfied (more precisely,  in the language of Definition \ref{def:LFT} below, the set $\{X,Y\}$ is neither locally nor globally of finite type). The fact that $\{X,Y\}$ don't satisfy the UFG condition can be either seen as a consequence of  Theorem \ref{thm:Hermann} below (if it did, the orbits would have to coincide with the integral manifolds) or it can be shown with direct calculations (the problem arising on the line $x=0$). For the reader's convenience this calculation is contained in the Appendix, see Lemma \ref{lem:nonufgexamplecont}. \hf
}
\end{example} 
We say that a distribution $\Delta$ on $\R^N$ satisfies the {\em (maximal) integral manifolds property}  if through every point of $\R^N$ there passes a (maximal) integral manifold of $\Delta$. 
The following fundamental result, due to Sussman (see \cite[Theorem 4.2]{Sussman}), completely characterizes the distributions enjoying the maximal integral manifolds property. 
\begin{theorem}[Sussman's Orbit Theorem]\label{thm:sussorbit}
If $\Delta =\Delta_{\set}$ is a smooth distribution on $\R^N$, then the following statements are equivalent  
\begin{description}
\item[(a)] $\Delta_{\set}$ satisfies the maximal integral manifolds property;
\item[(b)] $\Delta_{\set}$ satisfies the integral manifolds property;
\item[(c)] the orbits of $\Delta_{\set}$ coincide with the integral manifolds of the distribution and the rank of  $\Delta_{\set}$ at each point $x \in \R^N$ is equal to the dimension of the integral manifold of  $\Delta_{\set}$ through $x$; 
\item[(d)] $\Delta_{\set}$  coincides with the smallest distribution which contains the Lie algebra generated by $\set$, $Lie\{\set\}$,  and is invariant under the vectors in $\set$. 
\end{description}
\end{theorem}
In view of the equivalence of $(a)$ and $(b)$ above, when either property hold we just say that the smooth distribution is {\em integrable}. It is clear that in this case every integral manifold is a maximal integral manifold. 
 Some standard facts about integrable distributions which are useful to bear in mind and that follow (easily) from what we have said so far:  if $\Delta_{\set}$ is integrable,  then 
\begin{description}
\item[i)] $\Delta_{\set}$ is involutive; 
\item[ii)] the  state space $\R^N$ is partitioned into orbits of $\Delta_{\set}$; 
\item[iii)] the rank of the distribution  is constant along the orbits (of $\Delta_{\set}$, which coincide with the integral manifolds of such a distribution). 
\end{description}
The latter fact is a consequence of the fact that $\Delta_{\set}$ is invariant under the vectors in $\set$ together with the observation that the maps $e^{tV}$ are diffeomorphisms for every fixed $t \in \R$ (hence the Jacobian matrix $\jacobian{(e^{tV}x)}{x}$, which maps the tangent space at $x$ into the tangent space at $e^{tV}x$, is always invertible).

\begin{definition} [{\cite[page 185]{Sussman}}] \label{def:LFT}Let $\set$ be a set of everywhere defined, smooth vector fields on $\R^N$ and 
$\Delta_{\set}$ be the associated distribution. The set $\set$ (as well as the distribution $\Delta_{\set}$) is {\em locally of finite type} or {\em locally finitely generated} (LFG) if for every $x \in \R^N$ there exist vector fields $X_1 \dd X_k$ such that
\begin{description}
\item[i)]  $\mathrm{span}\{X_1(x) \dd X_k(x) \} = \Delta_{\set}(x)$ 
\item[ii)] for every $X \in \set$ there exists a neighbourhood $U$ of $x$ and $C^{\infty}$ functions $\varphi_{i,j}$ defined on $U$ such that
$$
[X,X_i]=\sum_{j=1}^k \varphi_{i,j}(x) X_j(x) \quad {\mbox{for all }} x \in U \mbox{and every } i \in \{1 \dd k\}.
$$
\end{description} 
\end{definition}
We emphasize that if $\Delta_{\set}$ is LFG then the rank of $\Delta_{\set}$ need not be constant. 
\begin{note}\label{Tom}
\textup{We recall  the following useful criterion (see \cite [Lemma 2.1.4]{Isidori}):  if a distribution $\Delta$ is either of constant rank or locally of finite type, then it is invariant under a vector field $V$ if and only if $[V, \tau] \in \Delta$ whenever $\tau \in \Delta$.\hf}
\end{note}
\begin{definition} \label{def:GFT} With the same notation and setting as in Definition \ref{def:LFT}, $\set$ is {\em globally of finite type} if point i) of Definition \ref{def:LFT} holds with $k$ independent of $x$ and if for every $X\in \set$ there exist $C^{\infty}$ functions $\varphi_{i,j}$ defined on $\R^N$ such that
$$
[X,X_i]=\sum_{j=1}^k \varphi_{i,j}(x) X_j(x) \quad {\mbox{for all }} x \in \R^N \mbox{and every } i \in \{1 \dd k\}.
$$
\end{definition}
The next theorem gives a sufficient condition for integrability, which is easy to check in practice.
\begin{theorem}[Hermann, Lobry, Stephan and Sussman]\label{thm:Hermann}
If $\set$ is locally of finite type then $\Delta_{\set}$ is integrable; in particular,  the integral manifolds of $\Delta_{\set}$ coincide with the orbits of the vector fields of the set  $\set$. 
\end{theorem}
\begin{note}[Comments on Theorem \ref{thm:Hermann}]\label{note:chow}
\textup{ Seen from a control-theoretical point of view, the above statement gives a global decomposition of the state space $\R^N$ into sets reachable by piecewise integral curves of vector fields in $\set$.  To clarify this fact and provide some context,  it is useful to compare it  with the case where the HC  holds. Start by noting that under the HC the Lie algebra generated by the vectors in $\set$ is required to have constant rank (and the rank is assumed to be precisely $N$ at every point).   The control-theoretical meaning of the HC is expressed by  Chow's Theorem, see \cite{Bellaiche, BLU},  (and indeed in control theory the HC is known as {\em Chow's condition}).  Chow's theorem states that if the vectors $\{V_0 \dd V_d \}$ satisfy ({\bf HC}) then any two points in $\R^N$ are {\em accessible} or {\em reachable} in finite time from each other along integral curves of the vectors in $\set$. That is, given any two points $x,y \in\R^N$, there exists a piecewise integral curve $\gamma$ of vectors in $\set$, and a time $t>0$ such that $\gamma(0)=x$ and $\gamma(t)=y$. 
This is not the case if we simply assume that $\set$ is a LFG set of vector fields.  According to the above theorem, if $\set$ is LFG then,  for every $x \in \R^N$, the set of states { reachable from} $x$  in finite time coincides with the maximal integral manifold of $\Delta_{\set}$ through $x$. Because the rank of the distribution is not constant, and in particular it needs not be $N$ at any point, this implies that, in general, the orbits of $\Delta_{\set}$ will be proper subsets of $\R^N$ (as we have mentioned, they form a partition of $\R^N$).  
} \hf
\end{note}
We conclude this subsection by recalling the following result, which will be used later on. 
\begin{lemma}[{\cite[Theorem 2.1.9]{Isidori}}] \label{thm:intcurvepreservemanifolds}
Let $\Delta_{\set}$ be a smooth involutive distribution invariant under a vector field $W$. Suppose  $\Delta_{\set}$ is locally finitely generated. Let $x_1,x_2$ be two points belonging to the same maximal integral manifold of $\Delta_{\set}$. Then, for all $t \in \R$, the points 
$e^{tW}x_1$ and $e^{tW}x_2$ belong to the same maximal integral submanifold of $\Delta_{\set}$.
\end{lemma}
To clarify the  above statement: under the asumptions of the lemma, if $x_1, x_2$ belong to a given MIM of $\Delta_{\set}$, say $\mathcal{M}$ then $e^{tW}x_1, e^{tW}x_2 \in \tilde{\mathcal{M}}$, where $\tilde{\mathcal{M}}$ denotes another generic MIM of $\Delta_{\set}$. In general $\tilde{\mathcal{M}}$ will be different from ${\mathcal{M}}$ (unless $W$ belongs to $\Delta_{\set}$). For example see Example \ref{example:UFGHeisenberg}.

\subsection{Assumptions} 
Throughout the paper we will make the following standing assumptions. 
\begin{hypothesis}\label{SA} Standing assumptions:
\begin{enumerate}[label=\textbf{\textup{[SA.\arabic*]}}]
\item All the vector fields we consider in this paper are smooth, everywhere defined and globally Lipschitz.  \label{SA1}
\item In this paper we will consider partitions of $\R^N$ into submanifolds; each one of such submanifolds is generically denoted by $\curlys$, see  definition after Proposition \ref{intmanpro}. Throughout, the manifold topology $\tau$ (on $\curlys$) is assumed to be the Euclidean topology of $\curlys$, seen as a subset of $\R^N$; that is,  the open sets of $\curlys$ in the manifold topology $\tau$ are sets of the form $O \cap \curlys$, where $O$ is a Euclidean open set of $\R^N$. In   Appendix \ref{app:topology} we motivate the choice of such a topology and give further details about this assumption. 
\label{SAtop}
\item When we say that the obtuse angle condition \eqref{eq:OAC} (or its second order version \eqref{eq:OAC2}, respectvely) is satisfied, we mean that it is satisfied for some $\lambda_0>0$ large enough so that the estimate \eqref{eq:graddecayest} (\eqref{seconorderdecayeq}, respectively) follows.  \label{SAOAC}
\end{enumerate}
\end{hypothesis}

\begin{note}
\textup{
Assumption \ref{SA1} will be needed mostly to make sure that all the integral curves of the involved vector fields are well defined (and to guarantee well-posedness of the SDE \eqref{SDE}). However see Note \ref{technicalpoint} on this point. \hf
}
\end{note}

\section{Geometrical significance of the  UFG condition and implications for the corresponding SDE} \label{sec:geomofUFGprocesses}
In this section we come to explain how the general results outlined in Section \ref{subsec:geometry} applies to the study of the dynamics \eqref{SDE}, assuming that the vector fields $V_0 \dd V_d$ satisfy the UFG condition. For clarity, we will compare with the case in which  $V_0 \dd V_d$ satisfy the H\"ormander condition. Subsection 
\ref{Sec4:global} contains {\em global} results, Subsection \ref{sec:Coordinatechange} is focussed on {\em local} results.

\subsection{Global Results}\label{Sec4:global}
 Recalling the notation and nomenclature of Section \ref{subsec:UFG} and motivated by Theorem \ref{thm:sussorbit}, we introduce  two distributions associated with the vector fields $V_0 \dd V_d$; such distributions will play a fundamental role in the analysis of the UFG-dynamics \eqref{SDE}. Let 
\begin{itemize}
\item  $\delnn$ be the smallest distribution which contains the space $\mathrm{span}\{V_0,V_1\dd V_d\}$ and is invariant under the vector fields 
$\{ V_0, V_1 \dd V_d\}$; 
\item $\deln$ be the smallest distribution which contains the space $\mathrm{span}\{V_1\dd V_d\}$ and is invariant under the vector fields $\{V_0,V_1\dd V_d\}$. 
\end{itemize}
Let us denote by $n=n(x)$ the rank of the distribution $\deln(x)$. Notice that $n=n(x)$ is a function of the point $x\in\R^N$ and, as such its value can vary from point to point. As Lemma \ref{lemmaW0} below  demonstrates, if at some point $x \in \R^N$ the rank of $\deln$ is $n$, then the rank of $\delnn$ is {\em at most} $n+1$, hence the index $n+1$  in the notation for  $\delnn$. We will typically assume that $n<N$,  where $N$ is the dimension of the state space $\R^N$ in which the vector fields $V_0 \dd V_d$ live, see  Note \ref{Noten<N} on this point. 
We stress that $\deln$ may not contain the vector field $V_0$ itself (unless for example $V_0$ is a linear combination of $V_1 \dd V_d$).  Lemma \ref{lemmaW0} below gives a simpler equivalent description of the distributions $\deln$ and $\delnn$  (which is the one we gave in the introduction). 
\begin{lemma}\label{lemmaW0}
Let $V_0 \dd V_d$ be $d+1$ vector fields on $\R^N$ which satisfy the UFG condition for some $m$, see Note \ref{note:higherorderUFG}.  
Recall the decomposition \eqref{defvoperp},  the definition  of $\fRm$, given in \eqref{(R)},   and set $\fRmo:= \fRm \cup V_0$. Then 
	\begin{align}
\deln=\mathrm{span}\{ \fRm \} \quad \mbox{and} \quad 
	 \delnn =\mathrm{span}\{\fRmo\} \,.
	\end{align}
	In particular,
	$$
	\delnn(x)=\mathrm{span}(\deln(x),\voperp(x)).
	$$ 
\end{lemma}
\begin{proof}[Proof of Lemma \ref{lemmaW0}] A proof of this lemma in a general setting can be found in \cite[Lemma 1.8.7 and Remark 2.2.3]{Isidori}. For completeness (and to spare the reader from having to compare and match notations and setting with those in \cite{Isidori}), we have added a proof  in Appendix \ref{app:sec45}.
\end{proof}

\begin{note}\label{ufginvariant}
\textup{If the vector fields $V_0, V_1 \dd V_d$ satisfy the UFG condition then the distributions $\deln$ and $\delnn$  are locally of finite type because they are {\em globally} of finite type. This can be checked by using Note \ref{note:higherorderUFG} (and the fact that nested commutators can always be expressed as linear combinations of hierarchical commutators, see \cite[page 11-12]{BLU}).
} 
\hf
\end{note}

\begin{example}\label{silly}\textup{
In $\R^2$ let $V_0=1_{A}\pa_x$ and $V_1=1_{A^c}\pa_x+1_{A^c}\pa_y$ (strictly speaking here the coefficients are not smooth), where $A$ is the set $A=\{(x,y) \in \R : x\in [-1,1]\}$. Then $\deln(x,y)=\delnn(x,y)$ and they are both two-dimensional for every $(x,y) \in A^c$ while $\deln(x,y) \neq \delnn(x,y)$ for every $(x,y) \in A$, as on this set $\delnn$ is one dimensional while $\deln (x,y)=0$ for every $(x,y) \in A$. 
\hf
}
\end{example}

Since the UFG condition implies that the sets $\fRm$ and $\fRmo$ are locally of finite type,  we can apply  Theorem \ref{thm:Hermann} to the distributions given by the span of $\fRm$ and $\fRmo$. By Lemma \ref{lemmaW0},   the distributions $\deln$ and $\delnn$ coincide with span of $\fRm$ and $\fRmo$ respectively. As a corollary, we have the following proposition.  
\begin{prop}\label{intmanpro}
If the vector fields $V_0, V_1 \dd V_d$ satisfy the UFG condition, then both $\delnn$ and $\deln$ enjoy the integral manifolds property. In particular the integral manifolds of $\delnn$ coincide with the orbits of $\delnn$ (and the same holds for the distribution $\deln$).  
\end{prop}

We denote by $\capitals$ ($\curlys$, respectively) a generic MIM of the distribution $\deln$ ($\delnn$, respectively). Consistently, $\capitals_x$ ($\curlys_x$, respectively) will denote the MIM of $\deln$ ($\delnn$, respectively) through the point $x \in \R^N$. It is easy to see that for every $x\in \R^N$,  $\capitals_x \subseteq\curlys_x$, so that  $\curlys_x$ is a disjoint union of integral manifolds of $\deln$. Notice that $n=n(x)$ is constant along the orbits $S$ of $\deln$.  

It is important to observe that any deterministic dynamics started on a maximal integral manifold $\curlys$ of $\delnn$ and following the integral curves of the  fields $V_0 \dd V_d$, will remain in $\curlys$ for any positive time (see Note \ref{note:chow}). On the other hand, if $X_0=x$ is the initial datum of the  stochastic  dynamics $\eqref{SDE}$ and $X_0 \in \curlys_x$,  then $X_t \in \overline{\curlys}_x$ for all $t \geq 0$. This is a consequence of the Stroock and Varadhan support theorem, which we recall below,  see \cite{Bellet} for more details.

\begin{theorem}[Stroock and Varadhan]\label{thm:stroocksuppthm}
Let $X_t^{(x)}$ be the solution of the SDE \eqref{SDE}.  The support\footnote{The support of the law of a random variable $X$ denotes the smallest closed set $A$ such that $\mathbb{P}(X\in A)=1$. } of the law of $\{X_t^{(x)}\}_{t\in [0,T]}$ in path space, coincides with the closure in $(C([0,T];\mathbb{R}),\lVert\,\cdot\,\rVert_\infty)$ of the set of paths $(p_t)_{t\in[0,T]}$ such that $(p_t)$ satisfies the control problem:
	\be\label{stoccontrprob}
	dp_t=V_0(p_t)dt+\sqrt{2}\sum_{i=1}^dV_i(p_t)\psi_i(t)dt, \quad p_0=x \, , 
	\ee
	for some $\psi_1,\dd\psi_d:[0,T]\to \mathbb{R}$ piecewise constant functions.
\end{theorem}
Informally, Theorem \ref{thm:stroocksuppthm} says that  the stochastic dynamics \eqref{SDE} will access in  time $t$ the (closure) of the set  reachable in time $t$ by the control problem \eqref{stoccontrprob}, as we vary the {\em controls } $\psi_1 \dd \psi_d$ in a suitable set of functions. 
\begin{excursus}\textup{
We would like to further elaborate on the comment started before Theorem \ref{thm:stroocksuppthm}. To this end, consider the following one-dimensional SDE:
\be\label{parex}
dX_t= -\sin(X_t) dt + \cos(X_t) \circ dB_t \,. \footnote{This is a known example, see for example \cite{Hairer}.}
\ee
Here $V_0=-\sin(x) \pa_x , V_1 = \cos(x) \pa_x$, and $[V_0, V_1]=\pa_x$, so that these fields satisfy both the HC and the PHC. According to Chow's theorem (see Note \ref{note:chow}),  if $V_0, V_1$ satisfy the HC then any two points in $\R$ can be joined through integral curves of such fields.  However, if we start the dynamics \eqref{parex} at $x \in [-\pi/2, \pi/2]$ then the solution $X_t$ never leaves the interval $[-\pi/2, \pi/2]$. This is not in contradiction to the statement of Chow's theorem. The behaviour of the stochastic dynamics \eqref{parex} is related to the control problem \eqref{stoccontrprob}. On the other hand, when we say that under the HC any two points in $\R^N$ can be joined by integral curves of vectors in $\set$, this   is equivalent  to saying that the set of points reachable from $x$ by the control system
\be\label{dcs}
dp_t=V_0(p_t) \psi_0(t) dt+\sum_{i=1}^dV_i(p_t)\psi_i(t)dt, \quad p_0=x,
\ee
is indeed the whole space $\R^N$ (in the above the functions $\psi_1,\dd\psi_d:[0,T]\to \mathbb{R}$ are say piecewise constant {controls}).
  Clearly,  the set of points accessible by  \eqref{stoccontrprob} is a subset of the set of points accessible by \eqref{dcs}. 
In our example,  the support of the law of the solution to SDE \eqref{parex} is given by the (closure of the) set of points reachable by the control problem
\begin{equation*}\label{cp1}
dX_t= -\sin(X_t) dt + \cos(X_t) \psi_1(t) dt.  
\end{equation*}
On the other hand, Chow's theorem applied to the vector fields $V_0, V_1$ refers to the problem 
\begin{equation*}\label{cp2}
dX_t= -\sin(X_t) \psi_0(t) dt + \cos(X_t) \psi_1(t) dt.  
\end{equation*}
Such a dynamics can indeed be stirred to access the whole real line, no matter where it is started. 
} \hf
\end{excursus}
The theory summarised in Subsection \ref{subsec:geometry} describes completely the sets accessible by the control problem \eqref{dcs}, which are precisely the orbits of the vector fields $V_0 \dd V_d$.  On the other hand, if we want to study the SDE \eqref{SDE} (under the UFG condition) then we are interested in understanding the behaviour of the control problem \eqref{stoccontrprob}. Unfortunately, in full generality, one can only state the  following (see \cite[Section 2.2]{Isidori}). 
\begin{lemma}\label{lemmareachstoc}
With the notation and nomenclature introduced so far,  let $V_0 \dd V_d$ be  smooth vector fields on $\R^N$ satisfying the UFG condition. Then the sets of points reachable from $x$ by the control problem \eqref{stoccontrprob} is a subset of $\curlys_x$ and it contains at least a non-empty open subset of $\curlys_x$. 
\end{lemma}
Combining the above and Theorem \ref{thm:stroocksuppthm} we obtain the following. 
\begin{prop}\label{correachstoc}
Consider the SDE \eqref{SDE} with initial datum $X_0=x$ and assume that the vector fields $V_0 \dd V_d$ satisfy the UFG condition.  Then $X_t \in \overline{\curlys_x}$ for every $t \geq 0$. \footnote{We clarify again that  the closure is intended to be in the Euclidean topology.}
\end{prop}
Let us reiterate that Proposition \ref{correachstoc} doesn't say that $X_t^{(x)}$ will explore the whole set $\overline{\curlys_x}$ (that is, it doesn't imply irreducibility of the process on $\overline{\curlys_x}$), it simply means that the process $X_t$ will not leave such a set. 
\subsection{Local considerations: an important change of coordinates}\label{sec:Coordinatechange}
Let  $x \in \R^N$ be a {\em regular point} of  a given  distribution $\Delta$, i.e. suppose there exists a  neighbourhood  of $x$ where the dimension of $\Delta$ is constant, say equal to $n$. If this is the case then,  locally, there exist $n$    linearly independent vector fields, $\{X_1 \dd X_n\}= \set_n$, generating the distribution. 
   Suppose furthermore that $\Delta_{\set_n}$ is involutive and $n<N$ (see Note \ref{Noten<N}) .   For some small enough $\epsilon>0$ we can define the   map $\Psi: (-\epsilon, \epsilon)^N \rightarrow \R^N$  as follows:
\begin{align*}
\Psi :    \, (-\epsilon, \epsilon)^N \quad &  \longrightarrow \quad  \R^N\\
 \mathbf{t}:= (t_1 \dd t_N)   \quad & \longrightarrow \quad  e^{t_1 X_1}e^{t_2 X_2}\, {\cdots} \, e^{t_N X_N} x \, ,
\end{align*}
where $X_1 \dd X_n$ are as above and $X_{n+1} \dd X_N$ are such that 
${\mathrm{span}}\{X_1 \dd X_n,X_{n+1} \dd X_N \} = \R^N$ (at least locally). 
The map $\Psi$ is, at least locally, a diffeomorphism on its image, so it admits an inverse, which we denote by $\Phi$. Differentiating the obvious identity
$(\Phi \circ \Psi)(\mathbf{t}) = \mathbf{t}$, one obtains
$$
(\jacobian{\Phi}{x})(\Psi(\mathbf{t})) \cdot (\jacobian{\Psi}{t}) (\mathbf{t}) = Id_{N \times N}. 
$$
Let us make the above notation more explicit. The map $\Phi$ is a map from (opens sets of) $\R^N$ to (opens sets of) $\R^N$, i.e. 
\begin{equation*}\label{ellell}
\Phi(x)= (\Phi^1(x) \dd \Phi^N(x)), \quad  x \in \R^N \, ,  
\end{equation*}
where   $\Phi^i: \R^N \rightarrow \R$ .
Therefore the $i$-th row of the matrix $\jacobian{\Phi}{x}$ is the gradient $\nabla \Phi^i$. On the other hand, the $j$-th column of the matrix $\jacobian{\Psi}{t}$ is the  vector 
$\frac{\pa \Psi}{\pa t_j}:= \{\frac{\pa \Psi^1}{\pa t_j} \dd \frac{\pa \Psi^N}{\pa t_j}\}^T$.
The first $n$ columns of the Jacobian matrix $(\jacobian{\Psi}{t}) (\mathbf{t})$ are linearly independent (because $\Psi$ is a diffeomorphism) and, from the above, we have
\be\label{ell}
\nabla \Phi^i \cdot \frac{\pa \Psi}{\pa t_j} = 0 \quad  \mbox{ for all } j=1 \dd n, i= n+1 \dd N.
\ee
By the involutivity of $\Delta_{\set_n}$  the vectors  $\{\frac{\pa \Psi}{\pa t_j}\}_{j=1}^n$   belong to  $\Delta_{\set_n}$;\footnote{See e.g. \cite[item (ii) on page 25]{Isidori}} moreover because they are linearly independent, they span $\Delta_{\set_n}$.  Therefore the vectors $\nabla \Phi^i$ are orthogonal to every vector of $\Delta_{\set_n}$, i.e. 
$$
\nabla \Phi^i \cdot \tau =0 \quad \mbox{for every } \tau \in \Delta_{\set_n} \mbox{ and for every } i=n+1 \dd N \,. 
$$
Now notice that $\Phi$ is (locally) invertible so it can be used as a (local) change of coordinates $\bfz=\Phi (x)$.  With these preliminaries in place, we have the following. 
\begin{prop}\label{changecoordgen}
Let  $\Delta$  be a smooth involutive distribution on $\R^N$ and $x_0$ a regular point of $\Delta$.  In particular, assume that there exists a neighbourhood of $x_0$ where the dimension of $\Delta$ is $n$. Then there exists a change of coordinates $\Phi$ (defined locally) such that
\begin{description}
\item[i)]  A vector field $V$ on $\R^N$ belongs to  $\Delta$ if and only if in the coordinates defined by $\Phi$, the last $N-n$ components of $V$ are zero; \footnote{If $V$ is any vector,  then the vector $\tilde{V}(z)= \left[ (\jacobian{\Phi}{\star} ) \cdot V(\star) \right] \vert_{\star = \Phi^{-1}(\bfz)}$ is the representation of $V$ after the change of coordinates $\Phi$. Indeed,  if $\gamma(t)=e^{tV}x$ and 
		$\tilde{\gamma}(t) = \Phi (\gamma(t))$ then the tangent vector to $\tilde{\gamma}$ is precisely $\tilde{V}$.}
\item[ii)] \label{item:characterisationundercoordinatechange}if $\Delta$ is invariant under a vector field $W$ then, in the coordinates defined by $\Phi$, the last $N-n$ components of $W$ are functions independent of the first $n$ coordinates. More explicitely, as per notation introduced in \eqref{coordinateblocks},  let 
$$\bfz=(z^1 \dd z^n, z^{n+1} \dd z^N)=(z^1 \dd z^n, \zeta, a)= \Phi(x^1 \dd x^N)$$ and let $\tilde{W}$ be the representation of $W$ in the new coordinates. Then 
$$
\tilde{W}(\bfz)=(\tilde{W}^1(\bfz) \dd \tilde{W}^n(\bfz), \tilde{W}^{n+1}(\zeta,a) \dd \tilde{W}^N(\zeta,a) \,.
$$
\end{description}

\end{prop}

\begin{proof}[Proof of Proposition \ref{changecoordgen}]
The proof is deferred to Appendix \ref{app:sec45}.
\end{proof}

We now  want to apply Proposition \ref{changecoordgen} to the vector fields appearing in the  SDE \eqref{SDE}. We assume that such vector fields  on $\R^N$ satisfy the UFG condition for some $m$. Let $\delnn$ and $\deln$ be the distributions defined at the beginning of Section \ref{sec:geomofUFGprocesses}.   We know that the rank of $\delnn$ is constant along the orbits of $\delnn$ (see comment before Definition \ref{def:LFT}). Let $x \in \R^N$ and consider the orbit of $\delnn$ through $x$. 
In view of Lemma \ref{lemmaW0},   if we assume that $\voperp(x) \neq 0$ then  the rank of $\delnn$ at $x$ is exactly $n+1$. Recall that $N$ is fixed and it is the dimension of the state space $\R^N$, while $n=n(x)$ is the dimension of the orbits of $\deln$ and it is constant along each one of such orbits. Notice that $\deln$ (and $\delnn$)  is  also involutive by construction, so we can use it to apply Proposition \ref{changecoordgen}.

With this in mind, let us   describe the coordinate change. This is obtained by combining the following two steps. 

$\bullet$ {\em Step one}: because $\delnn=\mathrm{span}(\mathcal{R}_m , V_0)$ is the tangent space of  an $(n+1)$-dimensional submanifold of $\R^N$ one can always {\em locally} express the vector fields $V_0 \dd V_d$ as
$$
\tilde{V}_j= (\tilde{V}_j^1 \dd \tilde{V}_j^{n+1}, 0 \dd 0), \quad j=0,1,\dd ,d, 
$$
i.e. the last $N-(n+1)$ coordinates of the vectors $\tilde{V_j}$ are simply zero. 

$\bullet$ {\em Step two}: apply Proposition \ref{changecoordgen} using the distribution $\deln$ (possibly only to the first $n+1$ coordinates of the involved fields). Then, because $V_1 \dd V_d$ belong to $\deln$ and $V_0$ is invariant for $\deln$, one obtains, in the new local coordinates, (and recalling the notation introduced in Section \ref{sec:notation})
\begin{align*}
& \tilde{V}_0 = (\tilde{V}_0^1(\bfz) \dd \tilde{V}_0^{n}(\bfz), \tilde{V}_0^{n+1}(\zeta, a), 0 \dd 0) \\
%\qquad x^{n+1:N}:= (x^{n+1} \dd x^N)\\
& \tilde{V}_j = (\tilde{V}_j^1(\bfz) \dd \tilde{V}_j^{n}(\bfz), 0 \dd 0), \quad 
j=1 \dd d \,,
\end{align*}
where we keep the same notation $\tilde{V}_j$ for  the new representation of the vector fields after this further change of coordinates.  This shows that, in the new coordinates, the vector fields $V_0 \dd V_d$ take the form \eqref{newvectors1} - \eqref{newvectors0}.

We now want to express the SDE \eqref{SDE} in the new local coordinates. If $X_t$ is the original process, $\bfzz_t$ is the process in the new coordinates. In particular
$$
\bfzz_t= (Z_t, \zeta_t, a_t), 
$$
where $Z_t \in \R^n$ contains the first $n$ coordinates of $\bfzz_t$, $\zeta_t$ is the $(n+1)$-th coordinate of the process and $a$ contains the remaining $N-(n+1)$ components (which do not change in time, see below). 
  Putting everything together and using the convention \eqref{newvectors1} - \eqref{newvectors0}, one obtains that, in the new coordinates, the SDE \eqref{SDE} with initial datum $\bfzz_0=(z_0, \zeta_0, a_0)$ is simply
\begin{align}
&{Z}_t= z_0+ \int_0^t U_0(Z_s, \zeta_s, a_0) \, ds
+ \sum_{j=1}^d \int_0^t {U}_j({Z}_s, \zeta_s,a_0) \circ dB^j_s 
\label{ashtree1}\\
&\zeta_t=\zeta_0+ \int_0^t {W}_0({\zeta}_s, a_0) \,  ds \label{ashtree2}\\
& a_t = a_0 \,. \label{ashtree3}
\end{align}
Notice that from the above one can also deduce that, in the new coordinates, $\tilde{V_0}^{(\deln)}=(U_0, 0 \dd 0)$ while $\tilde{V_0}^{(\perp)}=(0 \dd 0, W_0, 0 \dd 0)$. Assuming for the moment that at the initial point $x=X_0$ the dimension of $\delnn$ is exactly $n+1$, the fact that  the last $N-(n+1)$ components of the dynamics remain constant  reflects the fact that, at least for a short enough time, the solution of the SDE remains in the integral submanifold of $\delnn$ from which it started, coherently with Lemma \ref{lemmareachstoc} and  Proposition \ref{correachstoc}.

If  at the initial point the rank of $\delnn$ is exactly $N$, i.e. $n+1=N$, then one simply has 
\begin{align}
& Z_t=  z_0+ \int_0^t U_0(Z_s, \zeta_s) \, ds
+ \sum_{j=1}^d \int_0^t  U_j(Z_s, \zeta_s) \circ dB^j_s   \label{localrep1}\\
&\zeta_t=\zeta_0+ \int_0^t W_0 (\zeta_s) \,  ds \, ,  \label{localrep2} 
\end{align}
and this time $\tilde{V_0}^{(\deln)}=(U_0, 0 \dd 0)$ while $\tilde{V_0}^{(\perp)}=(0 \dd 0, W_0)$.
In this  simpler case it is clearer that we have locally reduced  the SDE \eqref{SDE} to an ODE component, $\zeta_t$ (which evolves independently of all the other components) and an $(N-1)$ - dimensional SDE.  We emphasize that, because the change of coordinates is local, such a representation will hold only for small enough $t$.

\begin{example}[UFG-Heisenberg]\label{example:UFGHeisenberg}\textup{ Consider the following dynamics in $\R^3$  
\begin{align*}
dX_t & = - X_t dt \\
dY_t & = - Y_t dt + \sqrt{2} dW^2_t \\
dZ_t & = - 2 Z_t dt - \sqrt{2} Y_t \circ dW^1_t + \sqrt{2} X_t \circ dW^2_t 
\end{align*}
Here $V_0=(-x,-y, -2z)$, $V_1= (0,0,-y)$, $V_2=(0,1,x)$. This example was introduced in \cite{CrisanOttobre} and named the {\em UFG-Heisenberg dynamics} (as it comes from a modification of the Heisenberg group).   This is already globally in the form ODE+SDE. The ODE for the first coordinate can be solved explicitly, giving $X_t=x_0 e^{-t}$. Therefore, if we start the dynamics at $(x_0,y_0,z_0)$ with $x_0>0$ ($x_0<0$, respectively), then the system evolves (at least for finite time) in the semispace with positive $x$-coordinates (negative, respectively).  If the initial datum is on the plane $(0,y_0,z_0)$ then the dynamics remains confined to such a plane for all subsequent times. This is coherent with the following: for the above set of vector fields, one has
$\delnn((x,y,z)) \simeq \R^3$ if $x>0$ or $x<0$ and $\delnn((x,y,z)) \simeq \R^2$ when $x=0$. 
The distribution $\delnn$ has three orbits, namely the sets
\begin{align*}
\curlys_+ &=\{(x,y,z)\in \R^3: x>0\}, 
\quad \curlys_-=\{(x,y,z)\in \R^3: x<0\},\\  
 \curlys_0 &=\{(x,y,z)\in \R^3: x=0\} \,. 
\end{align*}
As for the distribution $\deln$, this spans $\R^2$ at every point. Moreover, the orbit of $\deln$ through the point $(b,y,z)$ is  the plane $x=b$. For this reason,   when working on this example we will simply denote by $S_b$ the orbit through the point $(b,y,z)$. In particular, notice that $\curlys_0=S_0$. \hf
}
\end{example}

\begin{example}[Random Circles]\label{ex:circle}\textup{
Consider the SDE
 	\begin{align}
 	dX_t&=-Y_tdt + \sqrt{2} X_t\circ dB_t  \label{circle1}\\
 	dY_t&=X_t dt + \sqrt{2} Y_t\circ dB_t,   \label{circle2}
 	\end{align}
where $B_t$ is a one-dimensional Brownian Motion. 
 	This system satisfies neither the HC nor the PHC, however the UFG condition is satisfied at level $m=1$. Indeed we have
 	\begin{equation*}
 	{V}_0(x,y)=\left(\begin{array}{c}
 	-y\\
 	x
 	\end{array}\right), \,{V}_1(x,y)=\left(\begin{array}{c}
 	x\\
 	y
 	\end{array}\right) \quad{\mbox{and}} \quad [V_1,V_0]= 0 \,.
 	\end{equation*}
 	For every $(x,y) \in \R^2$,  $\deln (x,y)=\mbox{span}\{{V}_1 (x,y)\}$;   except for the origin, the orbits of $\deln$ are radial half-lines. That is,   $S_{(x,y)}=(0,0)$ if $(x, y)=(0,0)$ and $S_{(x,y)}=\{(sx,sy),  s>0\}$ otherwise.  Indeed,   $S_{(x,y)}$ coincides with the set of points accessible by the integral curves of ${V}_1$, which can be found explicitly:
 	\begin{equation*}
 	e^{tV_1}(x,y)=\left(\begin{array}{c}x\, e^t\\
 	y\,e^t
 	\end{array}\right), \quad t \in \R \,.
 	\end{equation*}
Moreover, $V_0$ is orthogonal to $V_1$, so $\vodel =0$ and $\voperp=V_0$; therefore    $\delnn(0,0)=\{(0,0)\}$, $\delnn(x,y)=\mathbb{R}^2$ outside the origin,  
 	 $\mathscr{S}_{(x,y)}=\mathbb{R}^2\setminus\{(0,0)\}$ if $(x, y)\neq (0,0)$ and  $\mathscr{S}_{(0,0)}=\{(0,0)\}$. In this example the local change of coordinates in the neighbourhood of $(1,0)$ is given by the diffeomorphism
$$
\Phi_{(1,0)}(x,y) = \left(\arctan\left(\frac{y}{x}\right), \frac{1}{2}\log(x^2+y^2)\right).
$$
After such a change of coordinates, the SDE \eqref{circle1} - \eqref{circle2} can be expressed, locally, as 
\begin{align}
& d\zeta_t = dt\label{anglechange} \\
& dZ_t =   \sqrt{2}dW_t\label{radiuschange}
\end{align}
 Let $C_t= (X_t, Y_t) \in \R^2$. In Figure \ref{figcircles} below we plot the evolution of $C_t$, i.e. the solution of \eqref{circle1} - \eqref{circle2}. From the plots it should be clear that $(\zeta_t, Z_t)$ are just the polar coordinates of the point $C_t$: $\zeta_t$  represents the angle, which evolves deterministically with a simple anticlockwise motion,  while $Z_t$ (or, to be more precise, $\exp(2Z_t)$)  is the radius, which changes randomly according to the SDE \eqref{radiuschange}. }
\end{example}
%%%%%%%%%%%%%%%%%%%%%%%%
%%%%%%%%%%%%%%%%%%%%%%%%%5
%%%%%%%%%%%%%%55UNCOMMENT
%\begin{comment}
\begin{figure}[h] 
  	\centering
  	\subfloat[A plot of $(X_t,Y_t)$ for $0\leq t \leq \pi/2$.]{
  		\includegraphics[width=65mm]{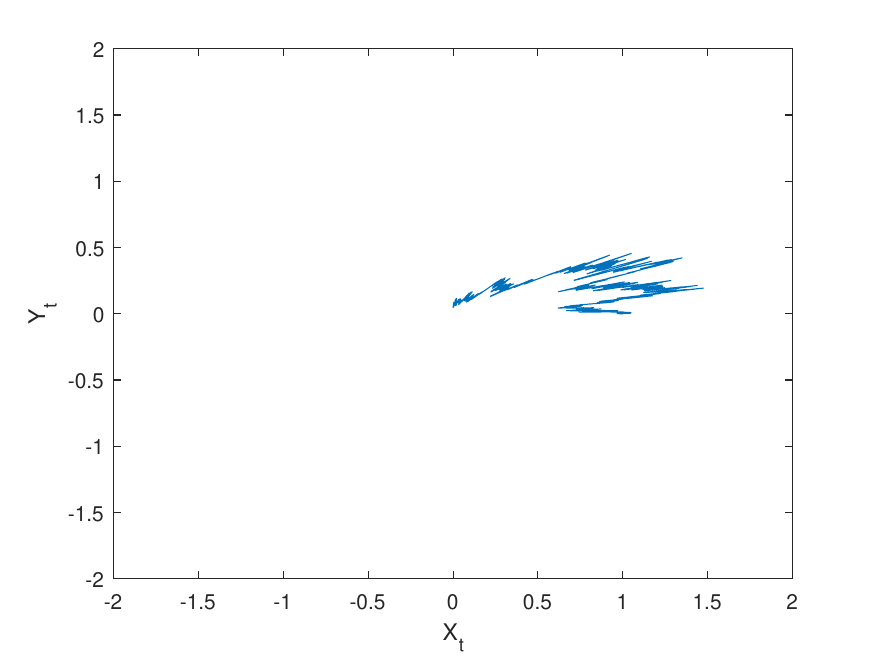}
  	}
  	\subfloat[A plot of $(X_t,Y_t)$ for $0\leq t\leq \pi$.]{
  		\includegraphics[width=65mm]{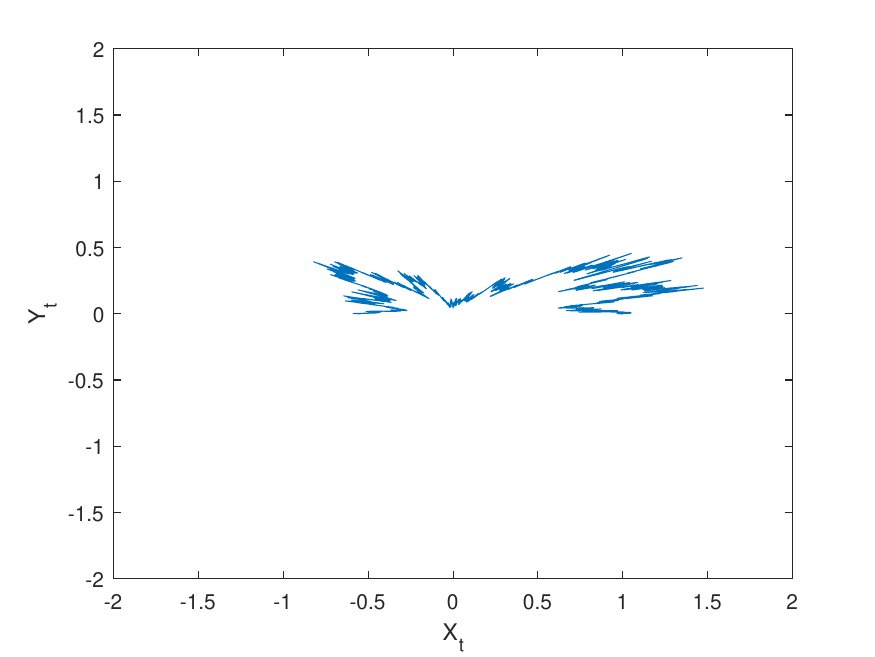}
  	}
  	\hspace{0mm}
  	\subfloat[A plot of $(X_t,Y_t)$ for $0\leq t \leq \frac{3\pi}{2}$]{
  		\includegraphics[width=65mm]{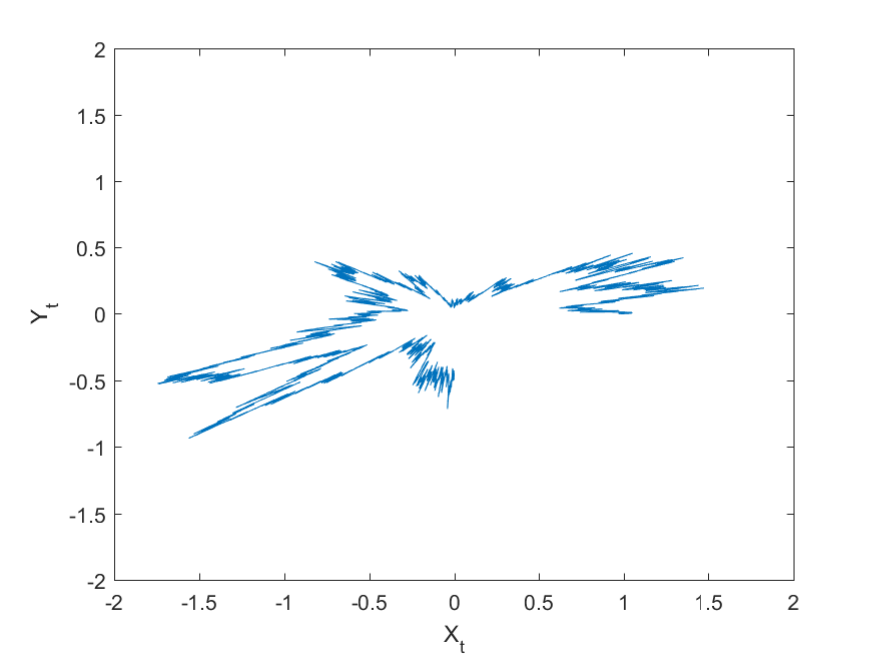}
  	}
  	\subfloat[A plot of $(X_t,Y_t)$ for $0\leq t \leq 2\pi$]{
  		\includegraphics[width=65mm]{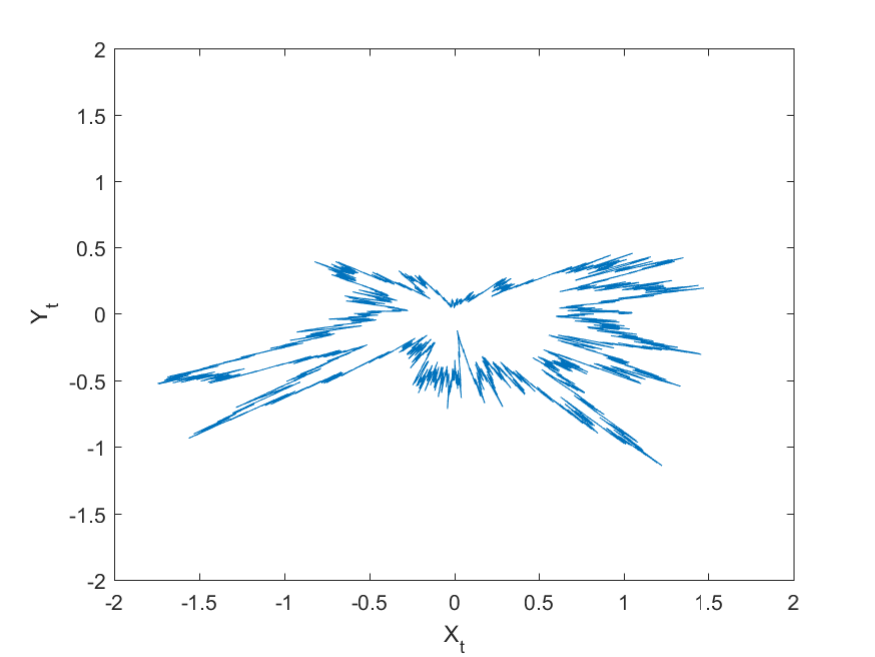}
  	}
  	\caption{The process $(X_t, Y_t)$ of Example \ref{ex:circle},  with initial condition $(X_0, Y_0)=(1,0)$. The angle of rotation evolves deterministically in counterclockwise sense, while the radius changes randomly, according to \eqref{radiuschange}.}\label{figcircles}
  \end{figure}
%\end{comment}
\begin{note}\label{Noten<N}
\textup{If the dimension $n$ of $\deln$ was equal to $N$ for every $x  \in \R^N$, this would imply that $\deln(x)=\delnn(x)$ for every $x \in \R^N$.  In particular,  the Parabolic H\"ormander Condition ({\bf PHC}) would hold.   This case is well studied in the literature and we do not wish to consider it here. For this reason many of the statements of this section are made under the assumption that $n<N$.  We need to emphasize that it may  happen that the two distributions coincide on a manifold (see Example \ref{example:UFGHeisenberg}, where the two distributions coincide on the plane $x=0$) and it may also happen that they both have full rank $N$ on a manifold, while they differ on other manifolds (see Example \ref{silly} below). The case  that is not interesting to our purposes is the one in which they coincide and have full rank on the whole of $\R^N$. Most of our theorems do cover that case as well (unless otherwise explicitly stated); but they are not really conceived in that framework. \hf}
\end{note}

\begin{note}\label{noteonsmoothness}\textup{The change of coordinates illustrated in this section will be an important technical tool throughout. We would like to point out how such a change of coordinates gives a different (and complementary) perspective on the smoothness results of Kusuoka and Stroock and of Crisan et all \cite{{KusStr82},{KusStr85},{KusStr87}, Kus03, {CrisanDelarue}} that we mentioned in the Introduction. As recalled in  Section \ref{context},  in these works the authors show   that if $f$ is a continuous and bounded function then,  under the UFG condition, the function $(\cP_tf)(x)$ is not necessarily smooth in every direction (as it would be the case under the H\"ormander condition), but it is in general only smooth in the directions $V_{[\alpha]}$, $\alpha \in \A_m$. In particular, it may not be differentiable in the direction $V_0$. In view of the decomposition \eqref{defvoperp} and of the change of coordinates presented in this section, this result is quite intuitive, as we explain. By \eqref{defvoperp}, it is clear that if $\voperp=0$ then $(\cP_t f)(x)$ is differentiable in the direction $V_0$ (as in this case $V_0$ is a combination of the vectors in $\fRm$) and, as a consequence, it is differentiable in $t$ as well. The loss of smoothness happens if and only if $\voperp\neq 0$.  For simplicity (and without any loss of generality), let us restrict to a manifold where $n+1=N$, so that the local change of coordinates gives \eqref{localrep1}-\eqref{localrep2}. As already observed,  the representation of $\voperp$ in the new coordinates is given by $\tilde{V_0}^{(\perp)}= (0 \dd 0, W_0)$, where $W_0$ is the function driving the ODE component. Hence $\voperp$ is inherently linked to the deterministic part of the system, which clearly doesn't provide any smoothness. This also explains why, while there is no smoothness in the direction $V_0$, the semigroup will always be smooth  in the direction $\pa_t - V_0$ (to be more precise, in the direction $\pa_t-\voperp$), as solutions of the ODE are constant in this direction.   Finally, the deterministic part of the dynamics is responsible for the lack of density (i.e. for the fact that the law of the process does not admit a density on $\R^N$).
It is useful to the purposes of this discussion to point out that the one-dimensional transport equation is an extreme example of UFG condition; that is, consider the PDE $\pa_t u(t,x) = \pa_x u(t,x), (t,x) \in \R_+ \times \R^N$, with initial datum $u(0,x)=f(x)$. Here $V_0= \pa_x$.  As is well known, the solution to such a PDE is just $u(t,x)=f(x+t)$, hence no smoothing  occurs in the space direction. However the solution is smooth in the direction $(\pa_t - \pa_x)= \pa_t - V_0$, as it is constant in such a  direction. Therefore, UFG diffusions include a vast range of behaviours, from smooth elliptic diffusions to deterministic equations. 
\hf
}
\end{note}

\begin{note}\label{technicalpoint}
\textup{
A final note on a technical point:
 as we have emphasized, to avoid having problems with the  well-posedness of the integral curves, we work under the standing assumption \ref{SA1}. After the change of coordinates the coefficients of the vector fields (in the new coordinates) may grow more than linearly, but they will still be smooth. Hence, in the neighbourhood in which they are defined, the vector fields will still be locally Lipschitz. The situation is more delicate with the vector $\voperp$: if $V_0$ is smooth, this is not the case for $\voperp$ as well, see Example \ref{ex:ODEwithmanyinvmeasctd}. 
Whenever this may cause issues, we will assume that $\voperp$ is at least such that the integral curve of $\voperp$ through a given point is unique and well defined (at least on given manifolds). \hf
}
\end{note}
We conclude this section by stating a couple of technical lemmata which will be useful in the following. 
\begin{lemma}\label{lem:boundary}
Assume the vector fields $V_0,\ldots, V_d$ satisfy the UFG condition. Let $\curlys$ be a maximal integral manifold of $\delnn$ and 
	 $S$ be an integral submanifold of $\deln$ such that $S \subseteq \curlys$.  Then $\pa S:= \bar{S}\setminus S$ is contained within 
$\pa \curlys:=\bar{\curlys} \setminus \curlys$.\footnote{Closures are meant in the Euclidean topology, see Appendix \ref{app:topology}. }
\end{lemma}

\begin{proof}[Proof of Lemma \ref{lem:boundary}]
The proof is deferred to Appendix \ref{app:sec45}. 
\end{proof}

The statement of Lemma \ref{lem:boundary} would clearly not be true if $S$ and $\curlys$ were two arbitrary sets, it only holds because of the particular structure of the  integral manifolds of $\deln$ and $\delnn$.  
As a side remark, notice that while $S \subseteq \curlys$ implies $\pa S \subseteq \pa \curlys$, it is not the case, in general,  that the boundary of $\curlys$ is the union of boundaries of orbits of $\deln$, see Example \ref{example:UFGHeisenberg}. 

\begin{lemma}\label{lemmavoonmanifold}
With the notation introduced so far, assume the vector fields $V_0 \dd V_d$ satisfy the UFG condition.  Let $x_0 \in\R^N$ and recall that $x_0$ belongs to  exactly one integral manifold of $\delnn$, the manifold  $\curlys_{x_0}$. Consider the vector field $\voperp$ (defined in \eqref{defvoperp}) and assume such a vector field is smooth. Then either    $\voperp(y)=0$ for every $y \in \curlys_{x_0}$ or  $\voperp(y) \neq 0$  for every $y \in \curlys_{x_0}$. 
\end{lemma}
\begin{proof}[Proof of Lemma \ref{lemmavoonmanifold}]
The proof is deferred to Appendix \ref{app:sec45}. 
\end{proof}

\section{Qualitative Results on UFG diffusions}\label{sec:5}
In this section we study the behaviour of the diffusion $X_t$  \eqref{SDE}  under the sole assumption that the vector fields $V_0 \dd V_d$ appearing in \eqref{SDE}  satisfy the UFG condition.  As observed also in \cite[Note 4.3]{CrisanOttobre}, under the sole UFG condition one cannot expect  to make any quantitative deductions on the behaviour of the process $X_t$. Neither can one expect the UFG condition itself  to imply any results about existstence or uniqueness of invariant measures, as there are many elliptic diffusions that don't have an invariant measure (the simplest example being Brownian motion on $\R$).  In order to study invariant measures and decay to equilibrium we will have to make further assumptions. Nonetheless, the geometric considerations made in the previous sections allow us to prove several qualitative statements on the behaviour of the diffusion.  The main results of this section are Proposition \ref{prop:dimcanonlydecrease}, Proposition \ref{prop:SDElivesonsubmanifold} and Proposition \ref{thm:meszerosetsunderinvmeas}. Collectively, these three results impart a lot of intuition about UFG dynamics and cointain a lot of useful information. After each one of these three statements we have inserted a note to comment on the meaning of these propositions, see Note \ref{Note5.2}, Note \ref{Note5.4} and Note \ref{Note5.8}.  The results of Section \ref{sec:non-autonomous} and Section \ref{sec:longtimebehaviour} heavily rely on the statements of this section.

Recall that we denote by $\capitals$ ($\curlys$, respectively) a generic integral manifold of the distribution $\deln$ ($\delnn$, respectively). Consistently, $\capitals_x$ ($\curlys_x$, respectively) denote the integral manifold of $\deln$ ($\delnn$, respectively) through the point $x \in \R^N$. 

\begin{prop}\label{prop:dimcanonlydecrease}
 Assume that the vector fields $V_0, V_1 \dd V_d$  satisfy the UFG condition and let $X_t$ be the solution of the SDE \eqref{SDE}.  Let $\curlys$ be a maximal integral manifold of $\delnn$ and let $\pa{\curlys}$ be the boundary of $\curlys$, i.e. $\pa \curlys:= \bar{\curlys}\setminus \curlys$. Then the following holds:
\begin{description}
\item[i)] If $\pa \curlys$ is not empty, it  is a union of integral submanifolds of $\delnn$;\label{item:boundaryisunionofmanifolds}
\item[ii)] If $X_0=x \in \pa\curlys$ then $X_t \in \pa\curlys \,\,\,\mbox{for all } t>0 \mbox{ (almost surely)} \,.  $ 
\end{description} 
\end{prop}
\begin{proof}[Proof of Proposition \ref{prop:dimcanonlydecrease}]
The proof is deferred to Appendix \ref{app:sec45}.
\end{proof}

\begin{note}\label{Note5.2}\textup{Let us explain the meaning and consequences of Proposition \ref{prop:dimcanonlydecrease}. Suppose we start the SDE \eqref{SDE} at $x \in \R^N$. Because the integral manifolds of $\delnn$ partition $\R^N$,  $x$ belongs to one of such integral manifolds, the one  which we denote by $\curlys_x$.  As a consequence of Proposition \ref{correachstoc} we know that the process will never leave the closure of $\curlys_x$; however, if it started in the interior,  it could in principle hit the boundary (which is a manifold  whose dimension is lower than the dimension of $\curlys_x$) and then come back to the interior. What we prove here is that this is not possible. Furthermore, because the boundary of $\curlys_x$ is itself a union of integral manifolds of $\delnn$, one could repeat the previous reasoning once the process enters the boundary (if this is the case). As a result of iterating this line of thought, we have that, along the path of $X_t^{(x)}$, the rank of the distribution $\delnn$ can only decrease (or stay the same). In other words, we have shown that for every $x \in \R^N$ and $t\leq u$, one has
$$
\mathrm{rank}(\delnn(X_u^{(x)}))  \leq \mathrm{rank}(\delnn(X_t^{(x)})). 
$$ }
\hf
\end{note}

Before stating the next result we recall that the vector $\voperp$ has been defined in \eqref{defvoperp}. We also recall our assumption (see Note \ref{technicalpoint})  that $e^{t\voperp}(x)$ is well defined for all $x\in\R^N$ and $t\geq 0$.
\begin{prop}\label{prop:SDElivesonsubmanifold}
Let $X_t$ be the solution of the SDE \eqref{SDE} with initial condition $x_0\in\mathbb{R}^N$. If the vector fields $V_0, V_1 \dd V_d$ appearing in \eqref{SDE} satisfy the UFG condition then   
\begin{equation*}
X_t^{(x_0)}\in \overline{S}_{e^{t\voperp}(x_0)} \, , \quad \mbox{almost surely}. 
\end{equation*}
We clarify that $\overline{S}_{e^{t\voperp}(x_0)}$ is the closure (in the Euclidean topology) of the integral manifold of $\deln$ through the point 
$e^{t\voperp}(x_0) \in \R^N$. 
\end{prop}
\begin{proof}[Proof of Proposition \ref{prop:SDElivesonsubmanifold}] 
If $\voperp(x_0)=0$ then the result follows immediately from Proposition \ref{correachstoc} and  Lemma \ref{lemmaW0}. Indeed, by Proposition \ref{correachstoc} we know that $X_t^{(x_0)} \in \bar{\curlys}_{x_0}$ and by Lemma \ref{lemmaW0} (and Lemma \ref{lemmavoonmanifold}) we have $\curlys_{x_0}=S_{x_0}$.   So we only need to treat the case $\voperp(x_0)\neq 0$. This will be done by considering the control problem associated with the SDE \eqref{SDE} and by using  Stroock and Varadhan Support Theorem. We postpone this part of the proof to Appendix \ref{app:sec45}.
\end{proof}
\begin{note}\label{Note5.4}\textup{ Proposition \ref{prop:SDElivesonsubmanifold} clarifies the pivotal role of the vector $\voperp$. To convey more intuition about the role of $\voperp$, let us assume that $\voperp(x) \neq 0$ for every $x$ in $\curlys_{x_0}$, $x_0$ being the starting point of the SDE \eqref{SDE}.  We already know by  Proposition \ref{correachstoc} that $X_t^{(x_0)}$  will not leave $\bar{\curlys}_{x_0}$, so that we can consider $\bar{\curlys}_{x_0}$ to be the state space of the dynamics.  As already observed before Proposition \ref{intmanpro},   every $x \in \curlys_{x_0}$, belongs to exactly one orbit $S$ of $\deln$ and, moreover,  the union of the manifolds  $\{S_x\}_{x \in \curlys_{x_0}}$ gives precisely $\curlys_{x_0}$. In other words, the orbits of $\deln$ that belong to $\curlys_{x_0}$  partition $\curlys_{x_0}$. 
Furthermore, because $\voperp\neq0$ on $\curlys_{x_0}$ and the rank of $\delnn$ is constant on $\curlys_{x_0}$, one has (see Lemma \ref{lemmaW0}) that if $\curlys_{x_0}$ has rank $n+1$ then every orbit  $S_x, x \in \curlys_{x_0}$, will be a manifold of dimension $n$.   In particular, there is no $x \in \curlys_{x_0}$ such that $S_x=\curlys_{x}$ (so that the partition of $\curlys_{x_0}$ into orbits of the distribution $\deln$ is not the trivial one).  With this premise, it makes sense to ask the following question: while we know that the process will not leave $\overline{\curlys}_{x_0}$  for every $t\geq 0$, if we fix an arbitrary positive time $t>0$, can we tell more precisely where, within $\curlys_{x_0}$, $X_t^{(x_0)}$ is? In particular, can we  determine which submanifold $S$ it belongs to, i.e. which element of the partition of $\curlys_{x_0}$ is visited at time $t\geq 0$? The answer, given by Proposition \ref{prop:SDElivesonsubmanifold}, is the following: let $y=e^{t\voperp} x_0$. Then,  while $x_0\in S_{x_0}$, $X_t \in \overline{S}_y$. In other words, the vector $\voperp$ will make the SDE move from one submanifold of the partition (of $\curlys_{x_0}$) to another. Another question is whether it is  possible that $X_t$ will visit one of such submanifolds twice or whether it is the case that, once one of these submanifolds has been visited, it will never be hit again.  Example \ref{ex:circlescontinued} below shows that the submanifolds of the partition can be visited an arbitrary number of times.  }
\hf
\end{note}
\begin{example}\textup{
Recall the UFG-Heisenberg SDE introduced in Example \ref{example:UFGHeisenberg}. In this case $\voperp = (-x,0,0)$ and, as we have already mentioned,  $S_{(x_0,y_0,z_0)}$ is the plane $ S_{(x_0,y_0,z_0)} = \{(x,y,z): x=x_0\}.$
If $\voperp = (-x,0,0)$ then the integral curve of $\voperp$ through $(x_0, y_0, z_0)$ is $e^{t\voperp}(x_0, y_0, z_0)=(e^{-t}x_0, y_0, z_0)$ so that 
$$
S_{e^{t\voperp}(x_0, y_0, z_0)}= \{(x,y,z) \in \R^3: x= e^{-t}x_0\}. 
$$
It is therefore clear that if $(X_0,Y_0,Z_0)=(x_0,y_0,z_0)$ then $(X_t,Y_t,Z_t)= (x_0e^{-t}, Y_t, Z_t)\in \overline{S}_{(e^{-t}x_0,y_0,z_0)}$.
}
\end{example}
\begin{example}[Random Circles, Example \ref{ex:circle}, continued] \label{ex:circlescontinued}
\textup{Let us go back to Example \ref{ex:circle}. Consider the integral curve of $\voperp$, namely
 	\be\label{explrandcircvoperp}
 	e^{tV_0}(x,y) = e^{t\voperp}(x,y)= \left(\begin{array}{c}x\cos(t)-y\sin(t)\\
 	x\sin(t)+y\cos(t)
 	\end{array}\right).
 	\ee
 	To fix ideas, let $(x_0,y_0)=(1,0)$ be the initial condition of the SDE; then the integral curve of $\voperp$ through $(x_0,y_0)=(1,0)$ is the unit circle: 
$$
 e^{t\voperp}(1,0) = (\cos(t),\sin(t))
$$
and $S_{e^{t\voperp}\!\!(1,0)}= S_{(\cos(t), \sin(t))}$ is the (open half) radial line at an angle $t$ from the $x$-axis; that is, it is the (open half) radial line  that intersects the unit circle at the point $(\cos(t),\sin(t))$. 
On the other hand the solution of the SDE with initial datum $(x_0,y_0)=(1,0)$ is given by
 	\begin{align}
 	X_t^{(x_0,y_0)}&=\cos(t) e^{\sqrt{2} B_t}, \label{solcircles1}\\
 	Y_t^{(x_0,y_0)}&=\sin(t) e^{\sqrt{2} B_t}.\label{solcircles2}
 	\end{align}
 Therefore one can again explicitly verify that for every $t>0$,  $(X_t^{(x_0,y_0)}, Y_t^{(x_0,y_0)})$ belongs
  to $\overline{S}_{(\cos(t),\sin(t))}$. 
\hf
  }
\end{example}

\begin{prop}\label{thm:meszerosetsunderinvmeas}
With the notation introduced so far, assume the vector fields $V_0 \dd V_d$ satisfy the UFG condition and let 
$$E_t:=\{x\in\mathbb{R}^N:\mathbb{P}_x(X_t^{(x)}\notin \curlys_x )>0\} $$
and 
$$
E:=\{x\in\mathbb{R}^N:\mathbb{P}_x(X_t^{(x)}\notin \curlys_x)>0 \, \mbox{ for some } t>0\}.  
$$
Then, for any invariant measure $\mu$ of the SDE \eqref{SDE} (should at least one exist), we have
$\mu(E_t)=0$ for every $t>0$.  As a consequence,  
$\mu(E)=0$ as well. 
\end{prop}
\begin{proof}[Proof of Proposition \ref{thm:meszerosetsunderinvmeas}]
The proof is deferred to Appendix \ref{app:sec45}.
\end{proof}

\begin{note}\label{Note5.8}\textup{ Informally, 
Proposition \ref{thm:meszerosetsunderinvmeas} says that any invariant measure (should at least one exist) gives zero weight to the set of points that, under the action of the dynamics prescribed by the SDE \eqref{SDE},  leave in finite time the submanifold from which they start. That is,  the set of points $x$ such that $X_t^{(x)}\notin \curlys_x$ for some time $t>0$, has $\mu$-measure zero.  In view of Proposition \ref{prop:dimcanonlydecrease} this result is intuitive: in general,  if the dynamics leaves a set it can return infinitely many times to that set (when this happens the set is said to be  recurrent). Because along the trajectories of $X_t^{(x)}$ the rank of the distribution can only decrease, if the process $X_t^{(x)}$ leaves the  integral manifold $\curlys_x$ from which it started, it will  never return to it. The dynamics will therefore spend an infinite amount of time outside the manifold $\curlys_x$,  so that the invariant measure, if it exists, it can only be supported outside  such a manifold.  In other words,  the theorem says that  an integral submanifold $\curlys$ is a recurrent set if and only if the process never leaves it  (once it enters it). This argument constitutes  an informal proof of the theorem.  Notice also that this theorem doesn't say anything about say Geometric Brownian motion (see Example \ref{ex:GBM}) or the  UFG-Heisenberg process of Example \ref{example:UFGHeisenberg}, as such dynamics only leave the initial submanifold in infinite time; for any finite time  they stay in the submanifold from which they started.\hf}
\end{note}

\begin{lemma}\label{prop:pointwiselimit}
	Assume  the vector fields $V_0 \dd V_d$ appearing in \eqref{SDE} satisfy the UFG condition and that the Obtuse Angle Condition \eqref{eq:OAC} holds. Let $\cP_t$ be the semigroup defined in \eqref{semigroup}. Then, given a maximal integral submanifold $\capitals$ of $\deln$, we have
	\be\label{eq:pointwiselimit}
	\lim_{t\to\infty} \lv\cP_t f(x)-\cP_t f(y)\rv = 0, 
	\ee
 for all $f\in C_b(\mathbb{R}^N)$ and $x,y\in \capitals$.
\end{lemma}
\begin{proof}[Proof of Lemma \ref{prop:pointwiselimit}]
The proof  is deferred to Appendix \ref{app:sec45}
\end{proof}

\begin{prop}\label{invmesonstathyperpl}
	Consider the assumptions and setting of the previous lemma and 
	let $S$ be a maximal integral manifold of $\deln$.  Then, among all the invariant measures  $\mu$ of \eqref{SDE} (assuming at least one such measure exists), there exists at most one such that $\mu(S)=1$. Moreover, if such a measure exists,  then it is  ergodic (in the sense that $\cP_t\mathbbm{1}_E=\mathbbm{1}_E$ for some Borel set $E$,  implies that $\mu(E)=1$ or $0$) and for every $x\in S$ and $f\in C_b(\mathbb{R}^N)$ we have
	\begin{equation}\label{eq:convtoequil}
	\cP_t f(x) \to \int_{S} f(y) \mu(dy).
	\end{equation}
	
\end{prop}

\begin{proof}[Proof of Proposition \ref{invmesonstathyperpl}]
The proof is deferred to Appendix \ref{app:sec45}.
\end{proof}

\section{Long-time behaviour of UFG processes: the case of ``non-autonomous hypoelliptic diffusions"}\label{sec:non-autonomous}

In this section we set $N=n+1$  and study  stochastic dynamics in $\R^N = \R^{n+1}$ of the form 
\begin{align}
dZ_t & = U_0(Z_t, \zeta_t) dt +  \sum_{j=1}^d   U_j(Z_t, \zeta_t) \circ dB^j_t ,  \,  \label{SDEglobal}\\
d\zeta_t & = W_0 (\zeta_t)  dt   \label{ODEglobal}\\
Z_0 & =z, \quad \zeta_0=\zeta   \label{incondglobal}\,.
\end{align}
In other words, we consider systems for which the representation of the form ``ODE+SDE" \eqref{localrep1}- \eqref{localrep2} is global. \footnote{We are not claiming that this representation  necessarily results from the change of coordinates presented in Section \ref{sec:Coordinatechange}. } The above system consists of an $n-$dimensional  process, $Z_t \in \R^n$,  satisfying an SDE, equation \eqref{SDEglobal},  which is coupled with a one-dimensional autonomous ODE, \eqref{ODEglobal}. As in previous sections, $U_j: \R^n\times \R \rar \R^n, j \in \{0 \dd d\}$ and $W_0: \R \rightarrow \R$.  The evolution of $Z_t$ depends on the evolution of $\zeta_t$, but the ODE solution $\zeta_t$ evolves independently of the SDE. For the purposes of this paper, we don't think of $\zeta_t$ as  representing time, but rather as representing an additional space-coordinate. However notice that if $W_0 \equiv 1$ and $\zeta(0)=0$ then $\zeta_t=t$ and we recover a standard  time-inhomogeneous setting, i.e. in this case \eqref{SDEglobal} becomes a general  time-inhomogeneous SDE, namely 
\be\label{SDEnonhomog}
dZ_t  = U_0(Z_t, t) dt +  \sum_{j=1}^d   U_j(Z_t, t) \circ dB^j_t \,.
\ee

Going back to the representation of the form ``ODE+SDE'' \eqref{SDEglobal}-\eqref{ODEglobal} under consideration,  if we denote by $X_t$ the process $\R^N \ni X_t= (Z_t, \zeta_t)$,  then $X_t$ is the solution of an autonomous SDE.  The one-parameter semigroup associated to $X_t$  is, as usual, given by 
$$
(\cP_t f)(x): = \E f(X_t \vert X_0=x) = \E  \left[f (Z_t, \zeta_t)\vert (Z_0, \zeta_0)= (z, \zeta) \right], \quad x=(z, \zeta) \in \R^{n+1}, 
$$   
for any $f \in C_b(\R^{n+1}; \R)$. 
On the other hand one could consider the two-parameter semigroup associated with the non-autonomous process $Z_t$ alone. Indeed, if  we solve the ODE for $\zeta_t$ and substitute the solution back into the SDE for $Z_t$, then we can simply consider equation \eqref{SDEglobal} rather than the whole system.  To be more precise, let us denote by $\zeta_t^{\zeta}$ the solution at time $t$ of \eqref{ODEglobal} with initial datum $\zeta(0)=\zeta$. That is, $\zeta_t^{\zeta}=e^{tW_0}\zeta$.   Let also $Z_t^{s,z,\zeta}$ be the solution of the following SDE:
$$
Z_t=z+ \int_s^t U_0(Z_u, \zeta_u^{\zeta}) du +  \sum_{j=1}^d \int_s^t  U_j(Z_u, \zeta_u^{\zeta}) \circ dW^j_u \,.
$$
The two-parameter semigroup associated with the above non-autonomous SDE is given by
$$
(\cQ_{s,t}^{\zeta}g)(z)  := \E \left[ g(Z_t^{s,z,\zeta})\right],  \qquad  z \in \R^n, s\leq t \,,
$$
We emphasize that this two-parameter semigroup depends on $\zeta$, i.e. on the initial datum of the ODE. When we do not wish to stress this dependence we may just write $\cQ_{s,t}$. 
%{The notation for $\cQ_{s,t}$ should actually include the dependence on the initial %datum of the ODE,$\zeta$; i.e. we should  write $\cQ_{s,t}^{\zeta}$ inplace of %$\cQ_{s,t}$. We drop the superscript $\zeta$ to avoid further complicating the notation.}
 With this notation, one can equivalently rewrite  the definion of $\cP_t$ as 
\be\label{semigrglobalX}
(\cP_t f)(x)=\E \left[  f(Z_t^{0,z,\zeta}, \zeta_t^{\zeta})\right], \quad x=(z, \zeta) \in \R^{n+1}.
\ee
 To make explicit the relation between the two-parameter semigroup $\cQ_{s,t}^{\zeta}$ and the one-parameter semigroup $\cP_t$, fix $s \in \R$ and let $\hat{\zeta}=e^{sW_0}\zeta$. Notice that 
$$
Z_t^{0,z,\hat{\zeta}} = Z_{t+s}^{s, z, \zeta}
$$
where the equality is intended in law.
Therefore, for every $f \in C_b(\R^{n+1}; \R)$, and $z \in \R^n$, we have
\begin{align*}
(\cP_{t}f)(z, \hat{\zeta}) &  = \E \left[ f(Z_t^{0,z,\hat{\zeta}}, \zeta_t^{\hat{\zeta}})\right] 
 = \E \left[  f(Z_t^{0,z,\hat{\zeta}}, \zeta_{t+s}^{{\zeta}}) \right]\\
&= \E \left[  f(Z_{t+s}^{s, z, \zeta}, \zeta_{t+s}^{\zeta}) \right].
\end{align*}
Hence, 
\be\label{linksmanysemigr}
(\cP_{t}f)(z, \hat{\zeta}) = 
(\cQ_{s,s+t} f(\cdot, \zeta_{t+s}^{\zeta}) )(z) \,.
\ee
On  the right hand side of the above we mean to say that the semigroup $\cQ$ is acting on the function $f(\cdot, a)$ obtained by freezing the value of the last coordinate of the argument.  

From now on, unless otherwise specified, we write $Z_t$ for $Z_t^{0,z,\zeta_0}$. With this set up in place, we can start commenting on the long-time behaviour. Heuristically, if the solution of the ODE \eqref{ODEglobal} is unbounded, then one can't expect the process $X_t$ to have an invariant measure (see Proposition \ref{nounbounded})-- though the process $Z_t$ may still admit an invariant measure. So we restrict to the case in which the solution of the ODE is bounded. However, because \eqref{ODEglobal} is a one-dimensional  time-homogeneous ODE, if $\zeta_t$ is bounded then it can only either increase or decrease towards stable {\em stationary points} of the dynamics (a stationary point of the ODE \eqref{ODEglobal} is a point $\barz \in \R$ such that $W_0(\barz)=0$). We emphasise that there may be many such points.
For these reasons, we work under the assumption that $\zeta_t$ admits a finite limit, i.e. we assume that  the initial  datum $\zeta_0 \in \R$ is such that there exists a point $\barz= \barz(\zeta_0) \in \R$ such that 
\be\label{convODE}
\zeta_t^{\zeta_0} \rightarrow  \barz=\barz(\zeta_0) \quad \mbox{as } t\rightarrow \infty.  
\ee 
As customary, the notation $\barz=\barz(\zeta_0)$ is to emphasise the fact that the limit point will depend on the initial datum (when we don't wish to stress such a dependence we just denote a stationary point of the ODE by $\barz$). The dynamics \eqref{SDEglobal}-\eqref{ODEglobal} will, in general,  admit several invariant measures. As pointed out in the introduction, when this is the case, it is typically extremely difficult to determine the basin of attraction of each invariant measure. However in the setting of this section the basin of attraction of a given invariant measure will only depend on the behaviour of the ODE. (In the next section we will show that, despite the fact that the representation of the form ``ODE+SDE" is only local  for generic UFG processes, it is still the case that we can relate in a simple way the initial datum to the invariant measure to which the process is converging). Given an initial datum $\zeta_0$ for \eqref{ODEglobal}, let $\barz=\barz(\zeta_0)$ be the corresponding limit point of the ODE dynamics, as in \eqref{convODE}. Consider the SDE
\be\label{limitingproc}
d\barZ_t= U_0(\barZ_t, \barz)\, dt + \sum_{j=1}^d   U_j(\barZ_t, \barz) \circ dB^j_t ,  \quad \barZ_0=\bar{z},
\ee
with associated semigroup     
\begin{equation*}\label{limitingsemigr}
(\barQ_{t}g)(\bar{z}):= \E g(\barZ_t \vert \barZ_0=\bar{z}), \quad 
\bar{z} \in \R^n, g \in C_b(\R^n) \,. 
\end{equation*}
We will assume that the dynamics \eqref{limitingproc} is hypoelliptic, see Hypothesis 
\ref{item:smoothnessassumption} below for a more precise statement of assumptions. 
Moreover, under Hypothesis \ref{item:tightnessassumption},  the semigroup $\bar{\cQ}_t$ admits a unique invariant measure, $\mub=\mub(\barz, \zeta_0)$ (see Lemma \ref{lem:ergbarZ}). We emphasise that the asymptotic behaviour  of $\barZ_t$ is independent of the initial datum $\bar{z}$, see Lemma \ref{lem:ergbarZ}. 

In view of \eqref{convODE}, it is reasonable to guess that  the asymptotic behaviour of $Z_t= Z_t^{0,z,\zeta_0}$ is the same as the asymptotic behaviour of $\barZ_t$ which is the solution of \eqref{limitingproc}. This is the content of Theorem \ref{thm:mainglobalthm} below. Theorem \ref{thm:mainglobalthm} and  Theorem \ref{thm:mainglobalthmforX} are  the main results of this section; the former is concerned with the asymptotic behaviour of the semigroup $\cQ_{s,t}$, the latter describes the related asymptotic behaviour of the semigroup $\cP_{t}$. 
%Furthermore, Theorem \ref{thm:mainglobalthm} is stated with the perspective of %answering question i), Theorem \ref{thm:mainglobalthmforX} is phrased to  %address question ii). 
  We set first the assumptions used in the rest of this section and we comment on their significance in Note \ref{noteonhypglobal}. 

\begin{hypothesis}\label{hyp:nonautoassumptions}
With the notation introduced so far, we will consider the following assumptions:
\begin{enumerate}[label=\textbf{\textup{[H.\arabic*]}}]
\item The vector fields $V_0=(U_0, W_0),V_1=(U_1, 0),\dots,V_d=(U_d, 0)$ satisfy the UFG condition for some $m\geq 1$; moreover,  
$$span\{\mathcal{R}_m\}=\mbox{span} \{V_{[\alpha]}(x):\alpha\in \A_m\} \simeq \R^n, 
\quad \mbox{for every } x \in \R^{n+1}. $$	
\label{item:smoothnessassumption}
\item Define the measures $q_t^{s,z}$ by $q_t^{s,z}(A):=\cQ_{s,t}\mathbbm{1}_A(z)$,  for any Borel measurable $A \subseteq \mathbb{R}^{n}$. Then we require that for each $z\in\mathbb{R}^{n}$ the family of measures $\{q_t^{0,z}:t\geq 0\}$ on $\R^n$ is tight.\label{item:tightnessassumption}
\item The Obtuse angle condition \eqref{eq:OAC} is satisfied (with the understanding of Hypothesis \ref{SA} \ref{SAOAC}).
	\label{item:OAC}
\item The ODE \eqref{ODEglobal} has at least one stationary point $\barz$ and the initial datum   $\zeta_0 \in \R$ of  \eqref{ODEglobal} is such that \eqref{convODE} holds, for some limit point $\barz=\barz(\zeta_0)$. \label{item:convODE}
	\end{enumerate}
\end{hypothesis}

\begin{note}\label{noteonhypglobal}
\textup{
Some comments on the above assumptions, in particular on Hypothesis \ref{item:smoothnessassumption}. }
	\begin{itemize}
\item \textup{We start by remarking on the obvious fact that if  $X_t=(Z_t, \zeta_t)$, where $Z_t, \zeta_t$ are as in \eqref{SDEglobal}-\eqref{ODEglobal}, then $X_t$  solves an SDE of the form \eqref{SDE}, with $V_0=(U_0, W_0),V_1=(U_1, 0),\dots,V_d=(U_d, 0)$.
}

 \item\textup{With the notation of Section \ref{sec:geomofUFGprocesses} and Section \ref{sec:5}, assumption  \ref{item:smoothnessassumption} implies that the distribution $\deln(x)$ is $n$-dimensional for every $x \in \R^N$, with $n=N-1$. In the setting of this section, this is the maximum rank that the distribution $\deln$ can have (as $V_0= (U_0, W_0)$ is not contained in $\mathcal{R}_m$ when $W_0 \neq 0$).  In other words, for every $x \in \R^N$, the integral manifolds $S_x$ of $\deln(x)$ are $(N-1)$-dimensional manifolds. Because of the particularly simple structure of the SDE, such manifolds are just hyperplanes: for $x=(z,\zeta)$, $S_x=S_{(z,\zeta)}=\{u\in \R^{n+1}: u=(z, \eta), \eta=\zeta,   z \in \R^n \}$.  In this explicit setting Proposition \ref{prop:SDElivesonsubmanifold} is easy to check.  }  
\item \textup{ To reconcile the present work with the framework of  \cite{Cattiaux} and further elaborate on the meaning of Hypothesis \ref{item:smoothnessassumption}, let us assume for the moment that $W_0 \equiv 1$ and that $\zeta(0)=0$, so that \eqref{SDEglobal} becomes a standard time inhomogeneous SDE of the form \eqref{SDEnonhomog}. 
In this case the vector fields $U_0, \dd U_d$ are $\R^n$-valued maps  whose coefficients depend on time, i.e. $(z,t) \mapsto U_j(z,t) \in \R^n$. For simplicity, let also $n=1$. Then $V_0$ acts both on space and time, while   $V_1 \dd V_d$ act on the space coordinate $z$ only. That is, $V_0 = U_0(z,t) \pa_z + \pa_t$ while $V_j=U_j(z,t) \pa_z$  for $j=1\dd d$,  so that
\be\label{commtime}
[V_0, V_j]= [U_0, U_j] + (\pa_t U_j(z,t)) \pa_z  \quad j \in \{1 \dd d\}.
\ee
One can then rephrase Hypothesis \ref{item:smoothnessassumption} just in terms of the fields $U_0 \dd U_d$; from \eqref{commtime} it is then clear that Hypothesis \ref{item:smoothnessassumption} is equivalent to assuming that the Lie algebra 
$$
\bigcup_{k\geq 1} \mathrm{span} \{\mathfrak{L}_k^U(z,t)\},
$$
where $\mathfrak{L}_1^U(z,t):=\{U_1(z, t) \dd U_d(z, t)\}$ and, for $k >1 $, $\mathfrak{L}_k^U(z,t):=\{[U, U_j], U \in \mathfrak{L}_{k-1}^U, 1 \leq j \leq d\} \cup \{[U, U_0+ \pa_t], U \in \mathfrak{L}_{k-1}^U \}$,   
should be finitely generated and span $\R^n_z$, for every $(z,t) \in \R^n \times \R$.
 }

\textup{ Let us now go back to the general representation of the form ``ODE+SDE" \eqref{SDEglobal} - \eqref{ODEglobal}, without assuming $W_0=1$. 
Recall that in this context the vector fields $U_j$ are $\R^n$-valued functions of $n+1$ variables; that is, we view them as maps $\R^n \times \R \ni(z, \zeta) \mapsto U_j(z, \zeta) \in \R^n$.  Set again $n=1$ just for simplicity (everything we write in this comment would be true anyway). Then, as differential operators, $U_0 \dd U_j$ only act on  the variable $z$, while $W_0$ only  acts on the variable $\zeta$,  i.e.  we have the correspondence
$$
U_j(z,\zeta) \longleftrightarrow  U_j(z,\zeta)\pa_z \,\, \mbox{ for }j \in \{0 \dd d\}\quad \mbox{and} \quad W_0(\zeta) \longleftrightarrow W_0(\zeta)\pa_{\zeta} \,.
$$
One has  
\begin{align*}
[V_0, V_j]  &= [U_0\pa_z, U_j \pa_z]   +[W_0\pa_{\zeta}, U_j\pa_z]  
=[U_0\pa_z, U_j \pa_z]  + W_0(\zeta) (\pa_{\zeta} U_j) \pa_z  \, , \quad 1 \leq j \leq d\,.
\end{align*}
If we calculate the second term on the RHS of the above along a solution $\zeta_t$ of the ODE, we obtain 
  $$ W_0(\zeta_t) (\pa_{\zeta} U_j(z, \zeta_t)) \pa_z = \pa_t(U_j(z, \zeta_t)) \pa_z  .$$ 
This suggests that we may evaluate the vector fields along the solution of the ODE and then think of them  as functions of $z$ and time $t$, rather than as functions of $z$ and $\zeta$, i.e. $\R^n_z\times \R_t \ni (z,t)\mapsto U_j(z,\zeta_t^\zeta) \in \R^n, \, j \in \{0 \dd d\}$. If we do so, then Hypothesis \ref{item:smoothnessassumption} can be equivalently rephrased as follows:  the Lie algebra 
$$
\bigcup_{k\geq 1} \mathrm{span} \{\mathfrak{L}_k^U(z,\zeta_t)\}
$$
 is finitely generated and spans $\R^n_z$  for every $z \in \R^n$ and along any solutions $\zeta_t$ of the ODE \eqref{ODEglobal}. \footnote{Given an initial datum, the solution of the ODE is unique. When we say that this should hold along any solutions, we mean along all the solutions that one can obtain by starting from different initial data.}
}
\item \textup{ As is well known,  Hypothesis \ref{item:tightnessassumption} is implied by a Lyapunov-type condition;  namely, if  there exists some non-negative function $\varphi\in C^2(\R^n)$ with compact level sets and such that 
\be\label{LyapCond}
\mathcal{L}_t\varphi(z)\leq C_1-C_2\varphi(z), \quad \text{ for every } z\in \R^n, t\geq0, 
\ee
then \ref{item:tightnessassumption} is satisfied. 
Here $\mathcal{L}_t$ is the operator
\begin{equation*}
\mathcal{L}_t\psi(z)=U_0(z,\zeta_t)\cdot \nabla \psi(z)+\sum_{i=1}^d U_i(z,\zeta_t) \cdot \nabla(U_i(z,\zeta_t) \cdot \nabla \psi(z)), 
\end{equation*}
where $\nabla=(\pa_{z_1} \dd \pa_{z_n})$. 
\item The Obtuse angle condition does not imply tightness; in Example \ref{ex:grushin}  we show  that \ref{item:tightnessassumption} does not imply \ref{item:OAC} and  \ref{item:OAC} does not imply \ref{item:tightnessassumption}. Viceversa, we also show that  \ref{item:OAC} does not imply the existence of a Lyapunov function, condition \eqref{LyapCond}. }
	\end{itemize}
\hf
\end{note}

\begin{note}\label{notelinksec67}
\textup{ As already pointed out, if $X_t=(Z_t, \zeta_t)$, where $Z_t, \zeta_t$ are given by a representation of the form ``ODE+SDE'' \eqref{SDEglobal}-\eqref{ODEglobal}, then $X_t$  solves an SDE of the form \eqref{SDE}, with $V_0=(U_0, W_0),V_1=(U_1, 0),\dots,V_d=(U_d, 0)$. Hence $\voperp=(0 \dd 0, W_0)$ (see definition \eqref{defvoperp}).
We note in passing  that in this case one has
$$
\cZ_t:=e^{-t \voperp}X_t= e^{-t \voperp}(Z_t, \zeta_t)=
e^{-t \voperp}(Z_t, e^{t\voperp}\zeta_0)= (Z_t, \zeta_0). 
$$
(This is not of much use at the moment, but it will help at the beginning of Section \ref{sec:longtimebehaviour} to make a link between the setting of this section and the setting of the next). 
Therefore, while  $X_t$ belongs to the hyperplane $\mathcal{H}_{\zeta_t}:=\{x \in \R^{n+1}: x=(z,\zeta_t), z \in \R^n\}$ for each $t\geq 0$, $\cZ_t$ remains, for every $t \geq 0$,  on the same hyperplane, namely the hyperplane  $\mathcal{H}_{\zeta_0}:=\{x \in \R^{n+1}: x=(z,\zeta_0), z \in \R^n\}$ (which is precisely the manifold $S_{x_0}=S_{(z_0, \zeta_0)}$, see second bullet point in Note \ref{noteonhypglobal}) for every $t\geq 0$. \hf
}
\end{note}

\begin{lemma}\label{lem:ergbarZ}
Let Hypothesis \ref{hyp:nonautoassumptions} hold. Then the SDE \eqref{limitingproc} admits a unique invariant measure, $\mub$. Moreover, 
$$
(\barQ_tg)(z) \rar \int_{\R^n} g(z) \,\mub(dz), \quad \mbox{for every } z \in \R^n \mbox{ and every } g\in C_b(\R^{n})\,. 
$$
\end{lemma}
\begin{proof}[Proof of Lemma \ref{lem:ergbarZ}]
This is completely standard and we omit it. See for example \cite{DragKonZeg}. We just point out that the existence of the invariant measure comes  from assumption \ref{item:tightnessassumption} and the uniqueness is  a consequence of Hypothesis \ref{item:OAC} and Proposition \ref{invmesonstathyperpl}  . 
\end{proof}

\begin{theorem}\label{thm:mainglobalthm}
Let Hypothesis \ref{hyp:nonautoassumptions} hold. In particular, let $\barz=\barz(\zeta_0)$ be a stationary point for the ODE \eqref{ODEglobal} and $\mub$ be the invariant measure of the process \eqref{limitingproc}.  Then,  for every $s\geq 0$, 
$$
\lim_{t\rar \infty} (\cQ_{s,t}^{\zeta_0}g)(z) =  \int_{\R^n} g(z) \,\mub(dz),  \quad \mbox{for every } z \in \R^n \mbox{ and every } g\in C_b(\R^{n})\, .
$$ 
\end{theorem}
The proof of this theorem can be found after the statement of Theorem \ref{thm:mainglobalthmforX}.  Theorem \ref{thm:mainglobalthm} describes the asymptotic behaviour of the process $Z_t$. However, in this paper we are interested in  the process $X_t$. The long-time behaviour of the process $X_t$ is described by  Theorem \ref{thm:mainglobalthmforX} below, which is  just a straightforward consequence of Theorem  \ref{thm:mainglobalthm}. In order to state Theorem \ref{thm:mainglobalthmforX}, we clarify the following: while $Z_t$ is a process in $\R^n$ with invariant measure(s) $\mub=\mub(\barz, \zeta_0)$ supported  on $\R^n$,  $X_t$ is a process in $\R^{n+1}$; so, strictly speaking, any invariant measure of $X_t$ is a probability measure on $\R^{n+1}$. However such a measure is supported on the $n$-dimensional hyperplane 
$$\mathcal{H}_{\barz}:=\{x \in \R^{n+1}: x=(z,\barz), z \in \R^n\}$$ 
and it is just a trivial extension of the measure $\mub$. That is, let $\mu=\mu(\barz, \zeta_0)$ be the measure on $\R^{n+1}$ such that  
\be\label{extensioninvmeasZ}
\mu(A)= \mub(A \cap \mathcal{H}_{\barz}) \quad \mbox{ for every Borel set }  A \subseteq\R^{n+1}.
\ee
In particular, 
$\mu(A)=\mub(A)$ if  $A \subseteq \mathcal{H}_{\barz}$ and 
$\mu(A)=0$ {if } $A \cap \mathcal{H}_{\barz}= \emptyset$.
Let $I_0(\barz)=\{\zeta_0 \in \R :  \zeta_t^{\zeta_0} \rightarrow \barz \mbox{ as } t \rar \infty\}$.  Let also $\mathcal{I}_0= \mathcal{I}_0(\barz):=\{x_0 \in \R^{n+1}: x_0=(z_0, \zeta_0), \zeta_0\in I_0(\barz), z_0 \in \R^n\}$. 

\begin{theorem}\label{thm:mainglobalthmforX}
Consider the process $X_t=(Z_t, \zeta_t) \in \R^{n+1}$ satisfying a representation the form of ``SDE+ODE'' \eqref{SDEglobal}-\eqref{ODEglobal} with  initial condition \eqref{incondglobal}  and  associated semigroup $\cP_t$, defined in \eqref{semigrglobalX}.  Let Hypothesis \ref{hyp:nonautoassumptions} hold. In particular, according to Hypothesis \ref{hyp:nonautoassumptions} \ref{item:convODE}, let  $\barz$ be a (any) stationary point of the ODE \eqref{ODEglobal}  and
$\mub$ be the invariant measure of the corresponding process \eqref{limitingproc}; let also  $\mu=\mu(\barz, \zeta_0)$ be the measure on $\R^{n+1}$ defined in \eqref{extensioninvmeasZ} and supported on the hyperplane $\mathcal{H}_{\barz}$. Then, for every $x \in \mathcal{I}_0= \mathcal{I}_0(\barz)$, we have 
$$
\lim_{t\rar \infty} (\cP_{t}f)(x) =  \int_{\R^{n+1}} f(u) \,\mu(du)
= \int_{\mathcal{H}_{\barz}} f(u) \,\mu(du),   
$$
for every $f\in C_b(\R^{n+1})$. The limit in the above  does not hold if $x\notin \mathcal{I}_0$; that is, $\mathcal{I}_0$ is the whole basin of attraction of the measure $\mu=\mu(\barz,\zeta_0)$.  
\end{theorem}

We now introduce some definitions that will be needed for the proof of Theorem \ref{thm:mainglobalthm}.  A family $\{\nu_t\}_{t\geq 0}$ of probability measures on $\R^n$ is said to be an \emph{evolution system of measures} for the two-parameter semigroup $\cQ_{s,t}$ if for all $0\leq s\leq t$ and $g\in C_b(\mathbb{R}^{n})$ we have
	\begin{equation}\label{eq:evolutionsystem}
	\int_{\mathbb{R}^{n}} \cQ_{s,t}g(z) \nu_s(dz)= \int_{\mathbb{R}^{n}} g(z) \nu_t(dz).
	\end{equation}
Let $\cQ_{s,t}^*$ denote the adjoint of $\cQ_{s,t}$ over the space $C_b(\mathbb{R}^{n})$, that is
\begin{equation*}
(\cQ_{s,t}^*\nu)(A) := \int_{\mathbb{R}^n}\cQ_{s,t}\mathbbm{1}_A(z) \nu(dz), \quad \text{for any Borel measurable } A\subseteq\mathbb{R}^{n}.
\end{equation*}
Then we can write \eqref{eq:evolutionsystem} as 
\begin{equation*}
\cQ_{s,t}^*\nu_s=\nu_t, \quad \text{ for all } \,0\leq s\leq t.
\end{equation*}
Further background  on  evolution system of measures can be found in \cite{Kunze, {daPratoRoeckner}}.

\begin{proof}[Proof of Theorem \ref{thm:mainglobalthm}]
The proof is  in three steps. 

{$\bullet $ {\em Step 1}}: We first construct a tight evolution system of measures, $\{\nu_t\}_{t\geq 0}$, for the semigroup $\cQ_{s,t}$. To this end,  take any point $z_0\in \mathbb{R}^{n}$, define $\nu_0=\delta_{z_0}$  and then let $\nu_t:=\cQ_{0,t}^*\nu_0$. Now $\nu_t$ is an evolution system of measures;  indeed, 
\begin{equation*}
\cQ_{s,t}^*\nu_s=\cQ_{s,t}^*\cQ_{0,s}^*\nu_0=\cQ_{0,t}^*\nu_0= \nu_t.
\end{equation*}
(A more general construction of the evolution system is given in \cite[Section 5]{Kunze}). 
To see that $\{\nu_t\}_{t\geq 0}$ is tight, fix $\varepsilon>0$; by Hypothesis \ref{item:tightnessassumption} we may take a compact set $K_\varepsilon \subset \R^n$ such that $q_t^{0,z_0}(K_\varepsilon) \geq 1-\varepsilon$. By definition of $\nu_t$ we then have
\begin{equation*}
\nu_t(K_\varepsilon) =(\cQ_{0,t}^* \nu_0) (K_{\epsilon})
= \cQ_{0,t}\mathbbm{1}_{K_\varepsilon}(z_0) \geq 1-\varepsilon.
\end{equation*} 

{$\bullet $ {\em Step 2}}: $\cQ_{s,t}g(z)-\nu_t(g)$ converges to zero as $t$ tends to $\infty$ for all $s\geq0, z\in\R^n, g\in C_b(\R^{n})$. We defer the proof of this fact to Lemma \ref{lem:convtoevolutionsystem}. Since $\{\nu_t\}_{t\geq0}$ is tight, by Prokhorov's Theorem there exists a diverging sequence $t_\ell$ such that $\nu_{t_\ell}$ converges weakly to some probability measure $\mu_0$, as $t_{\ell}$ tends to $\infty$. 

{$\bullet $ {\em Step 3}}: Show that $\mu_0=\bar{\mu}$. We defer the proof of this equality to Lemma \ref{lem:convofevolsystemtoequilibria}.
If $\mu_0=\bar{\mu}$, then $\nu_t$ converges weakly to $\bar{\mu}$ and 
the claim of the theorem follows; indeed,
\begin{equation*}
\lv\cQ_{s,t}g(z)-\bar{\mu}(g)\rv \leq \lv\cQ_{s,t}g(z)-\nu_{t}(g)\rv+ \lv\nu_t(g)-\bar{\mu}(g)\rv.
\end{equation*}
The first term converges to zero by Step 2 and the second term vanishes in the limit since $\nu_t$ converges weakly to $\bar{\mu}$ as $t\to\infty$. 
\end{proof}
\begin{note}\label{note:comparelunardi}\textup{The statements and proofs  of  Lemma \ref{lem:convtoevolutionsystem} and 
Lemma \ref{lem:convofevolsystemtoequilibria} are the core of the proof of Theorem \ref{thm:mainglobalthm}. The arguments used in the proofs of such lemmata are analogous in structure to those presented in \cite[Section 6]{AngiuliLorenziLunardi}. The main differences arise when dealing with the regularity of the semigroup, as \cite{AngiuliLorenziLunardi} assumes uniform ellipticity. Lemma \ref{lem:semigpconvergetosemigp} (needed to prove Lemma \ref{lem:convofevolsystemtoequilibria}) is the main place where we take care of the relaxed regularity assumptions.}
\hf
\end{note}

%%%%%%%%%%%%%%%%%%%%%%%%%%5
%%%%%%%%%%%%%%%%%%%%%%%%%%%
%%%%%%%%%%%%%%%%%%%%%%%%%%%%%%%
%%%%%%%%%%%%%%%%%%%%%%
\begin{comment}
\begin{lemma}\label{basin} \textcolor{blue}{I am not convinced!! Given an invariant measure $\mu$ who is $\zeta_0$? What about if there is only one equilibria $\barz$ and it is stable, then the basin is everything! Do you want the union over all $\zeta_0\in I(\barz)$?} With the notation of  Theorem \ref{thm:mainglobalthmforX}, the basin of attraction of the measure $\bar{\mu}(\barz,\zeta_0)$
 can be equivalently described as the set $\curlys_{x_0} \cup \mathcal{H}_{\barz}$, where $\curlys_{x_0}$ is the integral manifold of $\delnn$ that contains the hyperplane $\mathcal{H}_{\zeta_0}$ (and hence the initial datum $x_0=(z_0, \zeta_0)$ of the SDE \eqref{SDEglobal}-\eqref{ODEglobal}). 
That is, 
$$
\mathcal{I}_0(\barz) = \curlys_{x_0} \cup \mathcal{H}_{\barz} \,.
$$
\end{lemma}
\begin{proof} [Proof of Lemma \ref{basin}] The proof of this lemma is elementary, considering the particularly simple structure of the integral manifolds of $\deln$ and $\delnn$, so we don't detail it. Notice however that, while the hyperplane $\mathcal{H}_{\barz}$ is part of the closure of the manifold $\curlys_{x_0}$, it is not necessarily the whole closure, see Example \ref{ex:ODEwithmanyinvmeas}. 
\end{proof}
\end{comment}

Let  $p_t^x$ denote the measure defined by
\begin{equation*}
p_t^x(A)=\cP_t\mathbbm{1}_A(x), \quad \text{ for all Borel sets } A\subseteq \R^{n+1}.
\end{equation*}

\begin{prop}\label{nounbounded}
If $\zeta_t^\zeta \rightarrow \infty$ then the family of measures $\{p_t^{(z,\zeta)}\}_{t\geq 0}$ is not tight for any $z\in\R^n$ (hence, by Prokhorov's Theorem, there is no probability measure $\mu$ such that
$
\cP_tf(z,\zeta) \to \mu(f),
$
for all $f\in C_b(\R^{n+1})$). 
\end{prop}
\begin{proof}[Proof of Proposition \ref{nounbounded}]
Fix $z\in\R^n$ and let $x=(z,\zeta) \in \R^{n+1}$. Assume by contradiction  that $\{p_t^x\}_{t\geq 0}$ is tight. Then, for any fixed  $\varepsilon>0$  there exists a compact set $K_\varepsilon \subset \R^{n+1}$ such that $p_t^x(K_\varepsilon)>1-\varepsilon$ for all $t\geq 0$. Since $K_\varepsilon$ is compact we may take $R$ sufficiently large such that $K_\varepsilon\subseteq \R^n \times [-R,R]$;  then one has
\begin{equation}\label{eq:tightnessimpliesODEbounded}
\mathbb{P}_x \left(\vert\zeta_t^\zeta\vert \leq R \right) \geq p_t^x(K_\varepsilon)\geq 1-\varepsilon, \quad \mbox{for all } t\geq 0.
\end{equation}
However $\zeta_t^\zeta\to\infty$ so we may take $t$ sufficiently large that $\vert\zeta_t^\zeta\vert > R$. This contradicts \eqref{eq:tightnessimpliesODEbounded},  hence $p_t^x$ is not tight.
\end{proof}

\begin{example}[UFG-Gru\v{s}in Plane] \label{ex:grushin}
\textup{ We give here a simple example of a process that satisfies the Obtuse Angle Condition but is not tight. 
	Let $d=1$,  $N=2$ and
	\begin{equation*}
	V_0=k\zeta\partial_\zeta, \quad V_1=\zeta\partial_z, \quad  k \in \R \,.
	\end{equation*}
	This corresponds to the SDE
	\begin{align*}
	d\zeta_t&=k\zeta_tdt\\
	dZ_t&=\sqrt{2}\zeta_t\circ dB_t,
	\end{align*}
where $\{B_t\}_{t\geq 0}$ is a one-dimensional Brownian motion. Because $[V_1,V_0]=-kV_1$, we have
\begin{equation*}
([V_1,V_0]f)(V_1f)=-k(V_1f)^2
\end{equation*}
therefore the Obtuse Angle Condition, \eqref{eq:OAC} is satisfied if and only if $k>0$ (it is also shown in \cite[Example 4.4]{CrisanOttobre} that, if $k>0$,  then $V_1 (\cP_tf) (\cdot)$ decays exponentially fast with rate $-2k$). On the other hand, if $k>0$ the process is not tight. Indeed,  Hypothesis \ref{hyp:nonautoassumptions} \eqref{item:tightnessassumption} is satisfied if and only if $k<0$, as we come to show. To this end,   we first solve the SDE, and  find
	\begin{align*}
	\zeta_t & = \zeta e^{kt}\\
	Z_t&=Z_0+\sqrt{2}\zeta\int_0^t e^{ks} \circ dB_s.
	\end{align*}
	As a consequence of Proposition \ref{nounbounded}, the whole process $(Z_t,\zeta_t)$ is not tight if $k>0$. However in this case also the process $Z_t$, seen as a non-autonomous one dimensional SDE, is not tight when $k>0$. Indeed suppose by contradiction that \eqref{item:tightnessassumption} holds; then for any $\varepsilon>0$ there exists $R>0$ such that
	\begin{equation}\label{contresempio}
	Q_{0,t}\mathbbm{1}_{[-R,R]} (z)\geq 1-\varepsilon \,, \quad \mbox{for all } t\geq 0.
	\end{equation}
	 However if $Z_0=z$ then $Z_t$ is normally distributed with mean $z$ and variance $\zeta^2(e^{2k t}-1)/k$, so we may write
	\be\label{gaussz}
	Z_t=z+\zeta \sqrt{\frac{e^{2kt}-1}{k}} \xi
	\ee
	where $\xi$ is  a one-dimensional standard normal random variable. Then we have
	\begin{equation*}
	\cQ_{0,t}^\zeta\mathbbm{1}_{[-R,R]}(z) = 
\mathbb{E} \mathbbm{1}_{[-R,R]}(Z_t^{0,z,\zeta})= \mathbb{P}\left(\lv z+\zeta\sqrt{\frac{e^{2kt}-1}{k}} \xi\rv \leq R\right) \rightarrow 0 \quad \mbox{as } t \rar \infty, 
	\end{equation*}
	which contradicts \eqref{contresempio}.  Note that if $k=0$ then $Z_t=\sqrt{2}\zeta B_t$ which is not tight by a similar argument. However if $k<0$ then the process $Z_t$ is tight. 
Indeed,  assume that $k=-\ell<0$;   to see that $\{q_t^{0,z}\}_{t\geq0}$ is a tight family of measures,  it is sufficient to apply a Lyapunov criterion and show that the function  $\varphi(z)=z^2$ satisfies $\sup_{t}\cQ_{0,t}\varphi(y)<\infty$ (when  $k=-\ell<0$).  To prove the latter fact, observe that if  $(Z_s,\zeta_0)=(z,\zeta)$ then, by \eqref{gaussz}, we get 
\begin{equation*}
\cQ_{s,t}^\zeta\varphi(z)= \mathbb{E}[Z_t^2\rvert \zeta_0=\zeta, Z_s=z] = z^2+ \frac{\zeta^2e^{-2 \ell s}}{\ell}(1-e^{-2\ell (t-s)}) \leq z^2+\frac{e^{-2\ell s}}{\ell}\zeta^2.
\end{equation*}
If $k<0$ we see that $X_t=(Z_t,\zeta_t)$ converges in law.
}
\hf
\end{example}

\begin{example}\label{ex:ODEwithmanyinvmeas}
\textup{ We conclude this section with an example which satisfies all the points in  Hypothesis \ref{hyp:nonautoassumptions} in a non-trivial way, in the sense that it  exhibits many invariant measures. 
Take $k>1$ and consider the following SDE
\begin{align*}
d\zeta_t&=-\sin(\zeta_t)dt\\
dZ_t&=-kZ_tdt+\sqrt{2}\zeta_t\circ dB_t.
\end{align*}
In this case
\begin{equation*}
V_0=-\sin(\zeta)\partial_\zeta-kz\partial_z, \quad V_1=\zeta\partial_z, U_0=-kz\partial_z, U_1=\zeta\partial_z.
\end{equation*}
Then we have
\begin{align*}
[V_1,V_0]=[\zeta\partial_z,-\sin(\zeta)\partial_\zeta-kz\partial_z]=-k\zeta\partial_z-\sin(\zeta)\partial_z = \left(-k +\frac{\sin(\zeta)}{\zeta}\right)V_1.
\end{align*}
Note that the function $h(\zeta)={\sin(\zeta)}/{\zeta}$ is bounded and smooth,  when extended to the origin with the value $h(0)=1$,  so the UFG condition is satisfied at level $m=1$. Moreover,
\begin{equation*}
([V_1,V_0]f)(V_1f) = -(k+\frac{\sin(\zeta)}{\zeta})\lv V_1f \rv^2 \leq -(k-1)\lv V_1f \rv^2 
\end{equation*}
and hence \eqref{eq:OAC} is satisfied. To apply the results of Section \ref{sec:non-autonomous} we must show that Hypothesis \ref{hyp:nonautoassumptions} holds. Note that the vector field $V_1$ is non-zero except when $\zeta=0$ therefore Hypothesis \ref{hyp:nonautoassumptions} \ref{item:smoothnessassumption} is satisfied everywhere that $\zeta\neq 0$. To show that Hypothesis \ref{hyp:nonautoassumptions} \ref{item:tightnessassumption} holds we consider a function $\varphi\in C^2(\R)$ such that $\varphi(z)=\lv z\rv$ for $\lv z\rv>1$. Then, for $\lv z \rv >1$, one has
\begin{equation*}
\Lt_t\varphi(z)=-kz\varphi'(z)+\zeta_t^2\varphi''(z) =-kz\mathrm{sign}(z) = -k\varphi(z).
\end{equation*}
Therefore $\varphi$ is a Lyapunov function so by Note \ref{noteonhypglobal} we have that the measures $\{q_t^{0,z}:t\geq 0\}$ are tight for any $z\in \R$ and Hypothesis \ref{hyp:nonautoassumptions} \ref{item:tightnessassumption} is satisfied. We also have that $\zeta_t$ converges for any $\zeta\in\R$ and the limit $\overline{\zeta}$ is given by 
\begin{equation*}
\overline{\zeta} = \left\{\begin{array}{ll}
2n\pi \quad &\text{ for } \zeta \in ((2n-1)\pi,(2n+1)\pi) \text{ for some } n\in\mathbb{Z}\setminus\{0\}\\
(2n+1)\pi \quad &\text{ for } \zeta=(2n+1)\pi \text{ for some } n\in\mathbb{Z}\\
0 \quad &\text{ for } \zeta \in (-\pi,\pi).
\end{array}\right.
\end{equation*}
Hence for $\zeta\notin (-\pi,\pi)$ we may apply Theorem \ref{thm:mainglobalthmforX} to obtain that $X_t=(Z_t,\zeta_t)$ converges in law to a random variable which whose law corresponds to the unique invariant measure defined on the line $\R\times \{\overline{\zeta}(\zeta_0)\}$. Moreover, for $\zeta=n\pi$ for some $n\in\mathbb{Z}\setminus\{0\}$ we see that $\zeta_t=\zeta$ and $Z_t$ satisfies the Ornstein Uhlenbeck SDE
\begin{equation*}
dZ_t=-kZ_tdt+\sqrt{2}\zeta dB_t.
\end{equation*}
In particular, in this case $Z_t$ has a unique invariant measure and this is given by a Gaussian measure with mean $0$ and variance $\zeta^2/k$. Therefore for any $n\in\mathbb{Z}\setminus\{0\}$ and $\zeta\in ((2n-1)\pi,(2n+1)\pi)$ we have that $X_t$ converges in law to $(\frac{2n\pi}{\sqrt{k}}\xi ,2n\pi)$,  where $\xi$ is a one-dimensional standard normal random variable.\footnote{Since $Z_t$ satisfies a non-autonomous Ornstein Uhlenbeck equation one can also study its asymptotic behaviour more directly, see e.g. \cite{Geissert}. }
}
\hf
\end{example}

\section{Long-time behaviour of UFG diffusions: general case} 
\label{sec:longtimebehaviour}
In the previous section we investigated the case in which the representation of the form ``ODE+SDE" is global. In this section we study the  general UFG-case, in which such a representation is, in general,  only local. That is, we finally address the full problem of analysing the asymptotic behaviour of $\eqref{SDE}$, assuming that the vector fields $V_0 \dd V_d$ satisfy the UFG condition (see Definition \ref{defufg}).  This case is substantially richer than the one considered in Section \ref{sec:non-autonomous}; however the fact that, locally, we can always represent the SDE \eqref{SDE} as a system of the form ``ODE+SDE", still means that we should be able to identify a suitable ODE which drives the dynamics.  We will demonstrate that this is indeed the case and that such an ODE is the integral curve of the vector field $\voperp$; that is, the curve
\be\label{ODElocall}
\zeta_t=e^{t\voperp}x_0 \,, 
\ee
where $x_0$ is the initial datum of the SDE \eqref{SDE}, i.e. $X_0=x_0$. This should not be a surprise in view of Proposition \ref{prop:SDElivesonsubmanifold}. Nevertheless, 
to understand why this is the case, it is useful to build an analogy with the setting of the previous section: if the SDE  is of the form \eqref{SDEglobal}-\eqref{ODEglobal}, then $\voperp=(0 \dd 0, W_0)$. Therefore in the simplified setting \eqref{SDEglobal}-\eqref{ODEglobal}, the ODE \eqref{ODElocall} substantially reduces to \eqref{ODEglobal}. The previous sentence is correct for less than observing that \eqref{ODElocall} is an $N$-dimensional ODE, while \eqref{ODEglobal} is a one-dimensional curve. We keep using the notation $\zeta_t$ for both curves only to emphasize the analogy; however, while  the one-dimensional autonomous nature of the ODE \eqref{ODEglobal} implies that its solution has a limit, the zoology of possible behaviours for the curve \eqref{ODElocall} is much more varied. In this paper we only analyse the case in which the curve \eqref{ODElocall} converges to a limit and in future work we will treat more general cases. \footnote{For example the curve \eqref{ODElocall} could be periodic,  ergodic or chaotic (this is a non-exhaustive list of possibilities).} However, roughly speaking, in Theorem \ref{prop:invmeasureforODE}, we prove that a necessary condition for the SDE \eqref{SDE} to have  an invariant measure is that the ODE \eqref{ODElocall} should admit  one as well (notice that if the curve \eqref{ODElocall} converges to a limit point $\bar{x}$, then it admits the Dirac measure $\delta_{\bar{x}}$ as invariant measure). 
\begin{comment}
Observe that, as a consequence of this theorem, we can again deduce that if the curve \eqref{ODElocall} diverges then the SDE \eqref{SDE} will not  admit any invariant measures (the precise statement of this fact is contained in Corollary \ref{cor:div}). 
\end{comment}

As anticipated in the introduction, the above discussion motivates introducing the process
\begin{equation}\label{eq:Zdef}
\cZ_t:=e^{-t\wo}(X_t^{(x_0)}).
\end{equation}
Clearly $\cZ_{0}=x_0$, so $\cZ_t$ and $X_t^{(x_0)}$ start from the same point.  This process is  time-inhomogeneous (as we show at the beginning of Section \ref{sec7core}) and it  will have a central role in what follows, hence further comments on the definition  \eqref{eq:Zdef}  are in order:
\begin{itemize}
\item To continue drawing the useful parallel with Section  \ref{sec:non-autonomous}, notice that  this process plays in this context an analogous role to the one that $Z_t$ (solution of \eqref{SDEglobal}) has in Section \ref{sec:non-autonomous}, see Note \ref{notelinksec67}.  
\item Let us   recall   that if $X_0\in S_{x_0}$ then $X_t \in \overline{\curlys}_{x_0}$ for every $t\geq 0$ (see Proposition \ref{correachstoc}); more precisely, for every $t \geq 0$ $X_t$ belongs to the integral submanifold $\Stbar$ almost surely (see Proposition \ref{prop:SDElivesonsubmanifold}). We will make assumptions to guarantee that $X_t$  hits neither  the boundary of $S_{e^{t\voperp}\!\!x_0}$ nor the boundary of $\curlys_{x_0}$ {\em in finite time} (see  Hypothesis \ref{hyp:gennonautoassumptions} \ref{item:gentightnessofq}, Lemma \ref{lem:boundary} and Note \ref{noteonhyplocal} for more precise comments on this). Therefore $\cZ_t$ lives on the manifold  $S_{x_0}$, for every $t \geq 0$. So, in the end, while $X_t$ takes values in $\curlys_{x_0}$, $\cZ_t$ takes values in $S_{x_0}\subseteq \curlys_{x_0}$. One can informally think of $\cZ_t$ as being a ``projection" of $X_t$ on the submanifold $S_{x_0}\subseteq \curlys_{x_0}$,  see again Note \ref{notelinksec67}.  
\item Finally, on a small technical point, as we have already observed in Note \ref{technicalpoint}, $\voperp$ may not be uniformly Lipschitz. However, to avoid problems of well posedness and uniqueness, throughout this section we assume that $\voperp$ is indeed Lipschitz. 
\end{itemize}

We will show that the  time-inhomogeneous  process $\cZ_t$ can be studied by means of slight modifications of the approach used in Section \ref{sec:non-autonomous} to study the process \eqref{SDEglobal}. Therefore
 the strategy (and one of the main novelties) of this section is to use the auxiliary  time-inhomogeneous process $\cZ_t$ in order to make deductions on the behaviour of the time-homogeneous process $X_t$. We carry out this programme in Section  \ref{sec7core} below. Before moving on, we give a simple example which demonstrates that $\cZ_t \in S_{x_0}$ for every $t\geq 0$ and, in Section \ref{ztsec}, we gather further preliminary results on the process $\cZ_t$. 

%%%%%%%%%%%%%%%%%%%%%%%
%%%%%%%%%%%%%%%%%%%%%%%%
%%%%%%%%%%%%%%%%%%
%%%%%%%%%%%%%%%%%%%%%%555555555555

\begin{example}[Random Circles continued] 
\textup{Consider again Example \ref{ex:circle}, in the case in which the initial datum is $(x_0, y_0)=(1,0)$ . Using \eqref{explrandcircvoperp} and \eqref{solcircles1}-\eqref{solcircles2}, we have 
\begin{align*}
 	\cZ_t:=e^{-t\voperp}(X_t,Y_t) &= \left(\begin{array}{c}X_t\cos(-t)-Y_t\sin(-t)\\
 	X_t\sin(-t)+Y_t\cos(-t)
 	\end{array}\right)\\
 	&= \left(\begin{array}{c}e^{\sqrt{2} B_t}\cos(t) \cos(t)+e^{\sqrt{2} B_t}\sin(t) \sin(t)\\
 	-e^{\sqrt{2} B_t}\cos(t) \sin(t)+e^{\sqrt{2} B_t}\sin(t) \cos(t)
 	\end{array}\right)\\
 	&=e^{\sqrt{2} B_t}\left(\begin{array}{c}1\\
 	0
 	\end{array}\right)
 	\end{align*}
	In particular, $\cZ_t$ takes values in the positive half-line,  which is precisely $S_{(1,0)}= S_{(x_0, y_0)}$. 
}\hf
\end{example}
\begin{note}\label{note:compareBismut1}
\textup{In \cite{Bismut} the author considers the curve $e^{tV_0}(X_t^{(x)})$ while here we consider the curve $e^{t\voperp}(X_t^{(x)})$. This is due to two things: i) in \cite{Bismut}  the author was not concerned with the dynamics  of the process but was instead interested in studying the density of the SDE; ii) the work \cite{Bismut} assumes that the HC is satisfied, so that $\delnn$ has constant rank (constant and equal to $N$) and the decomposition of state space into integral manifolds of $\delnn$ becomes in that case trivial - there is only one submanifold, which is the whole $\R^N$.  \\
When studying the density of the law of the process there is no advantage in using $\voperp$ over $V_0$ - indeed even though $V_0$ is smooth,  $\voperp$ need not be continuous, see Note \ref{technicalpoint}. On the other hand, if one is interested in the dynamics of the process (and the HC is not enforced), then $\voperp$ becomes an indicator of whether the path has left the integral manifold of $\delnn$ from which it started or not. In other words for us it is important to know the behaviour of the manifolds $S_{e^{t\voperp}x}$ as $t$ varies. In particular, we wish to distinguish between the case when $S_{e^{t\voperp}x}$ is constant in $t$ which corresponds to the situation in which $\voperp=0$ and when $S_{e^{t\voperp}x}$ varies in time,  see Section \ref{sec:5} and Lemma \ref{lemmavoonmanifold}. 
}
\end{note}

\subsection{The auxiliary process $\cZ_t$ and its associated two-parameter semigroup}\label{ztsec}

By differentiating \eqref{eq:Zdef} we see that $\cZ_t$  satisfies the following SDE
\begin{align*}
d\cZ_t&= -\wo(e^{-t\wo}\!\!(X_t^{(x)}))dt+\left( \jacobian{e^{-t\wo}}{x}\right)\!\!(X_t^{(x)}) V_0(X_t^{(x)})dt\\
&+\sqrt{2}\sum_{i=1}^d\left(\jacobian{e^{-t\wo}}{x}\right)\!\!(X_t^{(x)})V_i(X_t^{(x)})\circ dB_t^i\\
%&=-\wo(\cZ_t)dt+\jacobian{e^{-t\wo}}{x}(e^{t\wo}(\cZ_t)) V_0(e^{t\wo}(\cZ_t))dt\\&+\sqrt{2}\sum_{i=1}^d\jacobian{e^{-t\wo}}{x}(e^{t\wo}(\cZ_t))V_i(e^{t\wo}(\cZ_t))\circ dB_t^i\\
&=-\wo(\cZ_t)dt+\Ad_{t\wo}V_0(\cZ_t)dt+\sqrt{2}\sum_{i=1}^d\Ad_{t\wo}V_i(\cZ_t)\circ dB_t^i, %\label{eq:SDEnonauto}
\end{align*}
where, as customary, we have set   $(\Ad_{tV}Y)(x):= (\jacobian{e^{-tV}}{x})(e^{tV}(x)) \cdot Y(e^{tV}x)$, for any two smooth vector fields $V$ and $Y$. By using \eqref{defvoperp}, the elementary property $\Ad_{tV}V = V$ and introducing the notation 
\begin{equation}\label{block}
\begin{aligned}
\cV_{0,t} &:= \Ad_{t\voperp}\vodel \\
\cV_{j,t} & := \Ad_{t\voperp}V_j \quad j \in \{1 \dd d\}, \\
%\label{eq:curlyVdef}\\
%\cV_{[\alpha],t} & := \Ad_{t\voperp}V_{[\alpha]} \quad \alpha \in \mathcal{A}_m
\end{aligned}
\end{equation}
we conclude that $\cZ_t$ satisfies the following SDE with time-dependent coefficients: 
\begin{align}
d\cZ_t&=\Ad_{t\wo}\vodel(\cZ_t)dt+\sqrt{2}\sum_{i=1}^d\Ad_{t\wo}V_i(\cZ_t)\circ dB_t^i \nonumber\\
&=\cV_{0,t}(\cZ_t) dt+\sqrt{2}\sum_{i=1}^d\cV_{i,t}(\cZ_t)  \circ dB_t^i \,, \label{eq:SDEnonauto}
\end{align}
As usual, we denote by $\cP_t$ the one parameter semigroup associated with $X_t$; the two-parameter semigroup associated with $\cZ_t$ is instead given by
\begin{equation*}
\zQ_{s,t}f(\lz) = \mathbb{E}\left[f(\cZ_t)\lvert \cZ_s=\lz\right], \quad \lz \in \So, s\leq t, f\in C_b(\R^N).
\end{equation*}
The semigroups $\zQ_{s,t}$ and $\cP_t$ are related as follows:
\begin{align}
\cP_tf(x) &= \zQ_{s,s+t}(f\circ e^{(s+t)\wo})(e^{-s\wo}(x)), &&x\in {\capitals}_{e^{s\voperp}(x_0)}, \, s\in\R,t\geq 0, \, f\in C_b(\R^N), \nonumber\\
\zQ_{s,t}g(\lz) &= \cP_{t-s}(g\circ e^{-t\wo})(e^{s\wo}(\lz)), &&\lz\in S_{x_0}, \, 0\leq s\leq t, \,  g\in C_b(\Sbar)\, . \label{eq:autonomoustonon}
\end{align}
We stress that $\{\zQ_{s,t}\}_{0\leq s\leq t}$ is defined on $S_{x_0}$ (as per Hypothesis \ref{hyp:gennonautoassumptions} below). In \eqref{eq:autonomoustonon} we consider functions $g$ which are continuous up to and including the boundary of $\So$ for purely technical reasons (see proof of Proposition \ref{lemTie}). \\
In Proposition \ref{lemTie} we make some clarifications on the smoothing properties of the semigroup $\zQ_{s,t}$. To state such a lemma, we need to  properly formulate  some preliminary facts. Consider the following ``hierarchy" of operators:
\begin{align*}
\cV_{[i],t} &:= \cV_{i,t} \qquad i=0, 1 \dd d \mbox{ (defined as in } \eqref{block}\mbox{)}\\
\cV_{[\alpha \ast 0],t} & := [\cV_{[\alpha],t}, \cV_{[0],t}+\partial_t], \qquad \alpha \in \A,\\
\cV_{[\alpha \ast i],t} & := [\cV_{[\alpha],t}, \cV_{[i],t}], \qquad \alpha \in \A, i=1 \dd d\,.
\end{align*}
For each $\alpha\in \A$ we can view the vector field $(\lz,t)\mapsto \cV_{[\alpha],t}(\lz)$ as a vector field on $\R^N$, the coefficients of which  depend on time or as a vector field on $\R^N\times\R$. We can define the UFG condition for vector fields in $\R^N\times\R$ in an analogous way to Definition \ref{defufg}. In Proposition \ref{lem:VufgimpliesAdVUFG} we prove that the set of vector fields $\{\cV_{[0],t}+\partial_t,\cV_{[1],t},\ldots,\cV_{[d],t}\}$ satisfy the UFG condition on $\R^N\times\R$ provided the vector fields $\{V_0,V_1,\ldots,V_d\}$ satisfy the UFG condition on $\R^N$.

\begin{prop}\label{lem:VufgimpliesAdVUFG} 
Assume that the vector fields $\{V_0,V_1,\ldots,V_d\}$ on $\R^N$ satisfy the UFG condition at level $m$; then the vector fields $\{\partial_t+\cV_{[0],t},\cV_{[1],t},\ldots, \cV_{[d],t}\}$ satisfy the UFG condition at level $m$ when viewed as vector fields on $\mathbb{R}^N\times \R$. Moreover, for any $\alpha\in\mathcal{A}_m$,
\begin{equation}\label{eq:commutatorsofcV}
\cV_{[\alpha],t} = \Ad_{t\voperp}V_{[\alpha]}.
\end{equation}
\end{prop}

\begin{proof}[Proof of Proposition \ref{lem:VufgimpliesAdVUFG}]
The proof is deferred to Appendix \ref{app:proofsoflongtime}.
\end{proof}

Recall from Section \ref{sec:preliminaries} that the map $\lz\in S_{x_0} \mapsto \cP_tf(\lz)$ is smooth (along the directions $V_{[\alpha]}$, $\alpha \in \A_m$) for any $f\in C_b(\R^N)$. In Proposition \ref{lemTie} we show that for each fixed $s<t$ the map $\lz\in \So \mapsto \zQ_{s,t}g(\lz)$ is also smooth in the directions $\cV_{[\alpha],s}$ for any $g\in C_b(\Sbar)$ and $\alpha\in \mathcal{A}_m$. A key observation to understand the statement of Proposition \ref{lemTie} is the following one:
\be\label{keyobs}
 V\in \deln  \mbox{ and } \deln \mbox{ is invariant under the vector field } W  \quad \Rightarrow \quad   \Ad_{tW}V\in \deln.\footnote{Indeed, by the definition of invariance (see Definition \ref{basicdefditstr}), we have that $\jacobian{e^{tW}}{x}(x)$ maps $\deln(x)$ to $\deln(e^{tW}(x))$.  Therefore $\jacobian{e^{tW}}{x}(e^{-tW}(x))$ maps $\deln(e^{-tW}(x))$ to $\deln(x)$. Now $V\in \deln$, so $V(e^{-tW}(x)) \in \deln(e^{-tW}(x))$ and we have that
$
\Ad_{tW}V(x)  = \jacobian{e^{tW}}{x}(x)V(e^{-tW}(x))  \in \deln(x).
$
That is, $\Ad_{tW}V\in \deln$.}
\ee
In particular, $\cV_{j,t} \in \deln$ for every $j \in \{0 \dd d\}$. 
%With this clarification in place, the following statement holds. 
\begin{prop}\label{lemTie} 
 Assume the vector fields $\{V_0,\ldots,V_d\}$ satisfy the UFG condition and that the vector $\voperp$ is uniformly Lipschitz. Then, for any $g\in C_b(\Sbar)$, the map $(z,s)\mapsto \zQ_{s,t}g(\lz)$ is differentiable in the time variable $s$ and in the spatial directions $\cV_{[\alpha],s}$ for any $\lz\in\Sbar, t>s, \alpha\in\mathcal{A}_m$. Moreover $\zQ_{s,t}g(\lz)$ satisfies the equation
	\begin{equation}\label{eq:bKeq}
	\partial_s\zQ_{s,t}g(\lz) = -\mathcal{L}_s\zQ_{s,t}g(\lz), \quad \mbox{ for any } z\in \So, s<t.
	\end{equation} 
	Here $\mathcal{L}_s$ is the differential operator defined as 
	\begin{equation*}
	\mathcal{L}_s\psi(\lz) = \cV_{0,s}\psi(\lz) + \sum_{i=1}^d \cV_{i,s}^2\psi(\lz)\, ,
	\end{equation*}
	for $\psi:\Sbar\to \R$ sufficiently smooth. %Moreover for $g\in C_b^2(\Sbar)$ such that $\mathcal{L}_tg\in C_b(\R^N)$ for all $t\geq 0$ we have
%	\begin{equation*}
%		\partial_t\zQ_{s,t}g(\lz) = \zQ_{s,t}(\mathcal{L}_tg)(\lz).
%	\end{equation*} 	
\end{prop}

\begin{proof}[Proof of Proposition \ref{lemTie}]
The proof is deferred to Appendix \ref{app:proofsoflongtime}.
\end{proof}

%\textcolor{blue}{This is just stating the Kolmogorov equations for non-autonomous SDE, needs a reference.}

\subsection{Convergence to Equilibria}\label{sec7core}
We now turn to the asymptotic behaviour of the process $\cZ_t$. As we have already stated, we will concentrate on the case in which the solution of the ODE \eqref{ODElocall} converges. Let us define the map
\begin{align*}
\Wlim:\mathrm{Dom}(\Wlim) \subseteq\R^N & \longrightarrow \R^N\\
x & \longrightarrow \lim_{t\to\infty}e^{t\voperp}(x).
\end{align*}
Here $\mathrm{Dom}(\Wlim)$ is the set of all points $x\in \R^N$ such that the integral curve $e^{t\voperp}(x)$ converges to a finite limit as $t$ tends to $\infty$. 

\begin{hypothesis}\label{hyp:gennonautoassumptions}
	Assume the following:
	\begin{enumerate}[label=\textbf{\textup{[A.\arabic*]}}]
		\item The vector fields $\{V_0,V_1,\ldots,V_d\}$ satisfy the UFG condition. \label{item:UFGgen}
\item The vector field $\voperp$ is uniformly Lipschitz. \label{item:LipschitzVoperp}
		\item Define the measures $p_{t}^{x}$ by $p_t^{x}(A)=\cP_{t}\mathbbm{1}_A(x)$ for any Borel measurable $A\subseteq \R^N$. The family $\{p_t^{x}: t\geq 0\}$ is tight for all $x\in\mathbb{R}^N$.\label{item:gentightnessofp}
		\item Define the measures $\fq_{t}^{s,\lz}$ by $\fq_t^{s,\lz}(A)=\zQ_{s,t}\mathbbm{1}_A(\lz)$ for any Borel measurable $A\subseteq \So$. Then we require that for each $\lz\in\So$ the measures $\{\fq_t^{0,\lz}:t\geq 0\}$ are tight on $\So$; that is,  for all $\varepsilon>0$ there exists a compact set $K_\varepsilon\subseteq \So$ such that $\fq_t^{0,\lz}(K_\varepsilon)\geq 1-\varepsilon$ for all $t\geq 0$.\label{item:gentightnessofq}
		\item \label{item:OACgen} The Obtuse Angle Conditions \eqref{eq:OAC} and \eqref{eq:OAC2} are satisfied (with the understanding of Hypothesis \ref{SA} \ref{SAOAC}).
\item \label{item:conv} The initial datum $x_0$ of the SDE \eqref{SDE}  is such that the curve \eqref{ODElocall} started at $x_0$, admits a limit, i.e. there exists $\bar{x} \in \R^N$ such that $\voperp(\bar{x}) =0$ and  $e^{t\voperp}(x_0) \rar \bar{x}$ as $t \rar \infty$. 
		\item \label{item:genODEconv} 
		Assumptions on the map $\Wlim$: the domain of $\Wlim$ contains the whole manifold $\Sbar$ and  the image of $S_{x_0}$ through $\Wlim$ is all contained in a submanifold of $\deln$. More explicitly,     there exists an integral submanifold  of $\deln$, $\Sinfty$,  such that $\voperp=0$ on $\Sinfty$  and the image of $S_{x_0}$ through $\Wlim$ is all contained in $\Sinfty$,  $\Wlim(S_{x_0})\subseteq \Sinfty$.  Furthermore we assume that $\Wlim$ is a continuous map from $\Sbar\cup \curlys_{x_0}\cup \Sinfty$ into $\R^N$.
	\end{enumerate}
\end{hypothesis}

%%%%%%%%%%%%%%%%%%orca

\begin{note}\label{noteonhyplocal}\textup{
Some comments on the above assumptions, in the order in which they are stated.  }
	\begin{itemize}
\item \textup{As a general premise, observe that, for every fixed $t\geq0$,  $X_t \in \curlys_{x_0}$ if and only if $\cZ_t \in \So$. Indeed,  $X_t=e^{t\voperp} \cZ_t$ so if $\cZ_t$ is in $\So$ then in particular it is in $\curlys_{x_0}$ and $X_t$ is just obtained by moving along an integral curve of $\voperp$; hence, by construction of the manifold $\curlys_{x_0}$, $X_t$ is still in $\curlys_{x_0}$. The validity of the reverse implication can be argued similarly (using Lemma \ref{lem:boundary} and Proposition \ref{prop:SDElivesonsubmanifold} as well). As a consequence, if  $\cZ_t$ doesn't hit the boundary of $\So$ {\em in finite time} then $X_t$ doesn't hit the boundary of $\curlys_{x_0}$   {\em in finite time}.
}
\item \textup{Hypothesis \ref{hyp:gennonautoassumptions} \ref{item:gentightnessofq} implies that $\cZ_t\in \So$ almost surely, for every $t \geq 0$, i.e. it implies that $\cZ_t$ doesn't hit the boundary of $\So$ in finite time.    Indeed, assume by contradiction that there exists $t_0>0$ such that $\mathbb{P}(\cZ_{t_0}\in \partial\So) =:\varepsilon>0$. Recall $\partial\So:=\Sbar\setminus\So$. By the previous bullet point  if $\cZ_{t_0}$  belongs to $\partial\So$ then $X_{t_0}\in\partial\curlys_{x_0}$. By Proposition \ref{prop:dimcanonlydecrease} we then have that $X_t$ is in the boundary of $\curlys_{x_0}$ for any $t>t_0$. That is, 
	$$
	\mathbb{P}(X_t\in\partial\curlys_{x_0}) \geq \mathbb{P}(X_{t_0}\in\partial\curlys_{x_0}) \geq \mathbb{P}(\cZ_{t_0}\in\partial\capitals_{x_0})=\varepsilon>0\, , \quad \mbox{for any } t>t_0. \footnote{The second inequality is an inequality rather than an equality because of Lemma \ref{lem:boundary}.}
	$$
We know from \ref{item:gentightnessofp} that, given $\varepsilon$ as in the above, there exists	 a compact set  $K_{\varepsilon/2} \subseteq S_{x_0}$ such that $\mathbb{P}(\cZ_t \in K_{\epsilon/2}) = \fq_t(K_{\varepsilon/2})\geq 1-\varepsilon/2$ for every $t \geq 0$.
% If $\cZ_t$ belongs to $K_{\varepsilon/2}$ then in particular $X_t$ must belong to $\curlys_{x_0}$. 
Now using that $\curlys_{x_0}$ and $\partial\curlys_{x_0}=\overline{\curlys}_{x_0}\setminus\curlys_{x_0}$ are disjoint, for every $t>t_0$ we have
\begin{align*}	
1 &=  \mathbb{P}(X_t\in\partial\curlys_{x_0}) +\mathbb{P}(X_t\in\curlys_{x_0}) \\
	& \geq \mathbb{P}(\cZ_{t_0}\in\partial\capitals_{x_0}) + \mathbb{P}(\cZ_t\in K_{\varepsilon/2}) + \mathbb{P}(\cZ_t\in (K_{\varepsilon/2})^C)\\
& \geq  \mathbb{P}(\cZ_{t_0}\in\partial\capitals_{x_0}) + \mathbb{P}(\cZ_t\in K_{\varepsilon/2})
	\geq \varepsilon + 1- \varepsilon/2= 1+\varepsilon/2 \, ,
\end{align*}
where in the first inequality we have used the observation in the first bullet point of this note and $(K_{\varepsilon/2})^C$ denotes complement in $\curlys_{x_0}$. 
	  Hence $\varepsilon=0$, i.e. $\cZ_t$ belongs to $\So$ almost surely.
\item \textup{Hypothesis \ref{hyp:gennonautoassumptions} \ref{item:conv} is the analogous of Hypothesis \ref{hyp:nonautoassumptions}. \ref{item:convODE}.
}
	\item \textup{Hypothesis \ref{hyp:gennonautoassumptions} \ref{item:genODEconv} is slightly more complicated to explain, so we observe that it is satisfied in  the representation of the form ``ODE+SDE''  \eqref{SDEglobal}-\eqref{incondglobal} of the previous section, if  $\zeta_t=e^{tW_0}\zeta_0$ converges to some $\barz$. Indeed in that case if $x_0=(z_0, \zeta_0)$ then $S_{x_0}=\mathcal{H}_{\zeta_0}$ and $\Sinfty=\mathcal{H}_{\barz}$ (both of these manifolds are $n$-dimensional hyperplanes in $\R^{n+1}$, hence they are closed).  Moreover, for every $x=(z, \zeta)\in \curlys_{x_0}$,  $\Wlim(x)= \Wlim((z,\zeta))=(z, \barz)$, hence   the map   $\Wlim$ is  continuous on $\curlys_{x_0}$.  If $x=(z, \barz) \in \Sinfty$ then $\Wlim(x)=x$, so $\Wlim$ is continuous on $\Sinfty$ as well.  Because in this case the map $\Wlim$ is just a projection on the plane $\Sinfty$,  $\Wlim$ is continuous on $\So \cup \curlys_{x_0} \cup \Sinfty = \curlys_{x_0} \cup \Sinfty$ (the equality holding because $\So \subset \curlys_{x_0}$).  
  }
\item \textup{ By \ref{item:conv} $\voperp(\bar{x})=0$; using Lemma \ref{lemmaW0}, this implies that $S_{\bar{x}}=\curlys_{\bar{x}}$. Hence, by Lemma \ref{lemmavoonmanifold}, $\voperp(x)=0$ for every $x \in S_{\bar{x}}$. So in reality \ref{item:conv} implies that part of \ref{item:genODEconv}  where we require  $\voperp$ to vanish on the whole $S_{\bar{x}}$. }
\item \textup{If we don't make any assumptions on the map $\Wlim$, when we look at the set $\Wlim(\So)$,  it may occur that this is not a connected set and, even if it were connected, it may be contained in more than one submanifold of $\deln$ (see Example \ref{ex:grushin}). If we assume that $\Wlim$ is continuous, because $\So$ is connected then also $\Wlim(\So)$ is; for simplicity,  we are also explicitly  assuming that $\Wlim(\So)$ is contained in just one submanifold of $\deln$, the manifold $\Sinfty$.  It could  also occur that on the limit manifold $\Wlim(\So)$ we have that $\voperp(x)\neq 0$ for every $ x \in \Wlim(\So)$,  see for instance  Example \ref{ex:ODEwithmanyinvmeas}. If this is the case, then one can take such a manifold as starting manifold and apply the theory that we explain here by taking starting points on this manifold; i.e. one can sort of ``repeat the procedure" illustrated here by starting the dynamics again on that manifold. So, in conclusion one just needs to study the case in which $\voperp(x)= 0$ for every $ x \in \Wlim(\So)$. Again for simplicity, we assume $\voperp(x)\neq 0$ for every $ x \in {\Sinfty}$.}
  \item {\textup{Finally, notice that if  $\Wlim$ is well defined and continuous on $\So$ then $\Wlim$ is also a well-defined and continuous map from $\curlys_{x_0}$ to $\R^N$. We show this fact in Lemma \ref{lem:continuityofWlim} , contained in Appendix \ref{app:misc}. Notice also that $\Wlim$ is the identity when restricted to  $\Sinfty$, hence $\Wlim$ is always well defined and continuous on $\Sinfty$. What we are requiring with the last point of  Hypothesis \ref{hyp:gennonautoassumptions} is that the map should be continuous not only on each one of the manifolds $\Sbar, \curlys_{x_0}$ and $\Sinfty$, but also that it should be continuous on the union of these three sets. The reason why we need continuity also on the closure of $S_{x_0}$ is, again,  technical, see proof of Lemma \ref{lem:invarianceofmukgen}}}  
  }
	\end{itemize}
\end{note}

Before we consider the behaviour of $X_t^{(x)}$ in the  case when $e^{t\voperp}(x)$ is convergent, we must first consider the trivial case, i.e. the behaviour of the process when we start it from the ``equilibrium manifold" $\Sinfty$, where $\voperp(x)=0$. We do this in Proposition \ref{prop:limitmeasureonmanifold} below, which is the analogous of  Lemma \ref{lem:ergbarZ}.
	
\begin{prop}\label{prop:limitmeasureonmanifold}
Let Hypothesis \ref{hyp:gennonautoassumptions} \ref{item:UFGgen}, \ref{item:gentightnessofp} and \ref{item:OACgen} hold. Let $\capitals$ be an integral submanifold of $\deln$ such that $\voperp=0$ on $S$. Then there exists a unique invariant measure $\overline{\mu}^{S}$ of $\cP_t$ supported on $\overline{\capitals}$ such that
\begin{equation}\label{eq:muSbarlimit}
\lim_{t\to\infty}\cP_tf(x) = \overline{\mu}^{S}(f), \text{ for all } x\in S,f\in C_b(\R^N).
\end{equation}
Moreover the convergence is uniform on compact subsets of $S$; that is,  for every compact set $K \subseteq S$ and every $f\in C_b(\R^N)$ we have  
$$
\lim_{t \rar \infty} \sup_{x \in K}\lv \cP_tf(x) - \overline{\mu}^{S}(f) \rv =0 \,.
$$
\end{prop}
\begin{proof}[Proof of Proposition \ref{prop:limitmeasureonmanifold}]
The proof is deferred to Appendix \ref{app:proofsoflongtime} .
\end{proof}

\begin{note}
	\textup{The assumption that $\voperp=0$ on $S$ implies that, if $x \in S$,  then the map $t\mapsto\cP_tf(x)$ is differentiable,   for any  $f\in C_b(\R^N)$. Indeed, as explained in the Introduction,  in general we have that $\cP_tf$ is differentiable in the direction $\partial_t-V_0$ and in the directions contained in $\deln$ (see Appendix \ref{app:UFG}) and satisfies
\begin{equation*}
(\partial_t-V_0)\cP_tf = \sum_{i=1}^d V_i^2\cP_tf.
\end{equation*}
However if $\voperp(x)=0$ for all $x\in S$ then $V_0(x)\in \deln(x)$ for all $x\in S$ and hence $\cP_tf$ is also differentiable in the direction $V_0$ on $S$. Therefore we have that $\cP_t$ is also differentiable in time, i.e. as a map $t\mapsto\cP_tf$,   and satisfies
\begin{equation*}
\partial_t\cP_tf =V_0\cP_tf +\sum_{i=1}^d V_i^2\cP_tf.
\end{equation*}
} \hf
\end{note}

By Hypothesis \ref{item:genODEconv} $\voperp=0$ on $\Sinfty$ so we can apply Proposition \ref{prop:limitmeasureonmanifold} to the manifold $\Sinfty$ and throughout the rest of the section we shall denote by $\mulim$ the invariant measure supported on $\Sinftybar$ such that \eqref{eq:muSbarlimit} holds for all $x\in \Sinfty$.  Such a measure exists and is unique by Proposition \ref{prop:limitmeasureonmanifold}. Similarly to what we did in Section \ref{sec:non-autonomous}, equation \eqref{extensioninvmeasZ},  we shall extend this to a measure $\mu^{\Sinfty}$ defined on $\R^N$ by setting
$$
\mu^{\Sinfty}(A) = \mulim(A\cap \Sinftybar), \text{ for any Borel measurable set } A\subseteq \R^N.
$$

For any $\overline{x}\in \R^N$, let $\mathcal{I}_0(\overline{x}) = \{x\in \R^N: \Wlim(S_x) \subseteq \Sinfty\}$. The set $\mathcal{I}_0$ is contained within the basin of attraction for the measure $\mu^{\Sinfty}$. Indeed, Theorem \ref{thm:mainglobalthm} below shows that for all $x\in \mathcal{I}_0(\overline{x})$ we have that $\cP_tf(x)$ converges to $\mu^{\Sinfty}(f)$, for all $f\in C_b(\R^N)$. 

\begin{theorem}\label{thm:mainlocalthm}
	Let Hypothesis \ref{hyp:gennonautoassumptions} hold. %In particular, by Hypothesis \ref{hyp:gennonautoassumptions} \ref{item:gentightnessofp} there is an invariant measure $\mulim$ supported on $S_\infty$. 
	Let $\overline{x}\in \R^N$ be such that $\voperp(\overline{x})=0$. Then there exists an invariant measure $\mu^{\Sinfty}$ supported on $\Sinftybar$ such that for each $x_0\in \mathcal{I}_0(\overline{x})$,  and $f\in C_b(\R^N)$ we have that $\cP_tf(x_0)$ converges to $\mu^{\Sinfty}(f)$. 
\end{theorem}

\begin{proof}[Proof of Theorem \ref{thm:mainlocalthm}]
	Throughout the proof we fix an arbitrary point $x_0\in \mathcal{I}_0$. The proof is split into 3 steps.
	
	{$\bullet $ {\em Step 1}}: We first construct a tight evolution system of measures, $\{\nu_t\}_{t\geq 0}$, for the semigroup $\{\zQ_{s,t}\}_{0\leq s\leq t}$ which are supported on $\So$. This can be done be acting analogously   to what we have done in  Step 1 of the proof of Theorem \ref{thm:mainglobalthm};  in particular we may define $\nu_t:=\zQ_{0,t}^*\delta_{x_0}$.\footnote{Note that using the same argument we could define $\nu_t=\zQ_{0,t}^*\delta_{x_0}$.} Note that $\nu_t(\So)=1$; indeed  by  Note \ref{noteonhyplocal} (second bullet point) we have that $\cZ_t\in \So$ almost surely when $\cZ_0=x_0$; hence
	\begin{equation*}
	\nu_t(\So) = \zQ_{0,t}^*\delta_{x_0}(\So) = \zQ_{0,t}\mathbbm{1}_{\So}(x_0) = \mathbb{P}(\cZ_t\in \So\rvert \cZ_0=x_0)=1, \quad \text{ for every } t\geq 0.
	\end{equation*}
	%{$\bullet $ {\em Step 2}}:{\color{green} $\zQ_{s,t}g(\lz)-\nu_t(g)$ converges to zero as $t$ tends to $\infty$ for all $s\geq0, \lz\in \So$ and $g\in C_b(\R^{N})$. This can be shown by pretty much repeating the proof of Lemma \ref{lem:convtoevolutionsystem}, so we don't give details.} {\color{red} why do you need this to be here? Indeed, do you need it at all? where do you use it in all the proofs of the Lemmata of Sec 7?} \footnote{One needs to first prove the following convergence:
%	\begin{equation*}
%	\lim_{t\to\infty}\lv\zQ_{s,t}g(z)-\zQ_{s,t}g(y) \rv =0 \quad \text{ for all }z,y\in \So,  g\in C_b(\Sbar).
%	\end{equation*}
%	Integrating the above with respect to $\nu_t$ in the variable $y$ one obtains
%	\begin{equation*}
%		\lim_{t\to\infty}\lv\zQ_{s,t}g(z)-\int_{\So}\zQ_{s,t}g(y)\nu_t(dy) \rv =0.
%	\end{equation*}	
%    Finally we use that $\{\nu_t\}_t$ is an evolution system of measures to conclude the proof. }
	Moreover, analogously to Step 2 in the proof of Theorem \ref{thm:mainglobalthm}, since the family $\{\nu_t\}_t$ is tight,  there exists a diverging sequence $\{t_\ell\}_\ell$ such that $\nu_{t_\ell}$ converges weakly to some probability measure $\mu_0$ as $t_{\ell}$ tends to $\infty$. 
		
	{$\bullet $ {\em Step 2}}: By construction, the measure $\mu_0$ is a measure on $\Sbar$;  we then  consider the probability measure  $\mu_0\circ (\Wlim)^{-1}$. \footnote{Here $({\Wlim})^{-1}(A)$ denotes preimage of $A$. } The latter measure is supported on $\Sinftybar$. One needs to show that $\mu_0\circ (\Wlim)^{-1}=\mulim$. Recall that $\mulim$ is the restriction of the measure $\mu^{\Sinfty}$ to $\Sinftybar$. The proof of this fact is deferred to  Lemma \ref{lem:limofnut}. Note that this is one of the places where we use that $x_0\in \mathcal{I}_0(\overline{x})$.
	This implies that $\nu_t$ converges weakly to $\mulim\circ \Wlim$ as $t$ tends to $\infty$.
Furthermore, 	by Hypothesis \ref{item:gentightnessofp} we can take a sequence $\{t_\ell\}_{\ell}$ such that $t_\ell\nearrow\infty$ and $p_{t_\ell}^{x_0}$ converges weakly to some probability measure $\nu^{x_0}$. 
	
	{$\bullet$ {\em Step 3}}: We show that $\nu^{x_0}$ is supported on $\Sinftybar$  and,  when we restrict it   to $\Sinftybar$, we have $\nu^{x_0}\vert _{\Sinftybar}=\mu_0\circ (\Wlim)^{-1}$. Lemma \ref{lem:limofpt} is devoted to proving this fact. Therefore, by Step 2 and the definition of $\mu^{\Sinfty}$ we have that $$\nu^{x_0}=\mu^{\Sinfty}.$$
	This implies that $p_t^{x_0}$ converges weakly to $\mulim$ as $t$ tends to $\infty$ for any $x\in \curlys_{x_0}$, that is, for every $f\in C_b(\R^N)$, $\cP_tf(x_0)$ converges to $\mu^{\Sinfty}(f)$ as $t$ tends to $\infty$. 
\end{proof}
%%%%%%%%%%%%%%%%%%%%%%
%%%%%%%%%%%%%%%%%%%%%%%%%%%%%%%
%%%%%%%%%%%%%%%%%%%%%%%%%%%%%%%%%%%%%%%5555
%%%%%%%%%%%%%%%%%%%%%%%%%%%%%%5
We now give a one dimensional example which satisfies all the assumptions we have made in this section. In particular, this example fits our framework in a non-trivial way as it exhibits many invariant measures. 
\begin{example}
\textup{ 
	Consider the SDE
\begin{equation*}
dZ_t^z=\sin(Z_t^z)dt+\sqrt{2}(1-\cos(Z_t^z))\circ dB_t, \quad Z_0=z,  \quad Z_t \in \R, 
\end{equation*}
where $(B_t)_{t\geq0}$ is a one-dimensional Wiener process.
In this case $V_0=\sin(z)\partial_z$, $V_1=(1-\cos(z))\partial_z$ and we have
\begin{equation*}
[V_1,V_0] = [(1-\cos(z))\partial_z,\sin(z)\partial_z] = \cos(z)(1-\cos(z))\partial_z-\sin(z)^2\partial_z=-V_1.
\end{equation*}
Therefore the vector fields $V_0, V_1$  satisfy the UFG condition; the above also shows that the obtuse angle condition \eqref{eq:OAC} is satisfied, with $\lambda_0=1$. Moreover, it is easy to show that  the function $(V_1\cP_t f)(x)$ decays exponentially fast in time, i.e. $\lambda_0$ is big enough that \eqref{eq:OAC} implies an estimate of the type \eqref{eq:graddecayest} for the fields $V_1$.   Because the coefficients of the equation are bounded the estimate is uniform on the whole real line, see \cite[Proposition 3.1, Proposition 3.4 and Theorem 4.2]{CrisanOttobre} alternatively by a direct calculation, see \cite[Example 4.4]{CrisanOttobre}. 
Since $V_0$ and $V_1$ both vanish whenever $z\in 2\pi\mathbb{Z}$ we have that the point measures $\delta_{2n\pi}$ are invariant measures for any $n\in\mathbb{Z}$. However there also exist invariant measures supported on $(2n\pi, 2(n+1)\pi)$ for any $n\in \mathbb{Z}$. 
Indeed let
\begin{equation*}
\rho_n(z):=\frac{\exp\left(-\frac{1}{1-\cos(z)}\right)}{C(1-\cos(z))} \mathbbm{1}_{(2n\pi, 2(n+1)\pi)}(z)
\end{equation*}
where $C$ is the normalization constant and $\mathbbm{1}_{(2n\pi, 2(n+1)\pi)}(z)$ is the characteristic function of the interval $[2n\pi, 2(n+1)\pi)]$. By direct calculation one can verify that, for every $n \in \mathbb{Z}$,   $\rho_n(z)$ satisfies  the stationary Fokker-Planck equation $\mathcal{L}^*\rho_n=0$, where
    \begin{equation*}
    \mathcal{L}^*\rho_n(z) = -\partial_z(\sin(z)\rho_n(z)) +\partial_z\left[(1-\cos(z))\partial_z\left((1-\cos(z))\rho_n(z)\right)\right] \,.
    \end{equation*}
Notice that if $X_0 \in [2n\pi, 2(n+1)\pi]$ (for some fixed $n \in \mathbb{Z}$) then $X_t \in [2n\pi, 2(n+1)\pi]$ for every $t\geq 0$. However, even if we restrict to one of the intervals $[2n\pi, 2(n+1)\pi]$, the process still admits three invariant measures on each one of such intervals. 
}
\hf
\end{example}
%%%%%%%%%%%%%%%%%%%%%
%%%%%%%%%%%%%%%%%%%%%%%
%%%%%%%%%%%%%%%%%%%%%%%%%%%5
\begin{example}[Example \ref{ex:ODEwithmanyinvmeas} continued]\label{ex:ODEwithmanyinvmeasctd}
\textup{ Recall that in this example $V_0=\sin\zeta \pa_{\zeta}-kz \pa_z$ and $V_1=\zeta\pa_z$. While $V_0$ is smooth,   $\voperp$ is not continuous.  Indeed,  for $\zeta\neq 0$ $\voperp(z,\zeta) = -\sin(\zeta)\partial_\zeta$, however for $\zeta=0$ $\voperp(z,0)=V_0(z,0)=-kz\partial_z$.
}
\hf
\end{example}

We conclude this section by stating and proving Theorem \ref{prop:invmeasureforODE} below. In order to state it, let us
define the following equivalence relation on  $\R^N$:
$$
x \sim y \quad \Leftrightarrow \quad x \in S_y \,.
$$
As customary, we denote by $[x]$ the equivalence class of $x$ under the equivalence relation $\sim$. Note that by Lemma \ref{thm:intcurvepreservemanifolds}, if $x \sim y$ then also $e^{t\voperp}x \sim e^{t\voperp}y$, therefore the flow map 
\be\label{flowmap}
[x] \longrightarrow [e^{t\voperp} x]=:e^{t\voperp}[x]
\ee
is well defined. Let now $q$ be the map $q: \R^N \rightarrow 
\R^N/ \sim$, defined as $q(x)=[x]$. If we endow the quotient set $\R^N/ \sim$ with the $\sigma$-algebra $\{E \subseteq \R^N/ \sim \,s.t.\, q^{-1}(E) \mbox{ is a Borel set of }\R^N\}$, then $q$ is a measurable map. If $\mu$ is a probability measure on $\R^N$,  we define the pullback measure $\tilde{\mu}$ on $\mathbb{R}^N/\sim$ as $\tilde{\mu}(E)=\mu(q^{-1}(E))$ for all $E\subseteq \mathbb{R}^N/\sim$.
%%%%%%%%%%%%%%%%%%%%
%%%%%%%%%%%%%%%%%%%%%%%%
\begin{theorem}\label{prop:invmeasureforODE} Consider the SDE \eqref{SDE} and the associated semigroup $\cP_t$ and assume that the vector fields $V_0 \dd V_d$ satisfy the UFG condition. 
If $\mu$ is an invariant measure for  $\cP_t$,  then $\tilde{\mu}$ is an invariant measure for the flow map \eqref{flowmap}. 
\end{theorem}
\begin{proof}[Proof of Theorem \ref{prop:invmeasureforODE}]
	Denote by $B_b(\R^N/\sim;\R)$ to be the set of all bounded and measurable functions $f:\R^N/\sim\to\R$. If $f\in B_b(\mathbb{R}^N/\sim; \R)$, then $f \circ q \in B_b(\R^N; \R^N)$, i.e. $f\circ q$ is a bounded and measurable function mapping from $\R^N$ to $\R^N$. By the definition of invariant measure, we have
	\begin{align*}
	\int_{\mathbb{R}^N/\sim}f([x]) \tilde{\mu}(d[x]) &= \int_{\mathbb{R}^N} f(q(x)) \mu(dx)\\
	&=  \int_{\mathbb{R}^N} \left( \cP_t(f\circ q)\right)(x) \mu(dx)\\
	&= \int_{\mathbb{R}^N} \mathbb{E}_x[f(q(X_t^x))] \mu(dx) \,.
	\end{align*}
	Let us now look more closely at the  expected value on the right hand side of the above: for any bounded and measurable function $h$ we can write
	\begin{align*}
	\int_{\mathbb{R}^N} \mathbb{E}_x[h(X_t^x)] \mu(dx) & = 
	\int_{\R^N} \mu(dx) \int_{\R^N}h(y)\mathbb{P}_x (X_t^x \in dy)\\
	& =\int_{\R^N} \mu(dx) \int_{\overline{\capitals_{e^{t\voperp}\!\!\! x}}}h(y)\mathbb{P}_x (X_t^x \in dy)\\
	&=\int_{\R^N} \mu(dx) \int_{\capitals_{e^{t\voperp}\!\!\! x}}h(y)\mathbb{P}_x (X_t^x \in dy)+\int_{\R^N} \mu(dx)
	\int_{\pa \capitals_{e^{t\voperp}\!\!\!x}}h(y)\mathbb{P}_x (X_t^x \in dy)\, ,
	\end{align*}
	where the second equality follows from Proposition \ref{prop:SDElivesonsubmanifold}.  Now the second term in the above vanishes by Proposition \ref{thm:meszerosetsunderinvmeas} (easy to prove for positive $h$, if $h$ is not positive just split into positive and negative part).  Indeed if $X_t^x \in \pa \capitals_{e^{t\voperp}\!\!x}$ then 
	$X_t^x \in \pa\curlys_x$ (see Lemma \ref{lem:boundary} for a proof of this fact).  Putting everything together we can write
	\begin{align*}
	\int_{\mathbb{R}^N/\sim}f([x]) \tilde{\mu}(d[x]) &= \int_{\mathbb{R}^N} \mathbb{E}_x[f(e^{t\voperp}([x]))]\mu(dx)\\
	&=\int_{\mathbb{R}^N} f(e^{t\voperp}(x))\mu(dx)\\
	&=\int_{\mathbb{R}^N/\sim} f(e^{t\voperp}([x]))\tilde{\mu}(d[x]), 
	\end{align*}
	where the penultimate equality follows from the fact that the object on the second line is completely deterministic and the last equality holds by the definition of the measure $\tilde{\mu}$.  This concludes the proof. 
\end{proof}

\begin{comment}
As a direct consequence we have the following. 
\begin{corollary}\label{cor:div}
Let $x_0$ be the initial datim of the SDE \eqref{SDE} and of the   ODE \eqref{ODElocall}. If $\zeta_t =e^{t\voperp x_0}$ diverges as $t \rar \infty$, then the flow map \eqref{flowmap} does not admit an invariant measure. Therefore,  the process $X_t^{(x_0)}$ solution of \eqref{SDE} does not admit an invarinat measure either. \footnote{Notice that this only means that $X_t^{(x_0)}$ does not admit an invariant measure. It may well be that, started from a different initial condition, the SDE would admit a stationary state. }
\end{corollary}
\end{comment}

\section{Existence of a density}\label{sec:8}

Analogously to what we did for the study of the long-time behaviour, we split this section into two subsections. That is, in Section \ref{sec:densityinSDEplusODE} we  consider the setting of Section \ref{sec:non-autonomous} and study SDEs of the form \eqref{SDEglobal}-\eqref{incondglobal}. In Section \ref{sec:densitygen} we consider the general UFG-case. This section makes use of several notions from Malliavin calculus, we will recall only some basic facts and  refer the reader to \cite{Nualart} for  more detailed background material. 
\begin{note}\label{note:compareBismut2}
\textup{
The techniques of proof that we use here are not different from those in \cite{Nualart, Watanabe, Bismut}; however our assumptions are significantly more general, so we can't just quote results and some proofs need to be re-sketched. Moreover, the results of this section can be seen as a generalization of the results of \cite[Section 5.2]{Bismut} to the UFG case. We indeed recall that in \cite[Section 5.2]{Bismut} the author assumes the validity of the H\"ormander condition (HC), which we have recalled in Section \ref{sec:preliminaries}, as well as boundedness of the coefficients of the SDE. Here we remove both such assumptions. We emphasize that, by a geometric point of view, imposing the validity of the H\"ormander condition amounts to assuming that the distribution $\delnn$ is equal to $\R^N$ at every point. } \hf 
\end{note}

Let $\mathbb{D}^{k,p}\subseteq L^p(\Omega)$ denote the Malliavin Sobolev space, that is the domain of the $k$th order Malliavin derivative in the space $L^p(\Omega)$. We also define the space
\begin{equation*}
\mathbb{D} = \bigcap_{p>1, k\in\mathbb{N}} \mathbb{D}^{k,p}.
\end{equation*}
%From now on we shall fix a time $t>0$, and let $X=X_t$ and $S=S_{e^{t\mathbf{W_0}}(x_0)}$. 
We shall denote by $\mathbb{D}'$ the dual space of $\mathbb{D}$, that is the space of all continuous linear maps from $\mathbb{D}$ to $\mathbb{R}$.  % domain of $k^{th}$ order Malliavin derivative operator $D$. 
Let us recall the following lemma, which is quoted from \cite[Theorem 2.2.1]{Nualart}.

\begin{lemma}\label{lem:existenceofmallmatrix}
	Fix $T>0$, let $\{X_t\}_{t\in[0,T]}$ denote the solution of the SDE \eqref{SDE} and assume that $V_0,V_1,\ldots,V_d$ are smooth vector fields which are globally Lipschitz. Then $X^i_t$ belongs to $\mathbb{D}^{1,p}$ for any $t\in[0,T], p\geq 1$ and $i=1,\ldots,N$. Moreover, for all $0 \leq t\leq T,p\geq 1$
	\begin{equation*}
	\sup_{0\leq r \leq t} \mathbb{E}\left[\sup_{r\leq s\leq T} \lvert D_r^j X^i_s\rvert^p\right]< \infty
	\end{equation*}
	and the Malliavin derivative $D_r^jX^i_t$ satisfies the following SDE, 
	\begin{equation}\label{eq:MalliavinderivativeSDE}
	D_r^jX_t^i = V_j^i(X_r) +\sum_{k=1}^{N}\int_r^t \partial_{x^k} V_0^i(X_s) D_r^j(X_s^k) ds +\sqrt{2}\sum_{\ell=1}^d\sum_{k=1}^{N}\int_r^t \partial_{x^k} V_\ell^i(X_s) D_r^j(X_s^k) \circ dW_s^\ell \, ,
	\end{equation}
for every $r\leq t$. 
\end{lemma}

Here we use the notation $D^k$ to denote the Malliavin derivative operator with respect to the Brownian motion $B^k$.\footnote{Note that $D^k$ denotes the 1st order Malliavin derivative with respect to the $k$th Brownian motion and is not to be confused with the $k$th-order Malliavin derivative.} Define the Malliavin matrix $\mathscr{M}_t=(\mathscr{M}_t^{ij})_{i,j=1}^{N}$ to be
\begin{equation*}
\mathscr{M}_t^{ij} = \sum_{k=1}^d\int_0^t D_s^k(X_t^i) D_s^k(X_t^j) ds. 
\end{equation*}
Again by \cite[Section 2.3]{Nualart} we can rewrite the Malliavin matrix in terms of the Jacobian matrix $J_t:=\frac{\partial X_t}{\partial x_0}$, details can be found in \cite[Section 2.3]{Nualart}. There it is also shown that $J_t$ is an invertible matrix and that the following holds
\begin{equation*}
\mathscr{M}_t = J_t\left(\sum_{k=1}^d\int_0^t J_s^{-1}V_k(X_s) V_k(X_s)^T(J_s^{-1})^T ds\right) J_t^T = J_t\mathscr{C}_t J_t^T
\end{equation*}
where the matrix $\mathscr{C}_t$ is the reduced Malliavin covariance matrix defined as
\begin{equation*}
\mathscr{C}_t = \sum_{k=1}^d\int_0^t J_s^{-1}V_k(X_s) V_k(X_s)^T(J_s^{-1})^T ds.
\end{equation*}

\subsection{Existence of a density on a suitable hyperplane}\label{sec:densityinSDEplusODE}

In this section we consider the SDE \eqref{SDEglobal}-\eqref{incondglobal}. We shall also assume Hypothesis \ref{hyp:nonautoassumptions} \ref{item:smoothnessassumption}, which states that the set of vector fields $\{V_{[\alpha]}(\lz_0,\zeta_0): \alpha \in \A_m\}$ span the $n$-dimensional hyperplane $\mathcal{H}_{\zeta_0}:=\{x=(\lz,\zeta):\zeta=\zeta_0\}$ for all $(\lz_0,\zeta_0)\in \mathbb{R}^{n}\times\R$. In this setting it is clear that the law of $X_t=(Z_t,\zeta_t)$ does not admit a density with respect to Lebesgue measure on $\R^{n+1}$; indeed for each fixed $t$, $\zeta_t$ is a deterministic point which implies that $\mathbb{P}_x\left(X_t\in\R^n\times\{\zeta_t\}\right)=1$ while $\R^n\times\{\zeta_t\}$ is a null set with respect to Lebesgue measure on $\R^{n+1}$. We prove that for every fixed $t\geq 0$ the law of the random variable $Z_t$ admits a density with respect to Lebesgue measure on $\R^n$. In terms of the process $X_t$ this implies that the law of  $X_t$ admits (for every fixed $t\geq 0$) a density with respect to the Lebesgue measure on the hyperplane $\mathcal{H}_{\zeta_t}:=\{x=(\lz,\zeta):\zeta=\zeta_t\}$. Moreover, since  from Section \ref{sec:non-autonomous} $X_t \in \mathcal{H}_{\zeta_t}$ almost surely, we have that $\mathcal{H}_{\zeta_t}$ is the maximal manifold such that $X_t$ admits a density with respect to the volume element on such a manifold. \footnote{Throughout our discussion we need to fix a canonical reference measure on
the manifold. Here, and subsequently, when we refer to the volume element on
a submanifold $M\subset 
%TCIMACRO{\U{211d} }%
%BeginExpansion
\mathbb{R}
%EndExpansion
^{N}$ we mean the measure on $M$ which is determined from the Riemannian
density associated to the induced Riemannian metric on $M$.}

To prove that the law of $Z_t$ admits a density we shall follow the same strategy of \cite[Section~2.3]{Nualart}. Note that by Hypothesis \ref{SA} and Lemma \ref{lem:existenceofmallmatrix} for each $t\geq 0$ and $i\in\{1,\ldots,n\}$ we have that $Z_t^i$ and $\zeta_t$ belong to $\mathbb{D}^{1,p}$ for all $p\geq 1$. First we note that the solution $X_t=(Z_t,\zeta_t)$ admits a Malliavin derivative. 

\begin{lemma}\label{lem:formofmalliavinmatrix}
	Let $\mathscr{M}_t$ denote the Malliavin matrix corresponding to the solution $X_t=(Z_t,\zeta_t)$ of the SDE \eqref{SDEglobal} - \eqref{incondglobal}. Then $\mathscr{M}_t$ has the form 
	\begin{equation}\label{eq:formofmalliavinmatrix}
	\mathscr{M}_t = \left(\begin{array}{cc}
	M_t & 0\\
	0 & 0\\
	\end{array}\right)
	\end{equation}
	where the matrix $M_t$ is the Malliavin matrix corresponding to $Z_t$.
\end{lemma}

\begin{proof}[Proof of Lemma \ref{lem:formofmalliavinmatrix}]
The proof is deferred to Appendix \ref{sec:proofsofMalliavinsec}.
\end{proof}

In \cite{{Nualart}} it is shown that if the Malliavin matrix is invertible then the law of $X_t$ admits a density on $\mathbb{R}^{n+1}$. We can see from \eqref{eq:formofmalliavinmatrix} that the matrix is not invertible; however we show that the Malliavin matrix $M_t$ corresponding to $Z_t$ is invertible almost surely and hence the law of $Z_t$ admits a density on $\R^n$, for every fixed $t>0$.

\begin{prop}\label{prop:malliavinmatrixinv}
	The reduced Malliavin covariance matrix $\mathscr{C}_t$ corresponding to the solution $X_t=(Z_t,\zeta_t)$ of the SDE \eqref{SDEglobal} - \eqref{incondglobal} is of the form
	\begin{equation*}\label{eq:malliavinmatrixform}
	\mathscr{C}_t=\left(\begin{array}{cc}
	C_t & 0 \\
	0   & 0
	\end{array}\right),
	\end{equation*}
	where $C_t$ is a random $n\times n$ symmetric matrix. Moreover, if we assume Hypothesis \ref{hyp:nonautoassumptions} \ref{item:smoothnessassumption} holds then $C_t$ is invertible $\mathbb{P}$-almost surely.
\end{prop}

\begin{proof}[Proof of Proposition \ref{prop:malliavinmatrixinv}]
The proof is deferred to Appendix \ref{sec:proofsofMalliavinsec}.
\end{proof}

\begin{theorem}\label{thm:denistyofZt}
	Assume Hypothesis \ref{hyp:nonautoassumptions} \ref{item:smoothnessassumption} and let $\{Z_t\}_{t\geq 0}$ be the solution of \eqref{SDEglobal}. Then the law of $Z_t$ is absolutely continuous with respect to the Lebesgue measure on $\R^n$. 
\end{theorem}

\begin{proof}[Proof of Theorem \ref{thm:denistyofZt}]
	Note the Malliavin matrix corresponding to $Z_t$ is $M_t$ which is invertible, indeed $M_t=J_tC_tJ_t^T$ and $C_t$ is invertible by Proposition \ref{prop:malliavinmatrixinv} therefore $M_t$ is invertible since the product of invertible matrices is invertible. By \cite[Theorem 2.1.2]{Nualart} we have that the law of $Z_t$ is absolutely continuous with respect to Lebesgue measure on $\R^n$, for each $t>0$.
\end{proof}

\subsection{Existence of a density on integral submanifolds}\label{sec:densitygen}

We now return to studying the general UFG-case. As in the previous section we cannot expect that the law of $X_t$ will in general admit a density with respect to Lesbegue measure on $\R^N$ and we will instead show that the law of $X_t$ admits a density with respect to the volume element on a suitable manifold. Indeed, we shall show that the law of $X_t$ admits a density with respect to the volume element on $\St$. Note that by Proposition \ref{prop:SDElivesonsubmanifold} we have $X_t^{(x)} \in \overline{S_{e^{t\voperp}(x)}}$ almost surely. In this section we shall assume Hypothesis \ref{hyp:gennonautoassumptions} \ref{item:UFGgen} and that $X_t$ cannot hit the boundary of the integral manifold ${S_{e^{t\voperp}(x)}}$, that is $X_t^{(x)} \in S_{e^{t\voperp}(x)}$, almost surely. In the first and second comment in Note \ref{noteonhyplocal} it is shown that under Hypothesis \ref{hyp:gennonautoassumptions} \ref{item:gentightnessofq} implies that $X_t$ cannot hit the boundary of the maximal integral submanifold.

Recall from Section \ref{sec:longtimebehaviour} the process $\{\cZ_t\}_t$ defined by \eqref{eq:Zdef}. Since $e^{-t\voperp}$ is a diffeomorphism the law of $X_t$ admits a density with respect to the volume element on $\St$ if and only if the law of $\cZ_t$ admits a density with respect to the volume element on $\So$. Let $\mathcal{V}_{[\alpha],t}$ be defined as in \eqref{eq:commutatorsofcV}, then recall that the process $\{\cZ_t\}_{t\geq 0}$ satisfies the SDE \eqref{eq:SDEnonauto}. 
%Now $\cZ_t$ is the solution of a time inhomogeneous SDE defined on a manifold $\So$ whose coefficients are time dependent vector fields. Note that $\So$ is an integral manifold of $\deln$ and hence the tangent space of $\So$ at the point $\lz$ is given by $\deln(\lz)$. Recall from \eqref{keyobs} that $\cV_{[\alpha]} \in \deln$ and moreover for each $t\in \R$ the vector fields $\cV_{[\alpha]}(\lz)$ span $\deln(\lz)$. 
% In particular, for each $\lz\in \So$ there exists $\alpha_1=\alpha_1(\lz),\ldots, \alpha_n=\alpha_n(\lz)\in\mathcal{A}_m$ (recall $n=\dim(\So)$) such that $V_{[\alpha_1]}(\lz),\ldots, V_{[\alpha_n]}(\lz)$ are linearly independent. %such that
%%$$
%%\mathrm{span}(V_{[\alpha_1]}(x),\ldots, V_{[\alpha_n]})(x) = \deln(x) = \tangsp{\So}{x}.
%%$$
%Note that as $\lz$ varies the choice of $\alpha_1,\ldots,\alpha_n$ may vary. Recall that $\cV_{[\alpha]}=\Ad_{t\voperp}V_{[\alpha]}$ by Proposition \ref{lem:VufgimpliesAdVUFG}. Assume there exist $\lambda_1,\ldots,\lambda_n\in\R$ such that
%\begin{equation*}
%\sum_{i=1}^N \lambda_i \cV_{[\alpha_i]}(\lz,t) =0.
%\end{equation*}
%Then we have that
%\begin{equation*}
%\Ad_{t\voperp}\left(\sum_{i=1}^N \lambda_i V_{[\alpha_i]} \right)(\lz)=0.
%\end{equation*}
%Now applying $\Ad_{-t\voperp}$ we have
%\begin{equation*}
%\sum_{i=1}^N \lambda_i V_{[\alpha_i]}(\lz)=0.
%\end{equation*}
%Since $V_{[\alpha_1]}(\lz),\ldots, V_{[\alpha_n]}(\lz)$ are linearly independent we must have that $\lambda_1=\ldots=\lambda_n=0$, that is the vector fields $\cV_{[\alpha_1]}(\lz,t),\ldots, \cV_{[\alpha_n]}(\lz,t)$ for all $t\geq 0$.
Now we wish to apply \cite[Theorem 3.4]{Schiltz} to show that the law of $\{\cZ_t\}_{t\geq0}$ admits a density with respect to the volume measure on $\So$. However, as noted in \cite{Cattiaux}, there is a mistake in the proof of \cite[Theorem 3.4]{Schiltz}, in particular the form of the H\"ormander condition  given by \cite[Assumption (H)]{Schiltz} is not sufficient for the conclusions of \cite[Theorem 3.4]{Schiltz} to hold. More precisely, they rely upon \cite[Theorem 1.1.3]{Florchinger} to show that  \cite[Assumption (H)]{Schiltz} implies a suitable integration by parts formula, which is shown to be incorrect by \cite{Cattiaux}. However under our conditions there is an integration by parts formula as shown in \cite[Section 3]{Nee}. Therefore we may use the strategy given in \cite{Schiltz} and the results of \cite{Nee} to prove that the law of $\cZ_t$ admits a density with respect to the volume measure on $\So$.

A vital tool for this argument is the integration by parts formula proved in \cite[Theorem 3.10]{Nee}; namely,  for $\Phi\in \mathbb{D}$ and $\alpha_1,\ldots,\alpha_M\in \mathcal{A}_m$ we have
\begin{equation*}\label{eq:IBPformula}
\mathbb{E}_x\left[\Phi V_{[\alpha_1]}\ldots V_{[\alpha_M]}f(X_t)\right]= t^{\frac{-\lVert\alpha_1\rVert - \ldots -\lVert\alpha_M\rVert}{2}} \mathbb{E}_x[\Phi_{\alpha_1,\ldots,\alpha_M}(t,x)f(X_t)], \quad \text{ for any } f\in C_V^\infty(\R^N),
\end{equation*}
for some random variable $\Phi_{\alpha_1,\ldots,\alpha_M}(t,x)$. By taking $f=g\circ e^{-t\voperp}$ we have
\begin{equation}\label{eq:IBPformulaZ}
\mathbb{E}_x\left[\Phi \cV_{[\alpha_1],t}\ldots \cV_{[\alpha_M],t}g(\cZ_t)\right]= t^{\frac{-\lVert\alpha_1\rVert - \ldots -\lVert\alpha_M\rVert}{2}} \mathbb{E}_x[\Phi_{\alpha_1,\ldots,\alpha_M}(t,x)g(\cZ_t)], \quad \text{ for any } g\in C_V^\infty(\R^N).
\end{equation}

Let us denote by $\mathcal{E}(\So)$ the space of all distributions (in this sentence distribution is meant in an analytic sense) on $\So$ with compact support. Recall that for any smooth function $f$ we can view this as a member of $\mathcal{E}(\So)$, denoted $F_f$, by setting 
\begin{equation*}
\langle F_f, \phi\rangle = \int_{S} f(x) \phi(x) \lambda_{\So}(dx), \quad \mbox{ for any } \phi\in C_c^\infty(\So)
\end{equation*}
where $\lambda_{\So}$ denotes the volume measure on $\So$.

\begin{lemma}
	Assume that $\cZ_t$ satisfies \eqref{eq:IBPformulaZ}. Then there exists a map $\Psi_t:\mathcal{E}(\So)\to \mathbb{D}'$ with the following properties
	\begin{enumerate}
		\item If $f\in C_c^\infty(\So)$ then $\Psi_t(f)=f(\cZ_t)$. Note that $f(\cZ_t)$ is identified as an element in $\mathbb{D}'$ by setting $\langle f(\cZ_t), G\rangle=\mathbb{E}[f(\cZ_t)G]$ for any $G\in \mathbb{D}$.
		\item The map $\Psi_t$ is continuous as a map from $\mathcal{E}(\So)$ to $\mathbb{D}'$. 
	\end{enumerate}
\end{lemma}
For a proof see \cite[Proposition 2.1]{Taniguchi}. 

Now we shall state some properties of the map $\Psi$, as proven in \cite[Proposition 2]{Watanabe}.

\begin{prop}
	Fix $t>0$ and let $\cZ_t$ be such that the map $\Psi_t$ is well defined for every $f\in \mathcal{E}(\So)$. Then let $I$ be some open set 
	\begin{enumerate}
		\item If $I\ni s\mapsto F_s$ is continuous (continuously differentiable), then $I\ni s\mapsto \Psi_t(F_s)$ is continuous (resp. continuously differentiable). In particular, for every $G\in\mathbb{D}$ the map $I\ni s\mapsto \langle \Psi_t(F_s),G\rangle$ is continuous and respectively continuously differentiable and 
		\begin{equation*}
		\left\langle \Psi_t\left(\frac{dF_s}{ds}\right), G\right\rangle = \frac{d}{ds}\left\langle \Psi_t(F_s), G\right\rangle.
		\end{equation*}
		\item If $I\ni s\mapsto F_s$ is continuous then for every $G\in\mathbb{D}$
		\begin{equation*}
		\left\langle \Psi_t\left(\int_I F_s ds\right),G\right\rangle=\int_I\langle  \Psi_t(T_s),G\rangle ds
		\end{equation*} 
		where $\int_I T_s ds$ is a tempered distribution (here distribution is meant in an analytic sense) and is defined by $\langle\int_I T_s ds, \phi\rangle = \int_I \langle T_s, \phi\rangle ds$.
	\end{enumerate}
\end{prop}

%\begin{proof}
%	Since the map $T\mapsto T(X)$ is continuous, if $s\mapsto T_s$ is continuous then $s\mapsto T_s(X)$ is also continuous. Similarly if $s\mapsto T_S$ is differentiable then
%	\begin{equation*}
%	\lim_{h\to 0}\frac{T_{s+h}-T_s}{h} =\frac{dT_s}{ds}
%	\end{equation*}
%	where the limit is in $\mathcal{E}(S)$. Now since the map $T\mapsto T(X)$ is continuous we have
%	\begin{equation*}
%	\lim_{h\to 0}\frac{T_{s+h}(X)-T_s(X)}{h} =\frac{dT_s}{ds}(X).
%	\end{equation*}
%	
%	Now we shall consider (2). Let $s\in I\mapsto T_s$ be continuous and take $f_{s,n}\in C_c^\infty(S)$ such that $f_{s,n}\to T_s$ as $n$ tends to $\infty$. Now by definition
%	\begin{equation*}
%	\left\langle \left(\int_I f_{s,n} ds\right)(X),G\right\rangle=\int_I\langle  f_{s,n} (X),G\rangle ds.
%	\end{equation*} 	
%	Letting $n$ tend to $\infty$ we have the result.
%\end{proof}

We can show that the law of $\cZ_t$ admits a density.
\begin{prop}
	Assume Hypothesis \ref{hyp:gennonautoassumptions} \ref{item:UFGgen}, and  assume that $X_t^{(x_0)}\in \St$ almost surely. %\textcolor{red}{Write assumptions better!} Assume that $\cZ_t$ satisfies \eqref{eq:IBPformulaZ} and that for each point $\lz\in \So$ and $t>0$ the tangent space $\tangsp{\So}{\lz}$ at $\lz$ is spanned by the vector fields $\cV_{[\alpha],t}(\lz)$ for $\alpha\in\mathcal{A}_m$. 
   Then for each $t>0$ the law of $\cZ_t^{(x_0)}$ admits a density with respect to the volume element on $\So$.
\end{prop}

\begin{proof}
	Note that the map $x\mapsto \delta_x$ is smooth, moreover its (weak) derivative $\frac{d}{x^i}\delta_x$ is given by $D_i\delta_x$, where $D_i\delta_x$ is defined by $\langle D_i\delta_x,\phi\rangle =-\partial_{x^i}\phi(x)$ for all $\phi$. Therefore $\Psi_t(\delta_x)$ is smooth and in particular $p(x):=\langle \Psi_t(\delta_x),1\rangle$ is smooth. It remains to show that $p(x)$ is the density of the law of $X_t$. Take $\phi\in C_c^\infty(S)$ then
	\begin{align*}
	\int_S \phi(x) p(x) \lambda_{\So}(dx) &= \int_S \phi(x) \langle \Psi(\delta_x),1\rangle \lambda_{\So}(dx)\\
	&= \langle \Psi_t\left(\int_{\So} \phi(x)  \delta_x \lambda_{\So}(dx)\right),1\rangle.
	\end{align*}
	Now for $f\in C_c^\infty(\So)$ we have 
	\begin{equation*}
	\langle \int_{\So} \phi(x)  \delta_x \lambda_{\So}(dx) ,f\rangle =  \int_{\So} \phi(x)  \langle\delta_x,f\rangle \lambda_{\So}(dx) = \int_S \phi(x) f(x) \lambda_{\So}(dx).
	\end{equation*}
	Therefore $\int_{\So} \phi(x)  \delta_x \lambda_{\So}(dx) = F_\phi$, and in particular $\Psi_t\left( \int_{\So} \phi(x)  \delta_x \lambda_{\So}(dx)\right) = \phi(\cZ_t)$. Now we have
	\begin{equation*}
	\int_S \phi(x) p(x) \lambda_S(dx) = \langle \phi(\cZ_t),1\rangle = \mathbb{E}[\phi(\cZ_t)].
	\end{equation*}
\end{proof}

\begin{theorem}\label{thmsec8}
	Assume the vector fields $V_0,V_1,\ldots,V_d$ are uniformly Lipschitz, satisfy the UFG condition and assume that $X_t^{(x_0)}\in\St$ almost surely. \footnote{As we have already mentioned, the latter fact follows for example from assuming Hypothesis \ref{hyp:gennonautoassumptions} \ref{item:gentightnessofq}. } Then for each $t>0$ the law of $X_t^{(x_0)}$ admits a density with respect to the volume element on $\St$.
\end{theorem}

{\bf Acknowledgments}. P. Dobson was supported by the Maxwell Institute Graduate School in Analysis and its
Applications (MIGSAA), a Centre for Doctoral Training funded by the UK Engineering and Physical
Sciences Research Council (grant EP/L016508/01), the Scottish Funding Council, Heriot--Watt
University and the University of Edinburgh. This work was partially supported by the grant 346300 for IMPAN from the Simons Foundation and the matching 2015--2019 Polish MNiSW fund.
\appendix
\numberwithin{equation}{section}

\section{Some technical results}\label{AppendixA}
We gather in this appendix some auxiliary results. In particular, Appendix \ref{app:topology} contains background material about the topology of the orbits of finitely generated smooth distributions. Appendix \ref{app:UFG} reports some known smoothing results on UFG semigroups, which are often used in the proofs of Appendix \ref{sec:mainproofs}. Appendix \ref{app:misc} contains precise statements and proofs of further technical facts which would have been cumbersome (and detracting from the main line of thought) if presented in the main body of the work. 

\subsection{Topology of orbits} \label{app:topology}

Here we give a brief justification of the reason why we make the standing assumption \ref{SAtop}. In short, assuming that the manifold topology of the manifolds $\curlys$ is the Euclidean topology is equivalent to assuming that such manifolds are embedded manifolds. In full generality, as explained in \cite[page 78]{Isidori}, elements of a global partition induced by distributions which enjoy the integral manifold property are immersed manifolds. We briefly explain the difference between an embedded and an immersed manifold. A detailed treatment of the matter can be found in \cite[Appendix A.2 and Appendix A.4]{Isidori}. 
Let $F: \mathcal{M}_1 \rightarrow \R^N$ be a continuous mapping of topological spaces and let $\mathcal{M}_2=F(\mathcal{M}_1)$. $\mathcal{M}_2$ can be endowed with two topologies: i) the topology of $\mathcal{M}_2$ as a subset of  the Euclidean space  $\R^N$,   so that the open sets in this topology are the sets $O$ of the form  $O=O'\cap \mathcal{M}_2$ for some $O'$ which is open in the Euclidean topology of  $\mathbb{R}^N$; ii) the topology induced by $\mathcal{M}_1$, where the open sets are the sets $U$ of the form $U =F(U')$, for some $U'$ which is open in the topology of $\mathcal{M}_1$. In general, the latter topology is stronger than the former.  With this premise, one can give the following definition. 
\begin{definition}
Let $F:\mathcal{M}_1 \rightarrow \R^N$ be a smooth mapping of manifolds. $F$ is an {\em immersion} if it is injective and $rank(\jacobian{F}{p})=dim(\mathcal{M}_1)$ for every $p \in \mathcal{M}_1$. F is an {\em embedding} if it is an immersion and the topology induced on $\mathcal{M}_2=F(\mathcal{M}_1)$ by the one on $\mathcal{M}_1$ coincides with the Euclidean topology of $\mathcal{M}_2$ as a subset of $\R^N$. 
\end{definition}

The reason why we consider only the case in which the manifolds of the partition are embeddings comes mostly from the need to use the Stroock and Varadhan support theorem: the closure appearing in the statement of such a theorem is intended in Euclidean sense. If the manifold topology was not the Euclidean topology we would have to consider two closures, the closure in the Euclidean topology and the closure in the manifold topology. This would make the exposition much more cloudy. Moreover we point out that in all our examples the manifolds at hand are embedded manifolds. It is possible that, under the assumption of this paper that the vector $\voperp$ is smooth and Lipshitz and that the integral curves of $\voperp$ are convergent, one may prove that the orbits $\curlys$ are indeed embedded manifolds. But this is beyond the scope of this paper.

\subsection{Known facts about UFG semigroups}\label{app:UFG}

In this appendix we gather some known facts that we use frequently. 

\begin{enumerate}[label=\textbf{\textup{[F.\arabic*]}}]
\item A semigroup $\cP_t$  of bounded operators  is Markov if 
$$
\cP_t 1= 1 \qquad \mbox{and} \qquad \cP_t f \geq 0 \mbox{ when } f\geq 0\,,
$$ 
where, in the above, $1$ denotes the function identically equal to one. Denoting by $\| \cdot \|_{\infty}$ the supremum norm, the above implies  that if $\|f\|_{\infty}< \infty$ then $\|\cP_t f\|_{\infty} \leq \|f\|_{\infty}$, i.e. the semigroup is a contraction in the supremum norm. Similarly the two parameter semigroups $\{\cQ_{s,t}\}_{0\leq s\leq t}$ and $\{\zQ_{s,t}\}_{0\leq s\leq t}$, considered in Section \ref{sec:non-autonomous} and Section \ref{sec:longtimebehaviour},  are both contractive in the supremum  norm.   \label{item:contraction}

\item Note that if the vector fields $V_0,V_1\ldots,V_d$ satisfy the parabolic H\"ormander condition then for any $f\in C_b(\R^N)$, the semigroup $\cP_tf(x)$  is smooth in all directions in $\R^N$ and moreover it is smooth in $t$. This is not generally the case if we assume the UFG condition.  However we have that for any $f\in C_b(\R^N)$ and $t>0$ the function $x\mapsto \cP_tf(x)$ is differentiable in the directions $V_{[\alpha]}$ for any $\alpha\in \A$. Moreover for any compact set $K$, $t>0$ there exists $C(K)>0, \omega>0$ such that
\begin{equation*}
\sup_{x\in K}\lv V_{[\alpha]} \cP_tf(x) \rv \leq C(K) e^{\omega t} t^{-\|\alpha\|/2} \lv f \rv_\infty. 
\end{equation*}
If the vector fields $V_{[\alpha]}$ are bounded then the above estimate holds uniformly on $\R^N$, for details see \cite[Chapter 3]{Nee}. In contrast to the case in which  the parabolic H\"ormander condition is enforced, when the UFG condition holds $\cP_tf$ need not be differentiable in the direction $V_0$;  however it is differentiable in the direction $\partial_t-V_0$. For more details see \cite[Appendix A]{CrisanOttobre}. \label{item:shorttime}

\item For $f\in C_V^\infty(\R^N)$ (the set $C_V^\infty(\R^N)$ has been defined in Section \ref{sec:notation}) we have that $(x,t)\mapsto \cP_tf$ is smooth in both $x$ and $t$, i.e. it is differentiable arbitrarily many times in every direction, see \cite{CrisanDelarue}. When $f\in C_b(\R^N)$ we may take a sequence $f_n\in C_V^\infty(\R^N)$ such that $\cP_tf_n\in C_V^\infty(\R^N)$ and for each compact set $K\subseteq \R^N$ we have that $\cP_tf_n$ and $V_{[\alpha_1]}\ldots V_{[\alpha_k]}\cP_tf_n$ converge uniformly over $K$ as $n$ tends to $\infty$ to $\cP_tf$ and $V_{[\alpha]}\ldots V_{[\alpha_k]}\cP_tf$ respectively for each $k\in \mathbb{N}, \alpha_1,\ldots,\alpha_k\in \A$. We shall denote by $\DVinfty$ the space of all  functions that can be approximated with the procedure just described. From what we have just said, the semigroup $\cP_tf$ belongs to $\DVinfty$ for any $f\in C_b$.  See \cite[Appendix A]{CrisanOttobre} for more details. %{\color{red}Alternative argument: Indeed, by \cite[Theorem 5.4]{Kunita} %then the map $x\mapsto X_t^{(x)}$ is smooth for all $t\geq 0$ and %$\mathbb{P}$-almost surely. Therefore $\cP_tf(x)$ is also differentiable in %$x$ (Why can we switch integration and differentiation?), and using the %PDE we obtain that $\cP_tf(x)$ is also differentiable in $t$,}
\label{item:smoothmapstosmooth}

% {\color{red} This is incomplete, you are not saying anything about lack of differentiability in the direction $V_0$ or differentiability in $\cV$. It needs to be explained here that a classical solution of (5) will be differentiable in direction $\cV$.} We say a function $f:\R^N\to\R$ is differentiable in the directions $V_{[\alpha]}$ if there exists a sequence $\{f_n\}_{n=0}^\infty \subseteq C_V^\infty(\R^N)$  {\color{red} this is not completely correct and anyway you cannot take it as a definition, this is a property. Moreover in the Appendix of the paper with Dan we don't take functions in $C_b^{\infty}$ as approximating sequences, we take functions in $C_V^{\infty}$}such that for each compact set $K\subseteq \R^N$ we have that $f_n, V_{[\alpha]}f_n$ and $V_{[\alpha]}V_{[\beta]}f_n$ converge uniformly over $K$ as $n$ tends to $\infty$ to $f, V_{[\alpha]}f$ and $V_{[\alpha]}V_{[\beta]}f$ respectively. We shall denote the space of all such functions by $\DVinfty$. See \cite[Appendix A]{CrisanOttobre} for more details. In particular we have for any $f\in C_b(\R^N)$ and $t>0$ the function $x\mapsto \cP_tf(x)$ belongs to $\DVinfty$. Moreover for any compact set $K$, $t>0$ there exists $C(K)>0, \omega>0$ such that
% \begin{equation*}
% \sup_{x\in K}\lv V_{[\alpha]} \cP_tf(x) \rv \leq C(K) e^{\omega t} t^{-\|\alpha\|/2} \lv f \rv_\infty. 
% \end{equation*}
% If the vector fields $V_{[\alpha]}$ are bounded then the above estimate holds uniformly on $\R^N$, for details see \cite[Chapter 3]{Nee}.
\end{enumerate}

\subsection{Miscellaneous technical facts}\label{app:misc}

\begin{lemma}\label{lem:nonufgexamplecont}
	Let $X$ and $Y$ be as in Example \ref{ex:nonufgexample}. Then the vector fields $\{X,Y\}$ do not satisfy the UFG condition, in the sense that whether we take $X=V_0$ and $Y=V_1$ or viceversa, the UFG condition is not satisfied. 
\end{lemma}

\begin{proof}[Proof of Lemma \ref{lem:nonufgexamplecont}] In the definition of UFG condition take $Y=V_0$ and $X=V_1$ (the other case is simple to show) and 
	assume that the UFG condition holds for some $m\in\N$. Denote by $\ad_X$ the map which takes a vector field $Z$ to $[X,Z]$, then note that
	\begin{equation}\label{eq:adXY}
	(\ad_X)^kY = \psi^{(k)}(x)\partial_y.
	\end{equation}
	Here $\psi^{(k)}$ denotes the $k^{th}$ derivative of $\psi$. Now $(\ad_X)^kY$ commutes with the vector field $Y$ and hence the only non-trivial vector fields in $\mathcal{R}_m$ are $X,Y$ and $(\ad_X)^kY$ for any $k\in \N$. 
	By the UFG condition there exist smooth functions $\varphi_{X},\varphi_{Y,k}$ such that
	\begin{equation*}
	(\ad_X)^{m+1}Y = \sum_{k=0}^m\varphi_{Y,k}(\ad_X)^kY + \varphi_{X}X.
	\end{equation*}
	We may write this as follows using \eqref{eq:adXY}
	\begin{equation*}
	\psi^{(m+1)}\partial_y = \sum_{k=0}^m\varphi_{Y,k}\psi^{(k)}\partial_y + \varphi_{X}\partial_x.
	\end{equation*}
	By considering the direction $\partial_x$ we have that $\varphi_X=0$, therefore we have
	\begin{equation*}\label{eq:ODEforpsi}
		\psi^{(m+1)} = \sum_{k=0}^m\varphi_{Y,k}\psi^{(k)}.
	\end{equation*}	
	Also note that since $\psi(x)=0$ for all $x<0$ we have that $\psi^{(k)}(x)=0$ for all $x<0$ and $k\in \N$;  as $\psi$ is smooth this gives that $\psi^{(k)}(0)=0$ for all $k\in \N$. In particular, $\psi$ solves the following initial value problem
	\begin{align*}
	\psi^{(m+1)}(x) &= \sum_{k=0}^m\varphi_{Y,k}(x,y)\psi^{(k)}(x), && \mbox{for all } x\geq 0\\
	\psi^{(k)}(0)&=0, && \mbox{for all } k\in\{0,1,2,\ldots,m\}.
	\end{align*}
	However since $\psi$ is smooth and the functions $\{\varphi_{Y,k}\}_{k\geq 0}$ are smooth,  there is a (at east locally) unique solution to this initial value problem;  the function which is constantly zero clearly satisfies the initial value problem. Therefore we have that $\psi \equiv 0$ (in a neighbourhood of zero),  which gives a contradiction and hence the UFG condition is not satisfied.
\end{proof}
%%%%%%%%%%%%%%%%%%%5
%%%%%%%%%%%%%%%%%%%%%55
%%%%%%%%%%%%%%%%%%%%55

\begin{lemma}\label{lem:submanifoldlastcoordinate}
	Assume that the vector fields $V_0,\ldots, V_d$ satisfy the UFG condition. Let $\curlys$ be a maximal integral submanifold of $\delnn$ and let $x,y\in \curlys$. Assume that $x,y$ lie in the same coordinate neighbourhood $\mathscr{U}_{x_0}$ of a coordinate transformation $\Phi_{x_0}$ constructed in Section \ref{sec:Coordinatechange} . Then $x$ and $y$ lie in the same maximal integral submanifold of $\deln$ if and only if $\Phi_{x_0}^{n+1}(x)=\Phi_{x_0}^{n+1}(y)$.
\end{lemma}

\begin{proof}[Proof of Lemma \ref{lem:submanifoldlastcoordinate}]
	Assume that $x,y$ both lie in the same maximal integral submanifold $S$ of $\deln$. Then there is a time $T>0$ and a path $p:[0,T]\to S$ satisfying the following ODE 
	\begin{equation*}
	\dot{p}(t)=\sum_{\alpha\in\mathcal{A}_m} {V}_{[\alpha]}(p(t))\psi_{\alpha}(t), \quad p(0)=x, \quad p(T)=y
	\end{equation*}
	for some piecewise linear input functions $\psi_\alpha:[0,T]\to\mathbb{R}$.

	Now let $\tilde{p}(t)=\Phi_{x_0}(p(t))$ and let $\tilde{V}$ denote the representation of $V$ in the coordinates defined by $\Phi_{x_0}$, then we have
	\begin{equation*}
	\dot{\tilde{p}}(t) = \sum_{\alpha\in\mathcal{A}_m} \tilde{V}_{[\alpha]}(\tilde{p}(t))\psi_\alpha(t).
	\end{equation*}
	
	Now by the properties in Proposition \ref{changecoordgen} we have that $\tilde{V}_{[\alpha]}^{n+1}=0$ for all $\alpha\in\mathcal{A}$, and hence
	\begin{equation*}
	\Phi_{x_0}^{n+1}(y)=\tilde{p}^{n+1}(T)=\tilde{p}^{n+1}(0)=\Phi_{x_0}^{n+1}(x).
	\end{equation*}

	Now assume that $\Phi_{x_0}^{n+1}(x)=\Phi_{x_0}^{n+1}(y)$. 
	
	Let $\tilde{\gamma}$ be any smooth curve that is contained in $(\Phi_{x_0}(\mathscr{U}_{x_0})) \cap (\mathbb{R}^n\times\{\Phi_{x_0}^{n+1}(x)\})$, and let $\dot{\tilde{\gamma}}(0)=\tilde{v}$. Define $\gamma = \Phi_{x_0}^{-1}(\tilde{\gamma})$ and $v=\dot{\gamma}(0)$.
 Now we have
	\begin{equation*}
	\dot{\gamma}(0) = \jacobian{\Phi_{x_0}^{-1}}{z}(\gamma(0)) \dot{\tilde{\gamma}}(0) = \jacobian{\Phi_{x_0}^{-1}}{z}(\gamma(0))\tilde{v}.
	\end{equation*}
	Since $\tilde{\gamma}$ is contained within $\mathbb{R}^{n}\times\{\Phi_{x_0}^{n+1}(x)\}$ we have that $\tilde{v} \in \mathbb{R}^n\times \{0\}$ and hence $v\in \deln(\gamma(0))$.
	Therefore the tangent space to $\Phi_{x_0}^{-1}(\operatorname{Im}(\Phi_{x_0}) \cap (\mathbb{R}^n\times\{\Phi_{x_0}^{n+1}(x)\}))$ at each point $x'$ in this set is $\deln(x')$. Therefore  $\Phi_{x_0}^{-1}(\operatorname{Im}(\Phi_{x_0}) \cap (\mathbb{R}^n\times\{\Phi_{x_0}^{n+1}(x)\}))\subseteq S_x$, where $S_x$ is the maximal integral submanifold of $\deln$ which passes through $x$.
	In particular, we have that $y \in S_x$ as required.
\end{proof}

\begin{lemma}\label{lem:fundamentalthmofcalc}
Assume the vector fields $V_0 \dd V_d$ satisfy the UFG condition. 	Let $x,y\in\mathbb{R}^N$ be connected by an integral curve of one of the vector fields $V_{[\alpha]}, \alpha \in \A_m$; that is,  $y=e^{TV_{[\alpha]}}(x)$ for some $T>0$ and $\alpha\in \A_m$.
	Then, for all $h\in \DVinfty$,\footnote{We recall that the set $\DVinfty$ has been  introduced in Appendix \ref{app:UFG} \ref{item:smoothmapstosmooth}.} we have 
	\begin{equation*}
	h(y)-h(x)=\int_{0}^T (V_{[\alpha]} h)(\gamma(s))ds.
	\end{equation*}
\end{lemma}

\begin{proof}[Proof of Lemma \ref{lem:fundamentalthmofcalc}]
	 By definition of directional derivative, and because $h$ is differentiable in the directions $V_{[\alpha]}$, $\alpha \in \A_m$,  one has
	\begin{align*}
	\frac{d}{ds}h(\gamma(s)) = V_{[\alpha]}h(\gamma(s)).
	\end{align*}
	Integrating from $0$ to $T$, and using that $\gamma(0)=x$ and $\gamma(T)=y$, the statement follows.  
\end{proof}

\begin{lemma}\label{lem:denistyargforvs}
	With notation of Section \ref{sec:non-autonomous}, suppose Hypothesis \ref{hyp:nonautoassumptions} \ref{item:smoothnessassumption} holds. For any $g\in C_b(\R^{n})$ define the functions $f(z,\zeta)=g(z)$ and $v_s(z,t):=\cP_t f(z,\zeta_{s-t})$. Then $v_s$ is smooth as a map from $\R^n\times (0,\infty)$ to $\R$; moreover it satisfies 
	\begin{equation}\label{eq:vPDE}
	\partial_t v_s(z,t) = U_0v_s(z,t) + \sum_{i=1}^d U_i^2v_s(z,t).
	\end{equation}
\end{lemma}

\begin{proof}[Proof of Lemma \ref{lem:denistyargforvs}]
		Note that $z\in\R^n\mapsto v_s(z,t)$ is smooth (in any direction in $\R^n$) for each fixed $t>0$ since $\cP_{t}f$ is differentiable in all the directions spanned by $V_{[\alpha]}$ for all $\alpha$ (which span $\R^n$). 
 To see that $t\mapsto v_s(z,t)$ is differentiable we first consider the case when $f$ belongs to $C_V^\infty(\R^{n+1})$ then $t\mapsto \cP_tf$ is differentiable (see Appendix \ref{app:UFG} \ref{item:smoothmapstosmooth}) and hence $t\mapsto v_s(z,t)$ is differentiable. Moreover using \eqref{eq:vintermsofP} we may differentiate $v_s$ to find
	\begin{align*}
	\partial_tv_s(z,t)
	%&=V_0\cP_tf(z,\zeta_{s-t})+ \sum_{i=1}^dV_i^2\cP_tf(z,\zeta_{s-t}) -\dot{\zeta}_{t-s}\cdot \nabla\cP_tf(z,\zeta_{s-t})\\
	&=V_0\cP_tf(z,\zeta_{s-t})+ \sum_{i=1}^dV_i^2\cP_tf(z,\zeta_{s-t}) - {W}_0(\zeta_{s-t}) \partial_{\zeta}\cP_tf(z,\zeta_{s-t})\\
	&=V_0\cP_tf(z,\zeta_{s-t})+ \sum_{i=1}^dV_i^2\cP_tf(z,\zeta_{s-t}) -\voperp\cP_tf(z,\zeta_{s-t})\\
	&=(V_0-\voperp)\cP_tf(z,\zeta_{s-t})+ \sum_{i=1}^dV_i^2\cP_tf(z,\zeta_{s-t}).
	\end{align*}
	Now using the equality $\vodel=V_0-\voperp$ (see \eqref{defvoperp}), we have
	\begin{equation*}
	\partial_tv_s(z,t) = \vodel v_s(z,t)+\sum_{i=1}^dV_i^2v_s(z,t).
	\end{equation*}
Note that, as differential operators, $\vodel=U_0$ and $V_i=U_i$ therefore we have that $v_s$ satisfies \eqref{eq:vPDE}.

To extend the proof to the case when $f$ belongs to $C_b(\R^{n+1})$ we apply the argument of \cite[Appendix A]{CrisanOttobre}, so we only sketch this part of the proof. By Appendix \ref{app:UFG} \ref{item:smoothmapstosmooth} if $f\in C_b(\R^{n+1})$ may take a sequence $f_n \in C_V^\infty(\R^{n+1})$ such that $f_n$ converges to $f$ and  $V_{[\alpha_1]}\dots V_{[\alpha_k]}\cP_tf_n$ converges uniformly on compacts of $\R^{n+1}\times (0,\infty)$ to $V_{[\alpha_1]}\dots V_{[\alpha_k]}\cP_tf$ for any $k\geq 1$, and $\alpha_1,\dots,\alpha_k\in \mathcal{A}_m$. By the above argument we have
\begin{equation*}
\partial_t(\cP_tf_n(z,\zeta_{s-t}))=U_0\cP_tf_n(z,\zeta_{s-t})+ \sum_{i=1}^dU_i^2\cP_tf_n(z,\zeta_{s-t}).
\end{equation*}
For any $h>0$  we have
\begin{equation*}
\frac{\cP_{t+h}f_n(z,\zeta_{s-(t+h)})-\cP_tf_n(z,\zeta_{s-t})}{h}=\frac{1}{h}\int_t^{t+h}U_0\cP_rf_n(z,\zeta_{s-r})+ \sum_{i=1}^dU_i^2\cP_rf_n(z,\zeta_{s-r}) dr \, ; 
\end{equation*}
therefore,  letting $n$ tend to $\infty$, we obtain 
\begin{equation*}
\frac{\cP_{t+h}f(z,\zeta_{s-(t+h)})-\cP_tf(z,\zeta_{s-t})}{h}=\frac{1}{h}\int_t^{t+h}U_0\cP_rf(z,\zeta_{s-r})+ \sum_{i=1}^dU_i^2\cP_rf(z,\zeta_{s-r}) dr.
\end{equation*}
Letting now  $h$ tend to $0$ we have that $(\cP_tf)(z,\zeta_{s-t})$ is differentiable with respect to $t$ and moreover
\begin{equation*}
\partial_t(\cP_tf(z,\zeta_{s-t}))=U_0\cP_tf(z,\zeta_{s-t})+ \sum_{i=1}^dU_i^2\cP_tf(z,\zeta_{s-t}).
\end{equation*}
That is, $v_s$ is differentiable in both $z$ and $t$ as a map from $\R^n\times (0,\infty)$ to $\R$ and satisfies \eqref{eq:vPDE}.%{\color{red} again not clear. Can we define say in Section 2 the concept of $R_m$-smooth or similar? so to avoid long-winded turns of phrases. Or also $C_V^{\infty}$ vs $C_V^{\infty,b}$}

\end{proof}

\begin{lemma}\label{lem:continuityofWlim}
	With the notation of Section \ref{sec:longtimebehaviour},  if the map $\Wlim$ is well defined on $\So$ (in the sense that $\So \subseteq \mathrm{Dom}(\Wlim)$)  and  it is continuous when restricted to $\So$,  then $\Wlim$ is also well defined and continuous on $\curlys_{x_0}$.
\end{lemma}

\begin{proof}[Proof of Lemma \ref{lem:continuityofWlim}]
	First note that given any point $x\in \curlys_{x_0}$ we can find some $s\in \R$ and $\lz\in \So$ such that $x=e^{s\voperp}(\lz)$, in which case we have
	\begin{equation}\label{eq:extensionofWlim}
	\Wlim(x) = \lim_{t\to\infty} e^{t\voperp}(x) = \lim_{t\to\infty} e^{(t+s)\voperp}(\lz)  = \Wlim(z).
	\end{equation}
	Now $\Wlim(z)$ is well defined by assumption and hence $\Wlim(x)$ is well-defined. 
	
	To show that $\Wlim$ is continuous on $\curlys_{x_0}$ take $\{x_k\}_k\subseteq \curlys_{x_0}$ and $x\in \curlys_{x_0}$ such that $x_k\to x$ as $k$ tends to $\infty$, then we must show that $\Wlim(x_k)$ converges to $\Wlim(x)$ as $k$ tends to $\infty$. Let $x_k=e^{s_k\voperp}(\lz_k)$ and $x=e^{s\voperp}(\lz)$ for some $s_k,s\in\R$ and $\lz_k,\lz\in\So$. Without loss of generality we may assume that $s=0$, otherwise consider the sequence $y_k:=e^{-s\voperp}(x_k)$.
	
	Recall from Section \ref{sec:Coordinatechange} that we may take a local neighbourhood $U_{x}$ of $x$ and a coordinate transformation $\Phi$. Then for $k$ sufficiently large we have that $x_k\in U_{x}$, and hence $\Phi(x_k)$ converges to $\Phi(x)$. By the uniqueness of integral curves we have that
	\begin{equation*}
	\Phi(e^{s\voperp}(y)) = e^{s\tilde{V_0}^{(\perp)}}(\Phi(y)). 
	\end{equation*}
	Recall that $\tilde{\voperp}$ denotes the representation of $\voperp$ in the coordinates defined by $\Phi$. Therefore 
		\begin{equation*}
	\Phi(x_k)=\Phi(e^{s_k\voperp}(z_k)) = e^{s_k\tilde{V_0}^{(\perp)}}(\Phi(z_k)). 
	\end{equation*}
	Since $\tilde{\voperp}$ only acts on the last coordinate we have that the first $n$ components of $\Phi(e^{s_k\voperp}(z_k))$ are equal to the first $n$ components of $\Phi(z_k)$. In particular, the first $n$ components of $\Phi(z_k)$ converge to the first $n$ components of $\Phi(z)$. Now $z_k$ and $z$ lie on the same integral submanifold of $\deln$ and hence by Lemma \ref{lem:submanifoldlastcoordinate} the last component of $\Phi(z_k)$ is equal to the last coordinate of $\Phi(z)$. Therefore $\Phi(z_k)$ converges to $\Phi(z)$, and since $\Phi$ is a diffeomorphism we have that $z_k$ converges to $z$.
	
	Now since $\Wlim$ is continuous on $\So$ and using \eqref{eq:extensionofWlim} we have
	\begin{equation*}
	\Wlim(x_k) = \Wlim(z_k) \to \Wlim(z)=\Wlim(x).
	\end{equation*}
Therefore $\Wlim$ is continuous on $\curlys_{x_0}$.
	\end{proof}

%%%%%%%%%%%%%%%%%%%%%%%%%
%%%%%%%%%%%%%%%%%%%%%55

\section{Proofs}\label{sec:mainproofs}
This appendix contains all the proofs that we omitted in the main text.

\subsection{Proofs of Section \ref{sec:geomofUFGprocesses} and Section \ref{sec:5}}\label{app:sec45}

\begin{proof}[Proof of Lemma \ref{lemmaW0}]
	First we show that $\mathrm{span}(\fRm)$ is contained in $\deln$. By definition $\deln$ contains $V_1,\dots, V_d$ and is invariant under $V_0,\dots,V_d$, hence by Note \ref{Tom} we have that $V_{[\alpha]}\in \deln$ for all $\alpha \in \A_m$. By linearity we have that $\mathrm{span}(\fRm)\subseteq \deln$.
	We show that $\deln$ is contained in $\mathrm{span}(\fRm)$. It is sufficient to show that $\mathrm{span}(\fRm)$ contains $V_1\dd V_d$ and is invariant under $V_0,V_1\dd V_d$. Since $V_1\dd V_d\in \fRm$ it suffices to show that every vector field in $\mathrm{span}(\fRm)$ is invariant under $V_0,V_1\dd V_d$. Every vector field ${V}$ in $\mathrm{span}(\fRm)$ can be locally expressed in the form
	\begin{equation}\label{eq:vectorfielddecomp}
	V=\sum_{\alpha\in\mathcal{A}_m} \varphi_\alpha V_{[\alpha]}
	\end{equation}
	for some smooth functions $\varphi_\alpha$. Therefore, again by Note \ref{Tom},  it is sufficient to show that $[V,V_j]\in\Delta_{\mathcal{R}_m}$ for $V$ given by \eqref{eq:vectorfielddecomp} and $j\in\{0,1\dd d\}$. Note that
	\begin{equation*}
	[V,V_j] = \sum_{\alpha\in\mathcal{A}_m} [\varphi_\alpha V_{[\alpha]},V_j] = \sum_{\alpha\in\mathcal{A}_m}  \varphi_\alpha [V_{[\alpha]},V_j]-V_j(\varphi_\alpha)V_{[\alpha]}.
	\end{equation*}
	Now $[V_{[\alpha]},V_j]$ and $V_{[\alpha]}$ are in $\mathrm{span}(\fRm)$ and hence $\mathrm{span}(\fRm)$ is invariant under $V_0, V_1\dd V_d$. Therefore $\deln = \mathrm{span}(\fRm)$;   similarly one can show that $\delnn = \mathrm{span}(\fRmo)$.
\end{proof}

\begin{proof}[Proof of Proposition \ref{changecoordgen}]
	Let us start by proving i). Construct $\Phi$ as described before the statement of Proposition \ref{changecoordgen}. After the change of coordinates $\Phi$ the vector $V$ is expressed as
	\begin{equation}\label{eq:representationofvectorfields}
	\tilde{V}(\bfz)= \left[ (\jacobian{\Phi}{x} ) \cdot V(x) \right] \vert_{x = \Phi^{-1}(\bfz)} \,. 
	\end{equation}
	As we have already observed,  the last $N-n$ rows of the Jacobian matrix $\jacobian{\Phi}{}$ are orthogonal to vectors in $\Delta$, see \eqref{ell}. Since $V \in \Delta$, the statement follows. 
	
	To prove ii), we first observe that by i), the vector fields $\{\pa_{z_j}\}_{j=1}^n$ belong to $\Delta$.  Moreover, by  Note \ref{Tom}, we have that $[\tilde{W}, \pa_{z_j}] \in \Delta$, for all $j=1 \dd n$. The field $[\tilde{W}, \pa_{z_j}]$ can be calculated explicitly:
	$$
	[\tilde{W}, \pa_{z_j}] = \left[ \sum_{i=1}^N \tilde{W}^i \pa_{z_i}, \pa_{z_j} \right] =   - \sum_{i=1}^N \frac{\pa \tilde{W}^i}{\pa z_j} \pa_{z_i} \,.
	$$
	Because $[\tilde{W}, \pa_{z_j}] \in \Delta$, one must have 
	$$
	\frac{\pa \tilde{W}^i}{\pa z_j}  = 0 \quad \mbox{for all } j=1 \dd n, i=n+1 \dd N \,.
	$$
	This concludes the proof. 
\end{proof}

\begin{proof}[Proof of Lemma \ref{lem:boundary}]
	Since $S\subseteq \curlys$ we have that $\overline{S}\subseteq \overline{\curlys}$, therefore it is sufficient to show that if $x\in \pa S$ then $x \notin \curlys$. Assume for a contradiction there exists some $x\in \pa S \cap \curlys$. Since $x\in \curlys$ there exists a neighbourhood $U \subseteq \curlys$ which contains $x$ and on which the coordinate transformation $\Phi$ constructed at the beginning of  Section \ref{sec:Coordinatechange} is well defined. Now $x\in \pa S$ implies there exists a sequence $\{x_k\} \subseteq S$ such that $x_k$ converges to $x$. For $k$ sufficiently large $x_k$ belongs to $U$ and,  since the coordinate transformation is smooth,  we have
	\begin{equation}\label{blububble}
	\Phi^{n+1}(x) = \lim_{k\to\infty} \Phi^{n+1}(x_k) \, ,
	\end{equation}
having used the notation \ref{ellell}. 
	However $x_k$ all belong to the same maximal integral submanifold of $\deln$ and hence $\Phi^{n+1}(x_k)$ is constant for $k$ large enough, by Lemma \ref{lem:submanifoldlastcoordinate}. However this implies, for $k$ large enough, that $\Phi^{n+1}(x) = \Phi^{n+1}(x_k)$; so by \eqref{blububble} and  Lemma \ref{lem:submanifoldlastcoordinate} we have that $x$ and $x_k$ lie in the same  maximal integral submanifold of $\deln$. However this gives a contradiction, since $x_k\in S$ and $x\notin S$.
\end{proof}

\begin{proof}[Proof of Lemma \ref{lemmavoonmanifold}]
We will prove that the set of points $K:=\{x \in \curlys_{x_0}: \voperp(x_0)=0\}\subseteq \curlys_{x_0}$ is both open and closed 
in (the topology of) $\curlys_{x_0}$, hence it has to be the whole manifold $\curlys_{x_0}$ -- see \ref{SAtop} and Appendix \ref{app:topology} for clarifications on the manifold topology.  Such a set is clearly closed (in $\R^N$ and hence in the manifold topology) as it is the intersection between $\curlys_{x_0}$ and the preimage of $0$ through a continuous function. To prove that it is also open, we will show that for any $x \in K$  there exists an open neighbourhood of  $x$, $O_x$, which is contained in $K$. 
Let $x \in \R^N$ such that $\voperp(x)=0$ and let $n=n(x)$ be the rank of $\delnn$ at $x$; then  there exist $n$ vectors in $\delnn(x)$ which span $\delnn$ at $x$. Notice that, by construction, such vectors must belong to $\deln(x)$, as by Lemma \ref{lemmaW0} $\deln(x)=\delnn(x)$ if $\voperp(x)=0$. By the smoothness of the vector fields and because $x$ is a regular point for both distributions, there exists a neighbourhood $O_x$ of $x$ such that the same $n$ vectors span $\delnn(y)$ for every $y \in O_x$. Because $\voperp(y)$ is orthogonal to all the vectors in $\delnn(y)(=\deln(y))$, it must be the case that $\voperp(y)=0$ on $O_x$ (otherwise the rank of $\delnn$ would increase, which is impossible as the rank stays constant on the orbits).  Therefore $O_x \subseteq K$ and the proof is concluded. 
\end{proof}

\begin{proof}[Proof of Proposition \ref{prop:dimcanonlydecrease}.] 
	We emphasize that this proof heavily relies on the fact that the integral manifolds of $\delnn$ coincide with the orbits of  $\delnn$, see Proposition \ref{intmanpro}.
	
	\noindent
	$\bullet$ {\em Proof of i)}. Let $\curlys$ be one of the integral manifolds of $\delnn$ and suppose    $x\in \pa \curlys$. To begin with, we show that $\curlys_x \subseteq \bar{\curlys}$. To this end,   let $y$ be any point in  $\curlys_x$. We want to show that $y \in \bar{\curlys}$.  By Proposition \ref{intmanpro} the integral manifold $\curlys_x$ is given by the orbit through $x$ of the vector fields in $\delnn$, and hence $y$ can be written as the end point of a curve which starts from $x$ and is a piecewise integral curve for vector fields in $\delnn$. By considering each piece  of the integral curve separately, if needed, we may assume that $y=e^{TV}(x)$ for some $T>0$ and $V\in \delnn$. Since $x\in\overline{\curlys}$, there is a sequence $\{x_k\}_k$ converging to $x$ and such that $\{x_k\}_k\subseteq \curlys$. Set $y_k:=e^{TV}(x_k)$ and note that $\{y_k\}_k$ belongs to $\curlys$ since $\curlys$ is an orbit of $\delnn$. We have that $y_k$ converges to $y$ since the map $z\mapsto e^{TV}(z)$ is continuous. Therefore $y \in\overline{\curlys}$ which implies that $\curlys_x\subseteq \overline{\curlys}$ as $y$ is an arbitrary point in $\curlys_x$. However $\curlys$ and $\curlys_x$ are both maximal integral submanifolds so they are either disjoint or they coincide;  since $x\in \curlys_x$ and $x\notin\curlys$ they must be disjoint,  hence $\curlys_x \subseteq \pa\curlys$.
	
	\noindent
	$\bullet$ {\em Proof of ii)}. Note that by the Stroock and Varadhan Theorem, Theorem \ref{thm:stroocksuppthm}, we have  $\mathbb{P}_x(X_t\in \overline{\curlys_x})=1$ for any $x\in\mathbb{R}^N$. From the reasoning in the proof of point i), we know that if $x \in \pa \curlys$  then $\curlys_x \subseteq \pa\curlys$, so that $\overline{\curlys_x} \subseteq \pa\curlys$. Therefore for any $x\in\pa\curlys$ we have
	\begin{equation*}
	\mathbb{P}_x(X_t\in \pa\curlys) \geq \mathbb{P}_x(X_t\in \overline{\curlys_x})=1. 
	\end{equation*}
\end{proof}

%%%%%%%%%%%%%%%%%5
\begin{proof}[Proof of Proposition \ref{prop:SDElivesonsubmanifold}]
Here we consider the case $\voperp(x_0) \neq 0$ (which, by Lemma \ref{lemmavoonmanifold}, implies $\voperp(x)\neq 0$ for every $x \in \curlys_{x_0}$).	Consider the control problem \eqref{stoccontrprob} associated with the SDE \eqref{SDE}.  
	If we can show that  any solution $p(t)$ of \eqref{stoccontrprob} has the property that $p(t)\in S_{{e^{t\voperp}(x_0)}}$ for all $t$,  the result then follows by the Stroock and Varadhan Support Theorem.\footnote{Note that by Theorem \ref{thm:stroocksuppthm} the path $\{X_t^{(x_0)}(\omega)\}_{t\in[0,T]}$ is a limit, in $C([0,T], \| \cdot \|_{\infty})$,  of solutions to the control problem \eqref{stoccontrprob}. Because uniform convergence implies pointwise convergence,  for each fixed $t\geq 0$ the point  $X_t$ is a limit  of $\{p(t):p \text{ is a solution to \eqref{stoccontrprob}} \}$.}  
	Let us  now define the set
	$$
	C:=\left\{t\in\R: p(t)\in \capitals_{e^{t\voperp}(x_0)}\right\}.
	$$
	Note that $C$ is non-empty since $0\in C$; if we can show that $C$ is open and closed as a subset of $\R$ then we must have that $C=\R$ which implies the desired result.
	Let us start by showing that $C$ is open in $\R$. To this end,  fix an arbitrary point $t_0\in C$;  without loss of generality we may assume that $t_0=0$ (otherwise we consider the path $q(t):=p(t+t_0)$). We will show that there exists an open neigbourhood of $0$ which is contained in $C$. To show this fact we will make use of the  (local) change of coordinates defined in Section \ref{sec:Coordinatechange}. Let $\Phi_{x_0}:\mathscr{U}_{x_0}\to \tilde{\mathcal{U}}_{x_0}$ be the coordinate transformation defined on a local neighbourhood of $x_0$. Take $\varepsilon>0$ sufficiently small that $p(t)\in \mathscr{U}_{x_0}$ for all $t\leq \varepsilon$. It is sufficient to show that $p(t)\in \St$ for all $t \in (-\varepsilon,\varepsilon)$.  Let $\tilde{p}(t)=\Phi_{x_0}(p(t))$;  consistently with the notation set in Section \ref{sec:notation}, we shall denote the first $n$ components of $\tilde{p}(t)$ by $\pn(t)$,  the $(n+1)^{th}$ component by $\pnn(t)$ and the last $N-(n+1)$ components by $a(t)$.  That is,  $\tilde{p}(t)=(\pn(t),\pnn(t),a(t)) = \Phi_{x_0}(p(t))$;  hence, in particular, 
    \begin{equation}\label{eq:pnnequalsPhin1}
    \pnn(t) = \Phi_{x_0}^{n+1}(p(t)).
    \end{equation}
   By Lemma \ref{lem:submanifoldlastcoordinate} and \eqref{eq:pnnequalsPhin1},  to prove that $p(t)\in S_{e^{t\voperp}(x_0)}$ it is sufficient to show that  the following holds
\be\label{HH}
\pnn(t)=\Phi_{x_0}^{n+1}(e^{t\voperp}(x_0)) \, .  
\ee 
	Differentiating the equation $\tilde{p}(t)= \Phi_{x_0}(p(t))$ with respect to $t$ and using \eqref{stoccontrprob} and \eqref{eq:representationofvectorfields},  we see that $\tilde{p}$ satisfies the equation
	\begin{equation*}
	\frac{d\tilde{p}(t)}{dt} = \tilde{{V}}_0(\tilde{p}(t)) + \sqrt{2}\sum_{i=1}^d \tilde{{V}}_i(\tilde{p}(t)) \psi_i.
	\end{equation*}
	Since ${V}_i\in \deln$ for $i=1,\dots,d$ and $\deln$ is invariant under ${V}_0$ we have, using Proposition \ref{changecoordgen} (ii) and the notation \eqref{newvectors1}-\eqref{newvectors0},
	\begin{align}
	\frac{d\pn(t)}{dt} &= U_0(\pn(t),\pnn(t),a(t)) + \sqrt{2}\sum_{i=1}^d U_i(\pn(t),\pnn(t),a(t)) \psi_i,\nonumber\\
	\frac{d\pnn(t)}{dt} &= W_0(\pnn(t),a(t)),\label{eq:ptildeWintcurve}\\
	\frac{da(t)}{dt} &=0. \nonumber
	\end{align}
(The above is completely analogous to what we have done to obtain \eqref{ashtree1}-\eqref{ashtree3}).
	 From equation \eqref{eq:ptildeWintcurve} we then have
	\begin{equation*}
	\pnn(t) =e^{tW_0}(\pnn(0)) \stackrel{\eqref{eq:pnnequalsPhin1}}{=} e^{tW_0}(\Phi_{x_0}^{n+1}(x_0)).
	\end{equation*}
	In order to prove \eqref{HH} it remains to show that $e^{tW_0}(\Phi_{x_0}^{n+1}(x_0)) = \Phi_{x_0}^{n+1}(e^{t\voperp}(x_0))$. By uniqueness of the integral curves, to prove this equality  we must show
	$$
	\frac{d}{dt}\Phi_{x_0}^{n+1}(e^{t\voperp}(x_0)) = W_0(\Phi_{x_0}^{n+1}(e^{t\voperp}(x_0))).
	$$
	This follows from
	\begin{align*}
	\frac{d}{dt}(\Phi_{x_0}^{n+1}(e^{t\voperp}(x_0))) &= \nabla_x {\Phi_{x_0}^{n+1}}{x}(e^{t\voperp}(x_0))\voperp(e^{t\voperp}(x_0)) \\
	&= W_0(\Phi_{x_0}^{n+1}(e^{t\voperp}(x_0))), 
	\end{align*}
where the last equality is a consequence of \eqref{eq:representationofvectorfields} and of the fact that $\tilde{V_0}^{(\perp)}=(0 \dd 0, W_0, 0 \dd 0)$ (see comments after \eqref{ashtree1}-\eqref{ashtree3}).  This proves \eqref{HH}, so $C$ is open. Now we show that $C$ is closed by showing that $\R\setminus C$ is open.
	
	Assume there exists $t_0$ such that $p(t_0) \notin S_{e^{t_0\voperp}(x_0)}$. Now we may take $\varepsilon$ sufficiently small that $p(t) \in \mathscr{U}_{p(t_0)}$ whenever $\lv t-t_0\rv<\varepsilon$. It is sufficient to show that $p(t)\notin C$ whenever $\lv t-t_0\rv <\varepsilon$. For a contradiction assume that there exists some $\overline{t} \in (t_0-\varepsilon,t_0+\varepsilon)$ with $\overline{t}\in C$ i.e. $p(\overline{t}) \in S_{e^{\overline{t}\voperp}(x_0)}$; then by the same argument as above, we have that $p(t_0) \in S_{e^{(t_0-\overline{t})\voperp}(p(\overline{t_0}))}$. Now by Lemma \ref{thm:intcurvepreservemanifolds} applied to the points $p(\overline{t})$ and $e^{\overline{t}\voperp}(x_0)$  and to the vector field $\voperp$ we have that 
	$$
	S_{e^{(t_0-\overline{t})\voperp}(p(\overline{t}))} = S_{e^{(t_0-\overline{t})\voperp}(e^{\overline{t}\voperp}(x_0))} =S_{e^{t_0\voperp}(x_0)}.
	$$
	Therefore $p(t_0) \in S_{e^{t_0\voperp}}$ and we have a contradiction since $t_0\notin C$ therefore there is no $\overline{t}\in (t_0-\varepsilon,t_0+\varepsilon)\cap C$, hence $(t_0-\varepsilon,t_0+\varepsilon)\subseteq \R\setminus C$. That is, $C$ is closed in $\R$ and we have $C=\R$ as required.
\end{proof}

\begin{proof}[Proof of Proposition \ref{thm:meszerosetsunderinvmeas}.] 
	For every $x \in\R^N$ and $t>0$, let $g_t(x):=\mathbb{P}_x(X_t^{(x)}\notin \curlys_x)$ and notice that 
	$E_t=\{x\in\mathbb{R}^N:\mathbb{P}_x(X_t^x\notin \curlys_x )>0\} = \{x \in \R^N: g_t(x)>0\}.$  Suppose the SDE \eqref{SDE} admit an   invariant measure,  $\mu$.  Because $E= \cup_{t>0}E_t$, if we prove that $\mu(E_t)=0$ for every $t\geq 0$, then  it follows that $\mu(E)=0$ (as $\{E_t\}_{t\geq 0}$ is an increasing sequence of sets).  So we concentrate on proving the first statement. To this end, define $\curlys^{\ell}$ to be the union of all the maximal integral submanifolds of $\delnn$ of dimension $\ell$ and notice that $\cup_{\ell=0}^N \curlys^{\ell}=\R^N$; moreover, for every (arbitrary but fixed) $t>0$, set  $E^{\ell}_t:=\{x\in\curlys^{\ell}:\mathbb{P}_x(X_t^x\notin \curlys^{\ell})>0\}.$  We now proceed in two steps.  
	
	{$\bullet $ {\em Step 1}}: show that
	\begin{equation*}
	E_t=\bigcup_{\ell=0}^N E^\ell_t.
	\end{equation*}
	Note that $E^\ell_t \subseteq E_t$;  indeed, if $x\in E^\ell_t$ then $\curlys_x\subseteq \curlys^\ell$ and 
	$
	\mathbb{P}_x(X_t^x \notin \curlys_x) \geq \mathbb{P}_x(X_t^x \notin \curlys^\ell) >0.
	$
	Therefore $x\in E_t$. It remains to show that if $x\in E_t$ then there exists some $\ell$ such that $x\in E^\ell_t$. Fix $x\in E_t$ and let $\ell$ denote the dimension of $\curlys_x$. By the Stroock and Varadhan Support Theorem (Theorem \ref{thm:stroocksuppthm})  we have that $\mathbb{P}_x(X_t^x\in \overline{\curlys_x})=1$ for every $t \geq 0$, hence 
	\begin{equation*}
	\mathbb{P}_x(X_t^x\in \partial\curlys_x) = \mathbb{P}_x(X_t^x\notin \curlys_x)= g_t(x).
	\end{equation*}
	By Proposition \ref{prop:dimcanonlydecrease} we have that $\partial\curlys_x$ is contained in the set  $\cup_{k< \ell}\curlys^k$. In particular, we have that $\partial\curlys_x$ is disjoint from $\curlys^\ell$ and hence
	\begin{equation*}
	g_t(x) = \mathbb{P}_x(X_t^x\in\partial \curlys_x) \leq  \mathbb{P}_x(X_t^x \notin \mathscr{S}^\ell).
	\end{equation*}
	Since $x\in E_t$ we have that $g_t(x)>0$ and therefore $\mathbb{P}_x(X_t^x \notin \curlys^\ell)>0$, which, by definition, gives that $x\in E^\ell_t$.

	{$\bullet $ {\em Step 2}}: %{\color{red} I have not double checked this step, please do} 
show that $\mu(E^\ell_t)=0$ for all $\ell \in \{0,\ldots, N\}$. 
 To this end, set $g^{\ell}_t(x):= \mathbb{P}_x(X_t^x\notin \curlys^{\ell})$; then the set of $x\in \curlys^\ell$ such that $g^\ell_t(x)>0$ is the set $E^\ell_t$. Therefore it is sufficient to show that $\int_{\curlys^{\ell}} g_t^{\ell}(x) \mu(dx)=0 \,\, \mbox{ for all } \ell\in \{0 \dd N\}$. Assume this is not the case; that is,  assume there exists some $\bar{\ell}$ such that
	\begin{equation}\label{eq:probofhittingboundaryatdimn}
	\int_{\curlys^{\bar{\ell}}} \mathbb{P}_x(X_t^x\notin \curlys^{\bar{\ell}}) \mu(dx)>0.
	\end{equation}
	We will let $\bar{\ell}$ be the maximum index such that \eqref{eq:probofhittingboundaryatdimn} holds.  Since $\mu$ is an invariant measure we have that
	\begin{equation}\label{eq:muofSn}
	\mu(\curlys^{\bar{\ell}}) = \int_{\mathbb{R}^N} \mathbb{P}_x(X_t^x\in \curlys^{\bar{\ell}}) \mu(dx) = \sum_{k=0}^N \int_{\curlys^k} \mathbb{P}_x(X_t^x\in \curlys^{\bar{\ell}}) \mu(dx).
	\end{equation}
	Fix $k\in\{0,\ldots,N\}$ and first consider the case when $k> \bar{\ell}$. Since $\bar{\ell}$ was chosen to be maximal such that \eqref{eq:probofhittingboundaryatdimn} holds we must have
	\begin{equation*}
	\int_{\curlys^k} \mathbb{P}_x(X_t^x\notin \curlys^k) \mu(dx)=0 \quad \mbox{if } k> \bar{\ell}.
	\end{equation*}
	This is equivalent to saying that the $\mu$- measure of the set $E^k_t= \{x \in \curlys^k: \mathbb{P}_x(X_t^x\notin \curlys^{k}) >0\}$ is zero. 
	Since $k \neq \bar{\ell}$,  we have  $\{x \in \curlys^k: \mathbb{P}_x(X_t^x\in \curlys^{\bar{\ell}})>0 \} \subseteq E^k_t$, so  the $\mu$- measure of the set  $\{x \in \curlys^k: \mathbb{P}_x(X_t^x\in \curlys^{\bar{\ell}})>0 \}$  is zero as well. Therefore 
	$$
	\int_{\curlys^k} \mathbb{P}_x(X_t\in \curlys^{\bar{\ell}}) \mu(dx)=0 \quad \mbox{for all } k> \bar{\ell}, 
	$$
	so that
	\be\label{sum1}
	\sum_{k=\bar{\ell} + 1}^{N}\int_{\curlys^k} \mathbb{P}_x(X_t^x\in \curlys^{\ell}) \mu(dx)=0 \,. 
	\ee
	Now consider the case $k<\bar{\ell}$. In this case we have 
	\be\label{sum2}
	\sum_{k=0}^{\bar{\ell}-1}\int_{\curlys^k} \mathbb{P}_x(X_t^x\in \curlys^{\bar{\ell}}) \mu(dx)=0 \,, 
	\ee
	as by Proposition \ref{prop:dimcanonlydecrease} the dimension of the manifold in which  $X_t$ evolves can only either decrease or stay the same along the paths of the SDE. Putting together \eqref{eq:muofSn}, \eqref{sum1} and \eqref{sum2}, one has
	\begin{equation*}
	\mu(\curlys^{\bar{\ell}}) = \int_{\curlys^{\bar{\ell}}} \mathbb{P}_x(X_t^x\in \curlys^{\bar{\ell}}) \mu(dx).
	\end{equation*}
	Writing $\mathbb{P}_x(X_t^x\in \curlys^{\bar{\ell}})$ as $1-\mathbb{P}_x(X_t^x\notin \curlys^{\bar{\ell}})$  we obtain
	\begin{equation*}
	\mu(\curlys^{\bar{\ell}}) = \int_{\curlys^{\bar{\ell}}} \mathbb{P}_x(X_t\in \curlys^{\bar{\ell}}) \mu(dx) = \mu(\curlys^{\bar{\ell}})- \int_{\curlys^{\bar{\ell}}} \mathbb{P}_x(X_t\notin \curlys^{\bar{\ell}}) \mu(dx), 
	\end{equation*}
	which gives
	\begin{equation*}
	\int_{\curlys^{\bar{\ell}}} \mathbb{P}_x(X_t\notin \curlys^{\bar{\ell}}) \mu(dx)=0.
	\end{equation*}
	This contradicts \eqref{eq:probofhittingboundaryatdimn} and hence we must have that the statement holds.
\end{proof}

%%%%%%%%%%%%%%%%%%%%%%%%%%%
%%%%%%%%%%%%%%%%%%%%%%%%%%%%%
%%%%%%%%%%%%%%%%%%%%%%%%%%%%5

\begin{proof}[Proof of Lemma \ref{prop:pointwiselimit}]
	Fix $x,y\in S$ and $f\in C_b(\mathbb{R}^N)$ and  assume first that $x,y$ are such that there exists a path $\gamma:[0,T]\to\mathbb{R}^N$  with $\gamma(0)=x, \gamma(T)=y$  and 
	$
	\dot{\gamma}(t)=V_{[\alpha]}(\gamma(t)),$ for some $\alpha\in\mathcal{A}_m$. Clearly the final time $T$ will depend on $x$ and $y$, $T=T_{x,y}$.
	By Lemma \ref{lem:fundamentalthmofcalc} and by Appendix \ref{app:UFG} \ref{item:smoothmapstosmooth} we have
	\begin{equation*}
	\cP_t f(y)-\cP_t f(x) = \int_0^{T_{x,y}} (V_{[\alpha]}\cP_t f)(\gamma(s)) ds.
	\end{equation*}
	Take a compact set $K$ such that $K\supseteq \gamma([0,T_{x,y}])$, then by \eqref{eq:graddecayest} we have
	\begin{equation*}
	\lvert \cP_t f(y)-\cP_t f(x)\rvert\leq \sup_{x\in K}(V_{[\alpha]}\cP_t f)(x) T_{x,y} \leq c_{t_0,K}^{\frac{1}{2}} e^{-\lambda(t-t_0)/2} T_{x,y}\lVert f \rVert_\infty.
	\end{equation*}
	Letting $t$ tend to $\infty$ we obtain the result. 	
	For any $x,y \in S$ we can take a piecewise integral curve connecting $x$ and $y$, hence applying the above argument to each piece of the curve we obtain \eqref{eq:pointwiselimit}.
\end{proof}

\begin{proof}[Proof of Proposition \ref{invmesonstathyperpl}]
	Assume there exists an invariant measure $\mu$ with $\mu(S)=1$; we must show that $\mu$ is the unique invariant measure such that $\mu(S)=1$. Integrating \eqref{eq:pointwiselimit} with respect to $\mu$ we obtain 
	\be \label{eq:integralofpointwiselim}
	\lim_{t\to\infty} \lv\cP_t f(x)-\int_S\cP_t f(y) \mu(dy)\rv = 0, \text{ for every } x\in \capitals.
	\ee
	(Here exchanging the integral and limit is justified by the dominated convergence theorem and using Appendix \ref{app:misc} \ref{item:contraction}). The invariance of $\mu$ and \eqref{eq:integralofpointwiselim} imply \eqref{eq:convtoequil}. Because \eqref{eq:convtoequil} holds for every $x\in S$, by the uniqueness of the limit we have that $\mu$ must be the only invariant measure supported on $S$.
	
	It remains to show that $\mu$ is ergodic. Suppose there exists $t>0$ and a Borel set $E\subseteq \mathbb{R}^N$ such that $\cP_t\mathbbm{1}_E = \mathbbm{1}_E$ $\mu$-almost everywhere. Then by the semigroup property we have for every $n\in\mathbb{N}$ $\cP_{nt}\mathbbm{1}_E = \mathbbm{1}_E$ $\mu$-almost everywhere. 
	Now squaring and integrating \eqref{eq:convtoequil} with respect to $\mu$ we have that $\cP_{t}f$ converges to $\int_Sfd\mu$ in $L_\mu^2$ for each $f\in C_b(\mathbb{R}^N)$. Then since $C_b(\mathbb{R}^N)$ is dense in $L_\mu^2$ (see \cite[Theorem 3.14]{rudin}), for all $f\in L_\mu^2$ we have
	\begin{equation*}
	\lim_{n\to\infty}\left\lVert\cP_{nt}f - \int_{S} f d\mu\right\rVert_{L_\mu^2} =0.
	\end{equation*}
	By taking a subsequence if necessary we get convergence $\mu$-almost everywhere.
	In particular, taking $f=\mathbbm{1}_E$ we get $\mu(E)= \mathbbm{1}_E(x)$ $\mu$-almost everywhere. Hence $\mu(E)=0$ or $1$, and $\mu$ is ergodic.
\end{proof}

\subsection{Proofs of Section \ref{sec:non-autonomous}} \label{app:proofsofnonautosec}

We include here the proofs of Lemma \ref{lem:convtoevolutionsystem} and Lemma \ref{lem:convofevolsystemtoequilibria}, on which the proof of Theorem \ref{thm:mainglobalthm} hinges. 

\begin{lemma}\label{lem:convtoevolutionsystem}
	Let Hypothesis \ref{hyp:nonautoassumptions} \ref{item:smoothnessassumption} and \ref{item:OAC} hold and assume the semigroup $\{\cQ_{s,t}\}_{0\leq s\leq t}$ admits an evolution system of measures $\{\nu_t\}_{t\geq 0}$; then, for each $g\in C_b(\R^n)$, $z\in \R^n$, and $s\geq 0$ we have
	\begin{equation}\label{eq:semigpconvergestoevolution}
	\lim_{t\to\infty}\left\lvert \cQ_{s,t}^\zeta g(z)-\int_{\R^n} g(y) d\nu_t\right\rvert \to 0.
	\end{equation}
\end{lemma}

\begin{proof}[Proof of Lemma \ref{lem:convtoevolutionsystem}]
	Fix $g\in C_b(\R^n)$ and let $f\in C_b(\R^{n+1})$ be a function that doesn't depend on the last variable and such that $f(z,\eta)=g(z)$ for every $\eta \in \R, z \in \R^n$; note that by \eqref{linksmanysemigr} we have
	\be\label{A?}
	\cQ_{s,t}^\zeta g(z) = \cP_{t-s}f(z, \zeta_s^\zeta).
	\ee
	Now for every fixed $z,y \in \R^n$ we can write
	\begin{equation*}
	\lvert \cQ_{s,t}g(z)-\cQ_{s,t}g(y)\rvert = \left\lvert\cP_{t-s}f(z,\zeta_s)-\cP_{t-s}f(y,\zeta_s)\right\rvert
	\end{equation*}
	By Note \ref{noteonhypglobal} we have that the hyperplane $S:=\{{x}=(z,\zeta)\in\R^{n+1}: \zeta= \zeta_s\}$ is the orbit of the vector fields $V_{[\alpha]}$,  $\alpha\in\mathcal{A}_m$. Since  $(z,\zeta_s)$ and $(y,\zeta_s)$ belong to $S$ we may take a piecewise integral curve connecting them. Without loss of generality we may take an integral curve $\gamma:[0,T]\to\mathbb{R}^{n+1}$ connecting $(z,\zeta_s)$ and $(y,\zeta_s)$, with $\dot{\gamma}_t= V_{[\alpha]}(\gamma_t)$. Clearly the time $T$ will depend on $z$ and $y$, i.e. $T=T_{z,y}$. Let  $K$ be a  compact set such that $\gamma([0,T_{z,y}]) \subseteq K$;  by Lemma \ref{lem:fundamentalthmofcalc} applied to the function $h=\cP_{t-s}f$ , which is in $\DVinfty$ by Appendix \ref{app:UFG} \ref{item:smoothmapstosmooth}, we have
	\begin{equation*}
	\left\lvert\cP_{t-s}f(z,\zeta_s)-\cP_{t-s}f(y,\zeta_s)\right\rvert \leq \int_0^{T_{z,y}} V_{[\alpha]}(\cP_{t-s}f(\gamma_u)) du.
	\end{equation*}
Because we let $t \rar \infty$, we can restrict to the case $t>s$. 
So fix $s_0>0$  such that $t-s>s_0$;   by \eqref{eq:graddecayest} we then have
	\begin{equation*}
	\left\lvert\cP_{t-s}f(z,\zeta_s)-\cP_{t-s}f(y,\zeta_s)\right\rvert \leq c_{s_0,r}e^{-\lambda (t-s-s_0)} \lVert f\rVert_\infty T_{z,y}.
	\end{equation*}
	Letting $t$ tend to $\infty$ and using \eqref{A?} we obtain
	\begin{equation}\label{eq:indepofinitcond}
	\lim_{t\to\infty}\lvert \cQ_{s,t}g(z)-\cQ_{s,t}g(y)\rvert =0.	  		\end{equation}
	The proof can now be concluded as follows: because  $\{\nu_t\}_t$ is an evolution system of measures (see \ref{eq:evolutionsystem}),  we can write
	\begin{align*}
	\left\lvert \cQ_{s,t}g(z)-\int_{\mathbb{R}^n} g(y) \nu_t(dy)\right\rvert &= \left\lvert \cQ_{s,t}g(z)-\int_{\mathbb{R}^n} \cQ_{s,t}g(y) \nu_s(dy)\right\rvert\\
	&\leq \int_{\mathbb{R}^n}\left\lvert \cQ_{s,t}g(z)- \cQ_{s,t}g(y) \right\rvert	\nu_s(dy).
	\end{align*}
	Using \eqref{eq:indepofinitcond} and  the dominated convergence theorem (which is applicable by Appendix \ref{app:misc} \ref{item:contraction}) we may take the limit as $t$ tends to $\infty$ and obtain \eqref{eq:semigpconvergestoevolution}.
\end{proof}

In order to prove Lemma \ref{lem:convofevolsystemtoequilibria}, which is the core of the last step of the proof of Theorem \ref{thm:mainglobalthm}, we must first prove the following  two results, Lemma \ref{lem:semigpconvergetosemigp} and Lemma \ref{lem:invarianceofmuk}.

\begin{lemma}\label{lem:semigpconvergetosemigp}
	Assume Hypothesis \ref{hyp:nonautoassumptions} holds. Then, for each $g\in C_b(\mathbb{R}^{n})$ and $z\in\R^n$ we have
	\begin{equation*}
	\cQ_{s-t,s}g(z) \to \barQ_{t}g(z)
	\end{equation*}
	as $s$ tends to $\infty$ uniformly on compacts of $\R^n\times(0,\infty)$. That is, for every fixed $T>0$ and $r>0$ we have
	\begin{equation*}\label{eq:uniformoncompacts}
		\lim_{s\to\infty}\sup_{z\in B_r}\sup_{[1/T,T]}\lv\cQ_{s-t,s}g(z) - \barQ_{t}g(z)\rv =0.
	\end{equation*}	
\end{lemma}

\begin{proof}[Proof of Lemma \ref{lem:semigpconvergetosemigp}]
	Fix $g\in C_b(\R^n)$ and consider
	\begin{equation*}
	v_s(z,t)=(\cQ_{s-t,s}^{\zeta}g)(z), \quad  z\in\mathbb{R}^{n},t>0.
	\end{equation*}	
	Like in the Proof of Lemma \ref{lem:convtoevolutionsystem}, define $f\in C_b(\R^{n+1})$ by $f(z,\eta)=g(z)$ for all $z\in \R^n$ and $\eta\in \R$; then by \eqref{A?} we have
	\begin{equation}\label{eq:vintermsofP}
	v_s(z,t)= \cP_{t}f(z,\zeta_{s-t}).
	\end{equation}
	Since $(\cP_{t}f)(x)$ is a continuous function,  we may take the limit as $s$ tends to infinity to obtain
	\begin{equation*}
	\lim_{s\to\infty}v_s=\cP_{t}f(z,\bar{\zeta}) = \barQ_t g(z) =:v(z,t).
	\end{equation*}
	We now wish to show that the above limit is uniform on compact subsets of $\R^n\times (0,\infty)$;  that is, we wish to show that  
	\begin{equation*}
	\lim_{s\to\infty} \, \sup_{t\in[\frac{1}{T},T]} \sup_{z\in B_R} \lv v_s(z,t)-v(z,t)\rv =0, \quad \mbox{for every fixed } R>0, T>0. 
	\end{equation*}
	To show this fact we shall use the Ascoli-Arzela Theorem. Indeed, assuming for the moment that we can apply such a theorem,  then we can  find a subsequence $s_k$ such that $v_{s_k}$ converges uniformly on $B_R\times[1/T,T]$. Since $v_{s_k}$ converges pointwise to $v$ we have that the limit is independent of the choice of sequence hence $v_s$ converges uniformly in $B_R\times [1/T,T]$ to $v$,  as $s$ tends to $\infty$. So,  if we show that the derivatives of $v_s$ are bounded on $B_R\times[\frac{1}{T},T]$, uniformly in $s$, then we may apply Arzela-Ascoli Theorem and the proof is concluded by the above line of reasoning.  By Lemma \ref{lem:denistyargforvs} the function $v_s$ is smooth in $(z,t)\in\R^n\times\R$ and satisfies \eqref{eq:vPDE}.

For any  point $(z,t)\in \mathbb{R}^n\times(0,\infty)$  there exist an open neighbourhood of $(z,t)$ and smooth functions $\varphi_{i,\alpha}$ such that, for any $i \in \{1 \dd n\}$,  the derivative  $\partial_i\equiv \pa_{z^i}$ can be expressed as
	\begin{equation*}
	\partial_i=\sum_{\alpha\in\mathcal{A}_m} \varphi_{i,\alpha} V_{[\alpha]}.
	\end{equation*}
	Therefore, to show that the  derivatives $\pa_iv_s$ are bounded on $B_R\times [0,T]$ it is sufficient to show that $V_{[\alpha]}v_s$ is bounded in $B_R\times [1/T,T]$. This follows from the estimates recalled in  Appendix \ref{app:UFG} \ref{item:shorttime}.  In particular, there exists constants $C(R),\omega(R)>0$ such that
	\begin{align*}
	\sup_{z\in B_R, t\in [1/T,T]}\lv V_{[\alpha]}v_s(z,t)\rv &= \sup_{z\in B_R, t\in [1/T,T]}\lvert V_{[\alpha]}\cP_{t} f(z,\zeta_{s-t})\rvert \\
	&\leq \sup_{t\in[1/T,T]}\, \sup_{x\in B_R\times \{\zeta_t: t\geq -T\}}\lvert V_{[\alpha]}\cP_{t} f(x)\rvert\\
	&\leq C(R) \lvert T\rvert^{\frac{\lVert\alpha\rVert}{2}} e^{\omega(R) T} \lVert f\rVert_{\infty}.
	\end{align*}
	Here we have used that $\zeta_t$ is convergent and hence $B_R \times \{\zeta_t: t\geq -T\}$ is a compact subset of $\R^{n+1}$. Similarly we may bound the second order derivatives $V_i^2v_s$, and using \eqref{eq:vPDE} we obtain a bound for the derivative with respect to $t$,  which is independent of $s$. %Now we may apply Arzela-Ascoli to obtain a subsequence $s_k$ such that $v_{s_k}$ converges uniformly on $B_R \times[1/T,T]$. Since $v_{s_k}$ converges pointwise to $v$ we have that the limit is unique, hence $v_s$ converges uniformly in $B_R\times [1/T,T]$ to $v$ as $s$ tends to $\infty$. 
\end{proof}

Using the tightness of the family $\{\nu_t\}_{t\geq 0}$,  there exists a divergent sequence $t_\ell$ such that $\nu_{t_\ell-k}$ converges weakly to some measure $\mu_k$ as $\ell$ tends to $\infty$, for each $k\in \N$. (We emphasise that, by a diagonal argument, the sequence $t_\ell$ can be chosen to be  independent of $k$). Moreover, $\{\mu_k\}_{k\in\N}$ is tight since $\{\nu_t\}_{t\geq 0}$ is tight (see \cite[Step 2 in the proof of Theorem 6.2]{AngiuliLorenziLunardi}).

\begin{lemma}\label{lem:invarianceofmuk}
	Assume Hypothesis \ref{hyp:nonautoassumptions} holds and construct $\{\mu_k\}_{k \in \N}$ as above. Then,  
	\begin{equation*}
	\int_{\mathbb{R}^{n}} \barQ_{k}g(z) \mu_k(dz) = \int g(z) \mu_0(dz), 
	\end{equation*}
for any $g\in C_b(\R^n)$ and every $k \in \N$. 
\end{lemma}

\begin{proof}[Proof of Lemma \ref{lem:invarianceofmuk}] We will consider the integral $\int_{\mathbb{R}^{n}}\cQ_{t_\ell-k,t_\ell}g(z) \nu_{t_\ell-k}(dz) $ and show the following:
	\be\label{fact1}
	\int_{\R^n} g(z) \mu_{0}(dz)=\lim_{\ell\to\infty}\int_{\mathbb{R}^{n}}\cQ_{t_\ell-k,t_\ell}g(z) \nu_{t_\ell-k}(dz) = \int_{\mathbb{R}^{n}} \barQ_{k}g(z) \mu_k(dz)\,, \quad \mbox{for every $k \in \N$}.
	\end{equation}	
	Let us start with showing the first equality in \eqref{fact1}. 
	Because  $\{\nu_t\}_{t\geq 0}$ is an evolution system of measures (and taking $\ell$ sufficiently large that $t_\ell>k$), we have
	$$
	\int_{\mathbb{R}^{n}}\cQ_{t_\ell-k,t_\ell}g(z) \nu_{t_\ell-k}(dz) = \int g(z) \nu_{t_\ell}(dz). 
	$$
	The above, combined with the fact that $\nu_{t_{\ell}}$ converges weakly to $\mu_0$, gives the first identity in \eqref{fact1}. 	To prove the second equality in \eqref{fact1}, observe the following:
	\begin{align*}
	\int_{\mathbb{R}^{n}}\cQ_{t_\ell-k,t_\ell}g(z) \nu_{t_\ell-k}(dz) - \int_{\mathbb{R}^{n}} \barQ_{k}g(z)\mu_k(dz) =&\int_{\mathbb{R}^{n}}\left(\cQ_{t_\ell-k,t_\ell}g(z)-\barQ_{k}g(z) \right) \nu_{t_\ell-k}(dz) \\
	+&\int_{\mathbb{R}^{n}} \barQ_{k}g(z)\nu_{t_\ell-k}(dz)- \int_{\mathbb{R}^{n}} \barQ_{k}g(z)\mu_k(dz)\\
	=&I_{1,\ell}+I_{2,\ell},
	\end{align*}
	having set
	\begin{align*}
	I_{1,\ell}&:=\int_{\mathbb{R}^{n}}\left(\cQ_{t_\ell-k,t_\ell}g(z)-\barQ_{k}g(z) \right) \nu_{t_\ell-k}(dz),\\
	I_{2,\ell}&:=\int_{\mathbb{R}^{n}} \barQ_{k}g(z)\nu_{t_\ell-k}(dz)- \int_{\mathbb{R}^{n}} \barQ_{k}g(z)\mu_k(dz).
	\end{align*}	
	Now $I_{2,\ell}$ converges to $0$ as $\ell\to \infty$ since $\nu_{t_\ell-k}$ converges weakly to $\mu_k$, by definition of $\mu_k$. 
	To see that $I_{1,\ell}$ vanishes when $\ell$ tends to $\infty$ fix $\varepsilon>0$ and take a ball $B_r$ such that $\nu_{t_\ell-k}(B_r)\geq 1-\varepsilon$ for all $\ell$ with $t_\ell>k$.  This is possible since the family $\{\nu_{t}: t\geq 0\}$ is tight. By Lemma \ref{lem:semigpconvergetosemigp} we know that $Q_{t_\ell-k,t_\ell}g(z)$ converges uniformly on compacts to $\barQ_{k}g(z)$;  hence, if  $\ell$ is sufficiently large,  we have
	\begin{equation*}
	\sup_{z\in B_r}\lvert \cQ_{t_\ell-k,t_\ell}g(z)-\barQ_{k}g(z)\rvert \leq \varepsilon.
	\end{equation*}
We can therefore derive the following estimate in $I_{1,\ell}$: 
	\begin{align*}
	I_{1,\ell}=\int_{\mathbb{R}^{n}}\left(\cQ_{t_\ell-k,t_\ell}g(z)-\barQ_{k}g(z)\right) \nu_{t_\ell-k}(dz) =& \int_{B_r}\left(\cQ_{t_\ell-k,t_\ell}g(z)-\barQ_{k}g(z)\right) \nu_{t_\ell-k}(dz)\\
	+& \int_{\mathbb{R}^{n} \setminus B_r}\left(\cQ_{t_\ell-k,t_\ell}g(z)-\barQ_{k}g(z)\right) \nu_{t_\ell-k}(dz) \\
	\leq& \varepsilon + 2\lVert g\rVert_{\infty} \varepsilon.
	\end{align*}
	As $\varepsilon$ is arbitrary we have that $I_{1,\ell}$ converges to $0$ as $\ell$ tends to $\infty$, and the claim follows.
\end{proof}

\begin{lemma}\label{lem:convofevolsystemtoequilibria}
	Assume Hypothesis \ref{hyp:nonautoassumptions} holds and,  as described before the statement of Lemma \ref{lem:invarianceofmuk}, let $\mu_0$ be the weak limit of the sequence $\nu_{t_{\ell}}$. Then  $\mu_0=\bar{\mu}$.
\end{lemma}
%%%%%%%%%%%%%%%%%%%%%%
%%%%%%%%%%%%%%%%%%%5

\begin{proof}[Proof of Lemma \ref{lem:convofevolsystemtoequilibria}]
	Take $g\in C_b(\mathbb{R}^{n})$. By Lemma \ref{lem:ergbarZ} we know that $\barQ_{k}g(z) \to \mu(g)$ as $k$ tends to $\infty$ for each $z\in\mathbb{R}^{n}$ and $g\in C_b(\mathbb{R}^{n})$. %However using Arzela-Ascoli we can obtain locally uniform convergence on $\R^n$. Indeed fix a ball $B_r$ then it is sufficient to show that $z \mapsto \barQ_{k}f(z)$ has bounded derivatives uniform in $k$ for $z\in B_r$. However it is shown in \cite{Nee} that $z \mapsto \barQ_{k}f(z)$ is smooth, and moreover for all $t\geq t_0$
	%\begin{equation*}
	%\sup_{x\in B_r}\left\lvert V_{[\alpha]}\barQ_{t}f(x)\right\rvert \leq C_{r,t_0} e^{-\lambda (t-t_0)} \lVert f\rVert_\infty \leq C_{r,t_0} \lVert f\rVert_\infty.
	%\end{equation*} 
	By an argument analogous to the one used in  the proof of Lemma \ref{lem:semigpconvergetosemigp} we have that $\barQ_{k}g(z)$ converges to $\mu(g)$ locally uniformly for $z\in \R^n$.
	
	Now fix $\varepsilon>0$;  since $\{\mu_k\}_{k}$ is a tight sequence,  we may take $B_r\subseteq \R^n$ such that $\mu_k(B_r)\geq 1-\varepsilon$ for all $t\geq 0$. Moreover,  for $k$ sufficiently large we have 
	\begin{equation*}
	\sup_{x\in B_r}\lvert\barQ_{k}g(z)-\bar{\mu}(g)\rvert \leq \varepsilon.
	\end{equation*}
	Then
	\begin{align*}
	\int_{\mathbb{R}^n} \barQ_{k}g(z) \mu_k(dz) - \bar{\mu}(g) &= \int_{B_r} \left[ \barQ_{k}g(z)-\bar{\mu}(g)\right] \mu_k(dz) + \int_{\mathbb{R}^n\setminus B_r} \left[\barQ_{k}g(z)-\bar{\mu}(g)\right] {\mu}_k(dz) \\
	&\leq \varepsilon +2\lVert g\rVert_\infty \varepsilon.
	\end{align*}
	Since $\varepsilon$ is arbitrary we deduce
	\begin{equation*}
	\int_{\mathbb{R}^n} \barQ_{k}g(z) \mu_k(dz) \to \bar{\mu}(g), \quad \mbox{as } k \rar \infty. 
	\end{equation*}
 However by Lemma \ref{lem:invarianceofmuk} we also have
	\begin{equation*}
	\int_{\mathbb{R}^n} g(z) \mu_0(dz) =\int_{\mathbb{R}^n} \barQ_{k}g(z) \mu_k(dz), \quad \text{ for every } k\in\N.
	\end{equation*}	
	Therefore $\bar{\mu}=\mu_0$.
\end{proof}

\subsection{Proofs of Section \ref{sec:longtimebehaviour}}\label{app:proofsoflongtime}

\begin{proof}[Proof of Proposition \ref{lem:VufgimpliesAdVUFG}]
First note to prove \eqref{eq:commutatorsofcV} it is sufficient to show that for any $\alpha \in \A_m$,  
\begin{align}
[\Ad_{t\voperp}V_{[\alpha]}, \partial_t+\cV_{0,t}] &= \Ad_{t\voperp}[V_{[\alpha]},V_0], \label{eq:commutatorofValpahandV0}\\
[\Ad_{t\voperp}V_{[\alpha]}, \cV_{i,t}] &= \Ad_{t\voperp}[V_{[\alpha]},V_i], &&\text{ for any } i\in\{1,\ldots,d\}.\label{eq:commutatorofValpahandVi}
\end{align}
By \cite[Proposition 8.30]{Lee}, for any two vector fields $U,V,W$ we have
\begin{equation}\label{eq:adcommuteswithAd}
[\Ad_{W}U,\Ad_{W}V] = \Ad_{W}[U,V].
\end{equation}
Therefore setting $W=t\voperp$, $U=V_{[\alpha]}$ and $V=V_i$ for any $i\in \{1,\ldots,d\}$ we have that \eqref{eq:commutatorofValpahandVi} holds.
To prove \eqref{eq:commutatorofValpahandV0} note that 
\begin{equation}\label{eq:commutatorwithpartials}
[\Ad_{t\voperp}V_{[\alpha]}, \partial_t] = -\partial_t\Ad_{t\voperp}V_{[\alpha]} = -\Ad_{t\voperp}[\voperp,V_{[\alpha]}]= \Ad_{t\voperp}[V_{[\alpha]},\voperp].
\end{equation}
In the above we have used \cite[Lemma 4.4.2]{Sontag} which states that $\partial_t\Ad_{t\voperp}V_{[\alpha]} = \Ad_{t\voperp}[\voperp,V_{[\alpha]}]$. Using \eqref{eq:adcommuteswithAd} and \eqref{eq:commutatorwithpartials} we have 
\begin{equation*}
[\Ad_{t\voperp}V_{[\alpha]}, \partial_t+\Ad_{t\voperp}\vodel] = \Ad_{t\voperp}[V_{[\alpha]},\voperp] + \Ad_{t\voperp}[V_{[\alpha]},\vodel] = \Ad_{t\voperp}[V_{[\alpha]},V_0].
\end{equation*}
Therefore \eqref{eq:commutatorofValpahandV0} holds. Hence we have that \eqref{eq:commutatorsofcV} holds.

It remains to show that for any $\alpha\in\mathcal{A}$ there exists smooth and bounded functions $\tilde{\varphi}_{\alpha,\beta}\in C_V^\infty(\R^N\times\R)$ such that
\begin{equation*}
\cV_{[\alpha],t}(x) = \sum_{\beta \in \A_m} \tilde{\varphi}_{\alpha,\beta}(x,t) \cV_{[\beta],t} (x).
\end{equation*}
Since the vector fields $\{V_0,V_1,\ldots,V_d\}$ satisfy the UFG condition \eqref{eq:UFG}, applying the operator $\Ad_{t\voperp}$ to \eqref{eq:UFG} we have
\begin{equation*}
\Ad_{t\voperp}V_{[\alpha]}(x) = \sum_{\beta \in \A_m} \Ad_{t\voperp}(\varphi_{\alpha,\beta} V_{[\beta]}) (x) = \sum_{\beta \in \A_m} \varphi_{\alpha,\beta}(e^{t\voperp}(x)) \Ad_{t\voperp}V_{[\beta]}(x).
\end{equation*}
Therefore if we show that the functions $\tilde{\varphi}(x,t):=\varphi_{\alpha,\beta}(e^{t\voperp}(x))$ are smooth, bounded and belong to the sets $C_V^\infty(\R^N\times\R)$ then we have that the vector fields $\{\partial_t+\cV_{[0],t},\cV_{[1],t},\ldots,\cV_{[d],t}\}$ satisfy the UFG condition when viewed as vector fields in both the time variable $t$ and spatial variables $\lz$. Since $\varphi_{\alpha,\beta}$ is smooth and bounded, and $\voperp$ is smooth we have that $\varphi_{\alpha,\beta}\circ e^{t\voperp}$ is smooth and bounded, it remains to show that for any $k\in\mathbb{N}$ and $\gamma_1,\ldots,\gamma_k\in \mathcal{A}$ we have
\begin{equation*}
\sup_{x\in \R^N, t\in \R}\lv(\cV_{[\gamma_1],t})\dots(\cV_{[\gamma_k],t})(\varphi_{\alpha,\beta}\circ e^{t\voperp})\rv < \infty
\end{equation*}
	By \cite[Proposition 8.30]{Lee} for every smooth function $f$ we have for all $\alpha \in \A$, $z\in \So$, $s\in\R$ that
	\begin{equation}\label{eq:relatedvectorfields}
	\cV_{[\alpha],s}(f\circ e^{s\voperp})(\lz) = (V_{[\alpha]}f)(e^{s\voperp}(\lz)).
	\end{equation}
Therefore 
\begin{align*}
\sup_{x\in \R^N, s\in \R}\lv(\cV_{[\gamma_1],t})\dots(\cV_{[\gamma_k,t]})(\varphi_{\alpha,\beta}\circ e^{t\voperp})\rv &= \sup_{x\in \R^N, t\in \R}\lv(V_{[\gamma_1]}\dots V_{[\gamma_k]}\varphi_{\alpha,\beta})(e^{t\voperp}(x))\rv  \\
&= \sup_{y\in \R^N}\lv(V_{[\gamma(1)]}\dots V_{[\gamma(k)]}\varphi_{\alpha,\beta})(y)\rv  < \infty.
\end{align*}
Therefore the UFG condition is satisfied.
\end{proof}

%\textcolor{blue}{I have included 2 proofs for the following Lemma, I prefer the first one, are you happy with it?}

\begin{proof}[Proof of Proposition \ref{lemTie}]
	For $f\in C_V^\infty(\R^N)$ we have that $\cP_tf$ is smooth (in every direction, see \ref{item:smoothmapstosmooth}) so by \eqref{eq:autonomoustonon} we obtain
	\begin{equation}\label{eq:autodifftononautodiff}
	\cV_{[\alpha],s}\zQ_{s,t}f(\lz) = \cV_{[\alpha],s}\left(\cP_{t-s}(f\circ e^{-t\wo}))(e^{s\wo})\right)(\lz) \stackrel{\eqref{eq:relatedvectorfields}}{=} V_{[\alpha]}\left(\cP_{t-s}(f\circ e^{-t\wo})\right)(e^{s\voperp}(\lz)).
	\end{equation}
	By differentiating \eqref{eq:autonomoustonon} with respect to $s$ we have, with a calculation analogous to the one in Lemma \ref{lem:denistyargforvs}, 
	\begin{align*}
	\partial_s\zQ_{s,t}f(\lz) &= -\mathcal{L}\cP_{t-s}(f\circ e^{-t\voperp})(e^{s\wo}(\lz)) + \voperp\cP_{t-s}(f\circ e^{-t\voperp})(e^{s\wo}(\lz))\\
	& \stackrel{\eqref{defvoperp}}{=} -\vodel\cP_{t-s}(f\circ e^{-t\voperp})(e^{s\wo}(\lz)) - \sum_{i=1}^d V_{i}^2\cP_{t-s}(f\circ e^{-t\voperp})(e^{s\wo}(\lz))\\
	&=-\cV_{0,s}\zQ_{s,t}f(\lz) - \sum_{i=1}^d\cV_{i,s}\zQ_{s,t}f(\lz) =-\mathcal{L}_s\zQ_{s,t}f(\lz).
	\end{align*}
	
	Now by a density argument analogous to the one in the proof of Lemma \ref{lem:denistyargforvs} we obtain the result for $f\in C_b(\R^N)$. To prove the result for  $g\in C_b(\Sbar)$ we may apply the Tietze Extension Theorem, see \cite[Chapter 2 Theorem 5.4]{Gamelin}, to extend $g$ to a function $f\in C_b(\R^N)$ such that $f=g$ on $\Sbar$ (this is where we need $g$ to be continuous up to and including the closure of $S_{x_0}$, as functions that are continuous on open sets don't necessarily admit a continuous extension to the whole $\R^N$, i.e. Tietze Extension Theorem would not apply). Since $\cZ_t$ takes values in $\So$ for every $t\geq 0$ we have that $\zQ_{s,t}g(\lz)=\zQ_{s,t}f(\lz)$ for any $\lz\in\Sbar$, hence the claim follows. 
	%See Appendix \ref{AppendixA}. 
\end{proof}

\begin{proof}[Proof of Proposition \ref{prop:limitmeasureonmanifold}]
	By Hypothesis \ref{hyp:gennonautoassumptions} \ref{item:gentightnessofp}, the family of measures $\{p_t^x\}_{t\geq 0}$ is tight and hence, by Prokhorov's Theorem,  there exist a measure $\overline{\mu}^{S}$ and a diverging sequence $\{t_k\}_k$ such that  $p_{t_k}^x$ converges weakly to $\overline{\mu}^{S}$ as $t_k\nearrow\infty$. Note that in general the sequence $t_k$ and the measure $\overline{\mu}^{S}$ may depend on the choice of $x\in \So$; however, by Lemma \ref{prop:pointwiselimit}, $p_{t_k}^x(\cdot)= (\cP_t{\mathbbm{1}_{\{\cdot\}}})(x)$ converges weakly to $\overline{\mu}^{S}$ for any choice of $x\in \capitals$. 
	We now show that such a convergence is also independent of the choice of divergent sequence. Let $s_k$ be a sequence such that $s_k \nearrow \infty$ and fix  $f\in C_b(\R^N)$ and $x\in S$; then
	\begin{align*}
	\cP_{t_k}f(x)-\cP_{s_k}f(x) = \int_{s_k}^{t_k} \partial_t\cP_tf(x) dt =  \int_{s_k}^{t_k} \mathcal{L}\cP_tf(x) dt.
	\end{align*}
	By \eqref{eq:graddecayest} (and \eqref{seconorderdecayeq}) there exists a constant $C=C(t_0,x)>0$ such that for all $t> t_0$ we have
	\begin{align*}
	\lv\vodel\cP_tf(x)\rv\leq C(t_0,x)\lVert f\rVert_\infty e^{-\lambda t},\quad 
	\lv V_i^2\cP_tf(x)\rv\leq C(t_0,x)\lVert f\rVert_\infty e^{-\lambda t}.
	\end{align*}
	Using that $\voperp=0$ we have that there exists a constant $C=C(t_0,x)>0$ such that
	\begin{equation*}
	\lv\mathcal{L}\cP_tf(x)\rv\leq C(t_0,x)\lVert f\rVert_\infty e^{-\lambda t}, \quad \text{ for all } t>t_0. 
	\end{equation*}
	
	Therefore 
	\begin{align*}
	\lv\cP_{t_k}f(x)-\cP_{s_k}f(x)\rv \leq C(x,t_0)\lVert f\rVert_\infty \lv\int_{s_k}^{t_k} e^{-\lambda t} dt \rv= \frac{C(x,t_0)\lVert f\rVert_\infty}{\lambda} \lv e^{-\lambda t_k}-e^{-\lambda s_k}  \rv.
	\end{align*}
	Letting $k$ tend to $\infty$ we have that $\lv\cP_{t_k}f(x)-\cP_{s_k}f(x)\rv$ vanishes in the limit and hence $\cP_{s_k}f(x)$ converges to $\overline{\mu}^{S}(f)$. Therefore $\cP_tf(x)$ converges to $\overline{\mu}^{S}(f)$ as $t$ tends to $\infty$.
	
	To show that $\overline{\mu}^{S}$ is an invariant measure take an arbitrary $s>0$ and $f\in C_b(\R^N)$;  then
	\begin{equation*}
	\overline{\mu}^{S}(\cP_s(f)) = \lim_{t\to\infty}\cP_{t}\cP_sf(x) = \lim_{t\to\infty}\cP_{t+s}f(x)= \lim_{t\to\infty}\cP_{t}f(x) = \overline{\mu}^{S}(f), \quad \mbox{for every } s \geq 0 \,.
	\end{equation*}
Hence $\overline{\mu}^{S}$ is an invariant measure. To show that the convergence is uniform on compact subsets of $S$ we apply Arzela-Ascoli. Indeed fix a compact set $K\subseteq \capitals$ then it is sufficient to show that $\cP_tf(x)$ has bounded derivatives uniformly in $t$ on $K$. However $x\mapsto \cP_tf(x)$ is differentiable in the directions $V_{[\alpha]}$  for all $\alpha\in \mathcal{A}$ which span the tangent space of $\capitals$ and, by the Obtuse Angle Condition, Assumption \ref{item:OACgen}, we have for all $t> t_0$ that \eqref{eq:graddecayest} holds. Hence we have that $\cP_{t}f(x)$ converges to $\overline{\mu}^{S}(f)$ uniformly on compact subsets of $\capitals$. Note that since \eqref{eq:muSbarlimit} holds for all $f\in C_b(\R^N)$, there is at most one measure satisfying \eqref{eq:muSbarlimit}.
\end{proof}

We now move on  to prove Lemma \ref{lem:limofnut} and Lemma \ref{lem:limofpt}, which are the backbone of the proof of  Theorem \ref{thm:mainlocalthm}).  Throughout this section, for any $f\in C_b(\Sinftybar)$, we let 
\begin{equation}\label{eq:fhat}
\hat{f}=f\circ \Wlim.
\end{equation}
In order to prove  Lemma \ref{lem:limofnut} we first state and prove the following two results.
\begin{lemma}\label{lem:semigpconvergetosemigpgen}
	Let Hypothesis \ref{hyp:gennonautoassumptions} \ref{item:UFGgen}, \ref{item:LipschitzVoperp}, \ref{item:genODEconv} and \ref{item:conv} hold. For any fixed $f\in C_b(\Sinftybar)$, define $\hat{f}$ as in \eqref{eq:fhat}. Then for any compact $K\subseteq \So$ and $T>0$ we have
	\begin{equation}\label{eq:semigpconvergetosemigpgen}
	\zQ_{t-s,t}\hat{f}(\lz) \to \cP_{s}\hat{f}(\lim_{\tau\to\infty}e^{\tau\voperp}(\lz))
	\end{equation}
	uniformly for $s\in [1/T,T]$ and $\lz\in K$, whenever $\lim_{\tau\to\infty}e^{\tau\voperp}(\lz)$ exists for all $\lz\in K$.
\end{lemma}

\begin{proof}[Proof of Lemma \ref{lem:semigpconvergetosemigpgen}]
	The proof is analogous to the proof of Lemma \ref{lem:semigpconvergetosemigp}, so we only sketch it and point out the main differences. Note that $\cP_t\hat{f}$ is continuous, using \eqref{eq:autonomoustonon} we have that \eqref{eq:semigpconvergetosemigpgen} holds pointwise. 
	%	By \eqref{eq:autonomoustonon} we have
	%	\begin{equation*}
	%	\zQ_{t-s,t}\hat{f}(\lz) = \cP_{s}(\hat{f}\circ e^{-t\wo})(e^{(t-s)\wo}(\lz)) = \cP_{s}(\hat{f})(e^{(t-s)\wo}(\lz)) .
	%	\end{equation*}
	%	Now take the limit as $t$ tends to $\infty$, and use that $\cP_t\hat{f}$ is continuous,
	%	\begin{equation*}
	%	\lim_{t\to\infty}\zQ_{t-s,t}\hat{f}(\lz) = \cP_{s}(\hat{f})(\lim_{t\to\infty}e^{(t-s)\voperp}(\lz)) .
	%	\end{equation*}
	To obtain convergence uniform on compact subsets of $\So\times (0,\infty)$, we use Arzela-Ascoli and the following estimate.

	%	It remains to show that we have uniform convergence on compact subsets of $\So$. To achieve this we shall apply Arzela-Ascoli theorem. Fix a direction $V_{[\alpha]}$ in $\deln$ and a compact set $K\subseteq \So$, then recall by \eqref{eq:autodifftononautodiff} we obtain the following estimate for some $C_K>0$ and $\omega>0$
	Fix a compact set $K\subseteq \So$ and $T>0$ then using \eqref{eq:autodifftononautodiff} and the short time estimates from \cite[Corollary 3.13]{Nee} (which have been recalled in \ref{item:shorttime}), there exists some constant $C(K,T)$ such that the following holds:
	\begin{align*}
	\sup_{\lz \in K,s\in[1/T,T]}\lvert \cV_{[\alpha],t-s}(\zQ_{t-s,t}\hat{f})(\lz)\rvert & \stackrel{=}\sup_{\lz \in K,s\in[1/T,T]} \lvert V_{[\alpha]}(\cP_{s}(\hat{f}\circ e^{-t\voperp}))(e^{(t-s)\voperp}(\lz))\rvert\\
	&=\sup_{s\in[1/T,T],x \in e^{(t-s)\voperp}(K)} \lvert V_{[\alpha]}(\cP_{s}(\hat{f}\circ e^{-t\voperp}))(x)\rvert\\
	&=\sup_{s\in[1/T,T],x \in K'} \lvert V_{[\alpha]}(\cP_{s}(\hat{f}\circ e^{-t\voperp}))(x)\rvert\\
	&\leq C(K',T) \lVert \hat{f}\rVert_{\infty};
	\end{align*}
	here $K'$ is defined as $K'=\bigcup_{\tau\geq-T}e^{\tau\voperp}(K)$.
The set $K'$ is compact under out assumptions, as for each $\tau$ the diffeomorphism $x \rar e^{\tau\voperp}(x)$ is a continuous function, the curve $e^{\tau\voperp}x$ is convergent and the map $\Wlim$ is assumed continuous. 
\begin{comment}
	Note that if we define $\Psi$ by 
	\begin{align*}
	\Psi:K\times [0,1]&\to \R^n\\
	(\lz,\theta)&\mapsto e^{\left(\frac{\theta}{1-\theta}-T\right)\voperp}(\lz)
	\end{align*}
	then $\Psi$ is continuous\footnote{\textcolor{blue}{may need more arguments} since $e^{t\voperp}(\lz)$ converges for each $\lz\in K$} which implies that $\Psi(K\times[0,1])$ is compact.
\end{comment}
	% However $\Psi(K\times[0,1])$ is $K'$ and hence $K'$ is compact.
	%	Also note that we may obtain a similar estimate for the derivative in the variable $s$ using \eqref{eq:bKeq}:
	%	\begin{equation*}
	%	\partial_s \zQ_{t-s,t}\hat{f} =-\Lt_{t-s}\zQ_{t-s,t}\hat{f} =-\cV_{0,t-s} \zQ_{t-s,t}\hat{f} - \sum_{i=1}^d (\cV_{i,t-s})^2\zQ_{t-s,t}\hat{f}.
	%	\end{equation*}
	%	Now $\vodel\in \deln$ so we may bound $\Ad_{(t-s)\voperp}\vodel \zQ_{t-s,t}f$ as above, and $\Ad_{-(t-s)\voperp}V_i^2\zQ_{t-s,t}f$ can be controlled similarly using the short time estimates from \cite{Nee}; there exists $C_K>0$ and $\omega>0$ such that
	%	\begin{equation*}
	%	\sup_{\lz\in K}\lvert V_i^2\zQ_{t-s,t}\hat{f}(\lz)\rvert \leq C_K s^{-1}e^{\omega s}\lVert f\rVert_\infty.
	%	\end{equation*}
	%	
	%	
	%	Therefore we may apply the Arzela-Ascoli theorem and obtain uniform convergence on $K$.
\end{proof}

Define $\mu_k$ to be the probability measure such that $\nu_{t_\ell-k}$ converges weakly to $\mu_k$, this measure is constructed analogously to the comment above Lemma \ref{lem:invarianceofmuk}.

\begin{lemma}\label{lem:invarianceofmukgen}
	%Assume Hypothesis \ref{hyp:gennonautoassumptions} holds. The measure $\mulim \circ \Wlim$ is an invariant measure of $\{\cP_t\}_{t\geq 0}$.
	Let Hypothesis \ref{hyp:gennonautoassumptions} \ref{item:UFGgen} and \ref{item:genODEconv} hold, and assume the semigroup $\{\zQ_{s,t}\}_{s\leq t}$ admits a tight evolution system of measures $\{\nu_t\}_{0\leq t}$ supported on $\So$. Then
	\begin{equation*}
	\int_{\Sinftybar} \cP_{k}\hat{f}(x) \, (\mu_k\circ (\Wlim)^{-1})(dx) = \int_{\Sinftybar} f(x) \, (\mu_0\circ (\Wlim)^{-1})(dx),
	\end{equation*}
	for any $f\in C_b(\Sinftybar)$ and $\hat{f}$ defined as in \eqref{eq:fhat}.
\end{lemma}

\begin{proof}[Proof of Lemma \ref{lem:invarianceofmukgen}]
	%	Fix $f\in C_b(\R^N)$ then let $\hat{f}=f\circ \Wlim$, it is sufficient to show that
	%	$$
	%	\int_{\R^N} \cP_kf(\Wlim(y)) \mu_k(dy) = \int_{\R^N} f(\Wlim(y)) \mu_0(dy)
	%	$$
	%	Let $\hat{f}=f\circ \Wlim$ and note that $\cP_t\hat{f}(\Wlim(y)) = \cP_tf(\Wlim(y))$. \textcolor{blue}{Could use explanation.}
	This proof is completely analogous to the proof of Lemma \ref{lem:invarianceofmuk}, so we only point out the main differences. It suffices to prove the following two expressions
	\begin{align}
	\int_{\Sbar} \hat{f}(\lz) \mu_0(d\lz) =\lim_{\ell\to\infty}\int_{\Sbar}\zQ_{t_\ell-k,t_\ell}\hat{f}(\lz) \nu_{t_\ell-k}(d\lz) &= \int_{\Sbar} \cP_k\hat{f}(\Wlim(\lz)) \mu_k(d\lz),\label{eq:weaklimandevosys}
	\end{align}
	compare to \eqref{fact1} for comparison. Let us start with the first equality in \eqref{eq:weaklimandevosys}. Since $\{\nu_{t}\}_{t\geq 0}$ is an evolution system of measures we have
	\begin{equation*}
	\int_{\Sbar}\zQ_{t_\ell-k,t_\ell}\hat{f}(\lz) \nu_{t_\ell-k}(d\lz) = \int_{\Sbar}\hat{f}(\lz) \nu_{t_\ell}(d\lz).
	\end{equation*}
	Since $\nu_{t_\ell}$ converges weakly to $\mu_0$ and $\Wlim$ is a continuous map from $\overline{S}_{x_0}$ to $\R^N$, by the continuous mapping theorem we have that $\nu_{t_\ell}\circ(\Wlim)^{-1}$ converges weakly to $\mu_0\circ(\Wlim)^{-1}$ and hence we obtain \eqref{eq:weaklimandevosys}. To prove the second equality in \eqref{eq:weaklimandevosys} like in the proof of Lemma \ref{lem:invarianceofmuk} we write
	\begin{align*}
	\int_{\Sbar}\zQ_{t_\ell-k,t_\ell}\hat{f}(\lz) \nu_{t_\ell-k}(d\lz) - &\int_{\Sbar} \cP_{k}\hat{f}(\Wlim (\lz))\mu_k(d\lz) \\=&\int_{\Sbar}\left(\zQ_{t_\ell-k,t_\ell}\hat{f}(x)-\cP_k\hat{f}(\Wlim (x))\right) \nu_{t_\ell-k}(d\lz) \\
	&+\int_{\Sbar} \cP_t\hat{f}(\Wlim (\lz))\nu_{t_\ell-k}(d\lz)- \int_{\Sbar} \cP_t\hat{f}(\Wlim (\lz))\mu_k(d\lz)\\
	=&I_{1,\ell}+I_{2,\ell},
	\end{align*}	
	having set
	\begin{align*}
	I_{1,\ell}&:=\int_{\Sbar}\left(\zQ_{t_\ell-k,t_\ell}\hat{f}(\lz)-\cP_k\hat{f}(\Wlim (\lz))\right) \nu_{t_\ell-k}(d\lz)\\
	I_{2,\ell}&:=\int_{\Sbar} \cP_k\hat{f}(\Wlim (\lz))\nu_{t_\ell-k}(d\lz)- \int_{\Sbar} \cP_k\hat{f}(\Wlim (\lz))\mu_k(d\lz).
	\end{align*}
	Observe that on the image of $\Wlim$ we have $\voperp=0$, by Hypothesis \ref{hyp:gennonautoassumptions} \ref{item:genODEconv}, and hence $\cP_k\hat{f}(\Wlim(\lz))=\cP_kf(\Wlim(\lz))$ therefore we can rewrite $I_{2,\ell}$ as
	\begin{equation*}
	I_{2,\ell}=\int_{\Sinftybar} \cP_kf(\lz)\left(\nu_{t_\ell-k}\circ (\Wlim)^{-1}\right)(d\lz)- \int_{\Sinftybar} \cP_k\hat{f}(\Wlim (\lz))\left(\mu_k\circ (\Wlim)^{-1}\right)(d\lz).
	\end{equation*}
	Now $I_{2,\ell}$ converges to $0$ as $\ell\to \infty$ since $\nu_{t_\ell-k} \circ (\Wlim)^{-1}$ converges weakly to $\mu_k\circ(\Wlim)^{-1}$. The term $I_{1,\ell}$ can be studied analogously to what we have done in the proof of Lemma \ref{lem:invarianceofmuk}, up to modifications in the same spirit of those made so far, so we omit the details.
	
	%	To see that $I_{1,\ell}$ vanishes in the limit as $\ell\to\infty$ fix $\varepsilon>0$ and take a compact set $K_{\varepsilon}\subseteq \capitals_{x_0}$ such that $\nu_{t_\ell-k}(K_\varepsilon)\geq 1-\varepsilon$ for all $\ell$ with $t_\ell>k$, this is possible since the family $\{\nu_{t}: t\geq 0\}$ is tight. By Lemma \ref{lem:semigpconvergetosemigpgen} we know that $\zQ_{t_\ell-k,t_\ell}\hat{f}(\lz)$ converges uniformly to $\cP_t\hat{f}(\Wlim (\lz))$ on compacts. Therefore for $\ell$ sufficiently large we have
	%	\begin{equation*}
	%	\sup_{x\in K_\varepsilon}\lvert \zQ_{t_\ell-k,t_\ell}\hat{f}(\lz)-\cP_{k}\hat{f}(\Wlim (\lz))\rvert \leq \varepsilon.
	%	\end{equation*}
	%	Hence,
	%	\begin{align*}
	%	\int_{\Sbar}\left(\zQ_{t_\ell-k,t_\ell}\hat{f}(\lz)-\cP_{k}\hat{f}(\Wlim (\lz))\right) \nu_{t_\ell-k}(d\lz) =& \int_{K_\varepsilon}\left(\zQ_{t_\ell-k,t_\ell}\hat{f}(\lz)-\cP_{k}\hat{f}(\Wlim (\lz))\right) \nu_{t_\ell-k}(d\lz)\\
	%	&+ \int_{\Sbar \setminus K_\varepsilon}\left(\zQ_{t_\ell-k,t_\ell}\hat{f}(\lz)-\cP_k\hat{f}(\Wlim (\lz))\right) \nu_{t_\ell-k}(d\lz) \\
	%	\leq& \varepsilon + 2\lVert f\rVert_{\infty} \varepsilon.
	%	\end{align*}
	%	As $\varepsilon$ is arbitrary we have that $I_{1,\ell}$ converges to $0$ as $\ell$ tends to $\infty$, and the claim follows.
\end{proof}

\begin{lemma}\label{lem:limofnut}
Suppose Hypothesis \ref{hyp:gennonautoassumptions} holds. Let $\mulim$ be defined as in the comment above Theorem \ref{thm:mainlocalthm}. Then $\mu_0\circ (\Wlim)^{-1} = \mulim$.
\end{lemma}

\begin{proof}[Proof of Lemma \ref{lem:limofnut}] This proof is completely analogous to the proof of Lemma \ref{lem:convofevolsystemtoequilibria} so we omit the details.
\end{proof}

\begin{lemma}\label{lem:limofpt}
	Assume Hypothesis \ref{hyp:gennonautoassumptions} holds, let $x_0$ be an arbitrary point in $\mathcal{I}_0(\overline{x})$ and let $\{\nu_t\},\mu_0$ be constructed as in the proof of Theorem \ref{thm:mainlocalthm}. Let $\{t_\ell\}$ be a divergent sequence such that $p_{t_\ell}^{x_0}$ converges weakly to some probability measure $\nu^{x_0}$. Then $\nu^{x_0}\vert_{\Sinftybar}=\mu_0\circ (\Wlim)^{-1}$. 
%, that is for any $f\in C_b(\R^N)$
	%$$
	%\int_{\R^N} f(y) \nu(dy) = \int_{\R^N} f(\Wlim(y)) \mu_0(dy).
	%$$
\end{lemma}

\begin{proof}[Proof of Lemma \ref{lem:limofpt}]
	First we note that $\nu^{x_0}$ and $\mu_0\circ (\Wlim)^{-1}$ are both supported on $\Sinftybar$. Indeed for the measure $\mu_0\circ (\Wlim)^{-1}$ this follows from Lemma \ref{lem:limofnut}. It is sufficient to show that given a function $f\in C_b(\R^N)$ such that $f(x)=f(y)$ whenever $x\in \overline{\capitals}_y$ then $\nu^{x_0}(f)=f(W^\infty(x_0))$. Let $f$ be such a function then by Proposition \ref{prop:SDElivesonsubmanifold} we have
	\begin{equation*}
	\cP_{t_\ell}f(x_0) = \mathbb{E}_{x_0}[f(X_{t_\ell})] = \mathbb{E}_{x_0}[f(e^{t_\ell\voperp}(x_0))] = f(e^{t_\ell\voperp}(x_0)).
	\end{equation*}
	Now letting $\ell$ tend to $\infty$ we have $\nu^{x_0}(f)=f(\Wlim(x_0))$ and hence $\nu^{x_0}$ must be supported on $\Sinftybar$. 
We now show that $\nu^{x_0}$ and $\mu_0\circ(\Wlim)^{-1}$ coincide on $\Sinftybar$. Take a function $f\in C_b(\R^N)$ and let $\hat{f}=f\circ W^\infty$ then it is sufficient to show that $\nu^{x_0}(\hat{f})=\mu_0(\hat{f})$. This follows from
	\begin{align*}
	\nu^{x_0}(\hat{f})&=\lim_{\ell\to\infty} \cP_{t_\ell}\hat{f}(x_0)=\lim_{\ell\to\infty} \zQ_{0,t_\ell}(\hat{f}\circ e^{t_\ell\voperp})(x_0)\\
	&=\lim_{\ell\to\infty} \zQ_{0,t_\ell}\hat{f}(x_0)=\mu_0(\hat{f}).
	\end{align*}
\end{proof}
%%%%%%%%%%%%%%%%%%%%%%5
%%%%%%%%%%%%%%%%%%%%%%%%%%%%%%
%%%%%%%%%%%%%%%%%%%%%%%%%%%%%5

\subsection{Proofs of Section \ref{sec:8}}\label{sec:proofsofMalliavinsec}

\begin{proof}[Proof of Lemma \ref{lem:formofmalliavinmatrix}]
	Recall that $V_i=(U_i,0)$ for $i=1,\dots,d$ and $V_0=(U_0,W_0)$ where $W_0$ is independent of the variable $z$. By \eqref{eq:MalliavinderivativeSDE} we have
	\begin{equation*}
	D_r^jX_t^{n+1} = D_r^j\zeta_t = \int_r^t \partial_{x^{n+1}} W_0(\zeta_s) D_r^j(\zeta_s) ds. 
	\end{equation*}
	The only solution to this differential equation is $D_r^j\zeta_t=0$. Therefore we have $\mathscr{M}_t^{i,n+1}=0$ and $\mathscr{M}_t^{n+1,j}=0$ for any $i,j=1,\ldots,n+1$. Hence $\mathscr{M}_t$ has the form \eqref{eq:formofmalliavinmatrix}. The $(i,j)^{th}$ entry of the matrix $M_t$ is
	\begin{equation*}
	M_t^{i,j}=\mathscr{M}_t^{ij} = \sum_{k=1}^d\int_0^t D_s^k(X_t^i) D_s^k(X_t^j) ds = \sum_{k=1}^d\int_0^t D_s^k(X_t^i) D_s^k(Z_t^j) ds.
	\end{equation*}
	Therefore $M_t$ is the Malliavin matrix corresponding to $Z_t$.
\end{proof}

\begin{proof}[Proof of Proposition \ref{prop:malliavinmatrixinv}]
	Note that $J_t^{n+1,i}= \frac{\partial}{\partial x^i}\zeta_t=0$ for any $i\in\{1,\ldots,n\}$ therefore $J_t$ has the form
	\begin{equation*}
	J_t=\left(\begin{array}{cc}
	\tilde{J}_t & a\\
	0 &b
	\end{array}\right)
	\end{equation*}
	for some random real numbers $a,b$ and a random $n\times n$ invertible matrix $\tilde{J}_t$. This implies that
	\begin{equation*}
	J_t^{-1}=\left(\begin{array}{cc}
	\tilde{J}_t^{-1} & -\tilde{J}_t^{-1}ab^{-1}\\
	0 &b^{-1}
	\end{array}\right)
	\end{equation*}	
	Now by Lemma \ref{lem:formofmalliavinmatrix} we have that
	\begin{align*}
	\mathscr{C}_t &= J_t^{-1}\mathscr{M}_t(J_t^{-1})^T\\
	&=\left(\begin{array}{cc}
	\tilde{J}_t^{-1} & -\tilde{J}_t^{-1}ab^{-1}\\
	0 &b^{-1}
	\end{array}\right)\left(\begin{array}{cc}
	M_t & 0\\
	0 & 0\\
	\end{array}\right)\left(\begin{array}{cc}
	(\tilde{J}_t^{-1})^T & 0\\
	-\tilde{J}_t^{-1}ab^{-1} &b^{-1}
	\end{array}\right)\\
	&=\left(\begin{array}{cc}
	\tilde{J}_t^{-1}M_t(\tilde{J}_t^{-1})^T & 0\\
	0 & 0\\
	\end{array}\right).
	\end{align*}
	Let $C_t=\tilde{J}_t^{-1}M_t(\tilde{J}_t^{-1})^T$ then it remains to show that $C_t$ is invertible.
	
	It is sufficient to show that $\ker(C_t)=\{0\}$ almost surely, that is if there exists $v(\omega)\in \mathbb{R}^n$ such that $v^TC_tv=0$ implies $v=0$ almost surely. Note that
	\begin{equation*}
	0=v^T(\omega)C_t(\omega)v(\omega) = \sum_{k=1}^d \int_0^t \lvert v^T J_t^{-1}V_k(X_s,Y_s) \rvert^2 ds.
	\end{equation*}
	Therefore if $v(\omega)\in \ker(C_t(\omega))$ then $v$ is orthogonal to the space $K_s:=\operatorname{span}\{J_r^{-1}V_k(X_r):0\leq r\leq s,k=1,\dots,d\}$. Hence it is sufficient to show that $K_s = \mathbb{R}^n$.
	
	Note that the family of vector spaces $\{K_s:s\geq 0\}$ is increasing and set $K_{0+}:=\cap_{s>0} K_s$. By the Blumenthal zero-one law, see Theorem 7.17 in \cite{{KaratzasShreve}}, $K_{0+}$ is a deterministic space with probability one. Define the stopping time
	\begin{equation*}
	\tau=\inf\{s>0:\dim K_s>\dim K_{0+}\}.
	\end{equation*}
	Note that $\tau>0$ with probability one. Let $v$ be orthogonal to $K_{0+}$ and non-zero, then we have $v\perp K_s$ if $s<\tau$, that is,
	\begin{equation*}\label{eq:orthogonalitycond02}
	v^TJ_t^{-1}V_k(X_s,Y_s) = 0, \quad k\in\{1,\dots,d\}, s<\tau.
	\end{equation*}
	This follows since $K_{0+}\subseteq K_s$ for all $s>0$ and for $s<\tau$ we have that $\dim(K_{0+})=\dim(K_s)$.
	
	Recall the set $\mathcal{R}_m$ was defined in \eqref{(R)}, we shall denote by $\distlevel{k}(x)$ to be the vector space spanned by the vectors of $\mathcal{R}_k$ evaluated at the point $x$. 
	
	By following the proof of \cite[Theorem~2.3.2]{Nualart} we obtain that $v$ is orthogonal to $\distlevel{k}(x_0)$ and hence obtain that $\distlevel{k}(x_0)\subseteq K_{0+}$ for all $k\in\mathbb{N}$. By setting $k=m$ we have that $\mathbb{R}^n \subseteq K_{0+} \subseteq K_s \subseteq \mathbb{R}^n$. Therefore $\ker(C_t) =\{0\}$ and we have that $C_t$ is invertible. 
\end{proof}

\end{document}